%% file: IRLS_tensor_arXiv.tex
\documentclass[a4paper,sort&compress]{amsart}

\input{IRLS_tensor_arXiv_shared}

\ifpdf
\hypersetup{
  pdftitle={Asymptotic Log-Det Sum-of-Ranks Minimization Via Tensor (Alternating) Iteratively Reweighted Least Squares},
  pdfauthor={Sebastian Kr\"amer}
}
\fi

\begin{document}

\begingroup
\begin{abstract}
    \input{IRLS_tensor_abstract}
\end{abstract}
\maketitle

\smallskip
\input{IRLS_tensor_content}
\section{Conclusions and outlook}
\input{IRLS_tensor_conclusions}
\begingroup
\appendix
\input{IRLS_tensor_appendix}
\endgroup
\section*{Acknowledgments}
\input{IRLS_tensor_acknowledgments}

\bibliographystyle{siam}
\bibliography{IRLS_tensor_bibliography}

\endgroup


\newpage
\setcounter{section}{0}
\setcounter{equation}{0}
\setcounter{figure}{0}
\setcounter{table}{0}
\setcounter{page}{1}
\makeatletter
\renewcommand{\theequation}{SM\arabic{equation}}
\renewcommand{\thefigure}{SM\arabic{figure}}
\renewcommand{\thetable}{SM\arabic{table}}
\renewcommand{\thesection}{SM\arabic{section}}
\renewcommand{\thepage}{SM\arabic{page}}
\renewcommand{\thetheorem}{SM\arabic{theorem}}
\renewcommand{\thesubsection}    {\thesection.\arabic{subsection}}
\renewcommand{\thesubsubsection} {\thesubsection.\arabic{subsubsection}}
\renewcommand{\theparagraph}     {\thesubsubsection.\arabic{paragraph}}
\renewcommand{\thesubparagraph}  {\theparagraph.\arabic{subparagraph}}
\begin{center}
 \textbf{\MakeUppercase{Supplementary Materials:}}
\end{center}

\input{IRLS_tensor_suppl_content}
\end{document}

%% file: IRLS_tensor_arXiv_shared.tex
\usepackage[utf8]{inputenc}

\makeatletter
\def\paragraph{\@startsection{paragraph}{4}%
  \z@\z@{-\fontdimen2\font}%
  \itshape}
\makeatother

\usepackage{xcolor} 
\usepackage{graphicx} 

\input{rwth-colors}

\usepackage[linkcolor=rwth]{hyperref} 
\hypersetup{
    hidelinks, 
    colorlinks=true,
    hypertexnames=false, 
    breaklinks,
    linkcolor=rwth,
    citecolor=rwth,
    urlcolor=rwth
}
\usepackage[all]{hypcap} 
\usepackage{url} 
\numberwithin{equation}{section} 
\usepackage[capitalize]{cleveref} 
\usepackage[sort]{cite} 

\usepackage{amsfonts,amssymb,amsopn,amsmath,amsthm} 
\usepackage{mathtools} 

\usepackage[font=footnotesize,textfont=it]{caption} 

\usepackage{enumitem} 
\usepackage{multirow} 
\usepackage{placeins} 

\usepackage{algorithm} 
\makeatletter
\renewcommand{\ALG@name}{\sc Algorithm} 
\makeatother
\usepackage{algorithmic} 

\usepackage{tikz}
\usepackage{graphicx}
\usepackage{pgfplots}
\usepackage{import}
\usepackage{tikzscale}
\usepackage[outline]{contour}

\pgfplotsset{compat=newest}
\pgfplotsset{plot coordinates/math parser=false}

\usetikzlibrary{%
  arrows, arrows.meta, calc,  fillbetween, fit,  intersections, patterns,  plotmarks,%
  shapes.geometric,  shapes.misc,  shapes.symbols,  shapes.arrows,%
  shapes.callouts,  
  automata,  backgrounds,  chains,  topaths,  trees,  petri,%
  mindmap,  matrix,  
  folding,  fadings,  through,  positioning,  scopes,  decorations.fractals,%
  decorations.shapes,  decorations.text,  decorations.pathmorphing,  decorations.pathreplacing,%
  decorations.footprints,  decorations.markings,  shadows
}

\newif\ifuseprecompiled
\useprecompiledfalse 
\usepgfplotslibrary{external}
\renewcommand{\email}[1]{\protect\href{mailto:#1}{#1}} 

\title[Asymptotic Log-Det Sum-of-Ranks Minimization via Tensor (A)IRLS]{Asymptotic Log-Det Sum-of-Ranks Minimization via Tensor (Alternating) Iteratively Reweighted Least Squares}

\author{Sebastian Kr\"amer}
\thanks{Institut f\"ur Geometrie und Praktische Mathematik,  RWTH Aachen University,  Templergraben 55,
52056 Aachen, Germany
  (\email{kraemer@igpm.rwth-aachen.de}, \url{https://www.igpm.rwth-aachen.de}).}

\theoremstyle{plain}
\newtheorem{theorem}{Theorem}[section]

\newtheorem{definition}[theorem]{Definition}
\newtheorem{proposition}[theorem]{Proposition}
\newtheorem{corollary}[theorem]{Corollary}
\newtheorem{lemma}[theorem]{Lemma}

\theoremstyle{definition}

\newtheorem{example}[theorem]{Example}
\newtheorem{experiment}[theorem]{Experiment}

\theoremstyle{remark} 
\newtheorem{remark}[theorem]{Remark}

\crefname{experiment}{Experiment}{experiments}
\crefname{hypothesis}{Hypothesis}{Hypotheses}

\crefname{equation}{}{} 

\newcommand{\R}{\ensuremath{\mathbb{R}}}
\newcommand{\N}{\ensuremath{\mathbb{N}}}
\DeclareMathOperator{\diag}{diag}
\DeclareMathOperator{\rank}{rank}

\makeatletter
\newenvironment{breakablealgorithm}
  {
   \begin{center}
     \refstepcounter{algorithm}
     \hrule height.8pt depth0pt \kern2pt
     \renewcommand{\caption}[2][\relax]{
       {\raggedright\textbf{\ALG@name~\thealgorithm} ##2\par}%
       \ifx\relax##1\relax 
         \addcontentsline{loa}{algorithm}{\protect\numberline{\thealgorithm}##2}%
       \else 
         \addcontentsline{loa}{algorithm}{\protect\numberline{\thealgorithm}##1}%
       \fi
       \kern2pt\hrule\kern2pt
     }
  }{
     \kern2pt\hrule\relax
   \end{center}
  }
\makeatother

\input{tikz_tensor_node_style}

\newcommand{\IRLS}{{\sc IRLS}}
\newcommand{\AIRLS}{{\sc AIRLS}}
\newcommand{\AandIRLS}{{\sc (A)IRLS}}
\newcommand{\SALSA}{{\sc SALSA}}
\newcommand{\ASRMK}{{ASRM}}

\def\alg{{(\mathrm{alg})}}
\def\rs{{(\mathrm{rs})}}

\newcommand{\COMMENTx}[1]{
  \leavevmode {//~#1}}
\def\sgets{:=}

\newcommand{\tikzminipage}[1]{
   \setbox0=\vbox{\hbox{
    \ifuseprecompiled
    \includegraphics{pdf_files/#1.pdf}
    \else
    \tikzsetnextfilename{#1}
    \subimport{tikz_base_files/#1/}{#1.tex}
    \fi}}  
   \begin{minipage}{\the\wd0}
   \begin{center}
    \ifuseprecompiled
    \includegraphics[width=\the\wd0]{pdf_files/#1.pdf}
    \else
    \tikzsetnextfilename{#1}
    \subimport{tikz_base_files/#1/}{#1.tex}
    \fi
  \end{center}
 \end{minipage}
 }

  \newcommand{\tikzfigure}[3]%
 {%
 \begin{figure}[#1]%
  \begin{center}%
    \ifuseprecompiled%
    \includegraphics[width=0.98\textwidth]{pdf_files/#2.pdf}%
    \else%
    \tikzsetnextfilename{#2}%
    \subimport{tikz_base_files/#2/}{#2.tex}%
  \end{center}%
  \vspace{-0.5cm}%
  \caption{#3}%
\end{figure}\noindent
}

\def\temporarycaptioncommand{}
\def\externaltablecaption#1#2#3{} 

\newcommand{\externaltable}[3]%
{%
\begin{table}[#1]%
\input{tables/#2_caption.tex}
\tiny
\begin{center}%
\input{tables/#2.tex}%
\end{center}%
\caption{\temporarycaptioncommand{} -- #3}
\end{table}\noindent
}

\def\barfigurecaption#1#2{} 
\def\buttonfigurecaption#1#2{} 

\newcommand{\pdffigure}[3]%
 {%
 \begin{figure}[#1]%
  \begin{center}%
    \includegraphics[width=0.98\textwidth]{external_pdf_files/#2.pdf}%
  \end{center}%
  \vspace{-0.5cm}%
  \caption{#3}%
\end{figure}\noindent
}
  

%% file: rwth-colors.tex
\definecolor{rwth}   {RGB}{  0  84 159}
\definecolor{rwth-75}{RGB}{ 64 127 183}
\definecolor{rwth-50}{RGB}{142 186 229}
\definecolor{rwth-25}{RGB}{199 221 242}
\definecolor{rwth-10}{RGB}{232 241 250}

\definecolor{black}   {RGB}{  0   0   0}
\definecolor{black-75}{RGB}{100 101 103}
\definecolor{black-50}{RGB}{156 158 159}
\definecolor{black-25}{RGB}{207 209 210}
\definecolor{black-10}{RGB}{236 237 237}

\definecolor{magenta}   {RGB}{227   0 102}
\definecolor{magenta-75}{RGB}{233  96 136}
\definecolor{magenta-50}{RGB}{241 158 177}
\definecolor{magenta-25}{RGB}{249 210 218}
\definecolor{magenta-10}{RGB}{253 238 240}

\definecolor{yellow}   {RGB}{255 237   0}
\definecolor{yellow-75}{RGB}{255 240  85}
\definecolor{yellow-50}{RGB}{255 245 155}
\definecolor{yellow-25}{RGB}{255 250 209}
\definecolor{yellow-10}{RGB}{255 253 238}

\definecolor{petrol}   {RGB}{  0  97 101}
\definecolor{petrol-75}{RGB}{ 45 127 131}
\definecolor{petrol-50}{RGB}{125 164 167}
\definecolor{petrol-25}{RGB}{191 208 209}
\definecolor{petrol-10}{RGB}{230 236 236}

\definecolor{turkis}   {RGB}{  0 152 161}
\definecolor{turkis-75}{RGB}{  0 177 183}
\definecolor{turkis-50}{RGB}{137 204 207}
\definecolor{turkis-25}{RGB}{202 231 231}
\definecolor{turkis-10}{RGB}{235 246 246}

\definecolor{grun}   {RGB}{ 87 171  39}
\definecolor{grun-75}{RGB}{141 192  96}
\definecolor{grun-50}{RGB}{184 214 152}
\definecolor{grun-25}{RGB}{221 235 206}
\definecolor{grun-10}{RGB}{242 247 236}

\definecolor{maigrun}   {RGB}{189 205   0}
\definecolor{maigrun-75}{RGB}{208 217  92}
\definecolor{maigrun-50}{RGB}{224 230 154}
\definecolor{maigrun-25}{RGB}{240 243 208}
\definecolor{maigrun-10}{RGB}{249 250 237}

\definecolor{orange}   {RGB}{246 168   0}
\definecolor{orange-75}{RGB}{250 190  80}
\definecolor{orange-50}{RGB}{253 212 143}
\definecolor{orange-25}{RGB}{254 234 201}
\definecolor{orange-10}{RGB}{255 247 234}

\definecolor{rot}   {RGB}{204   7  30}
\definecolor{rot-75}{RGB}{216  92  65}
\definecolor{rot-50}{RGB}{230 150 121}
\definecolor{rot-25}{RGB}{243 205 187}
\definecolor{rot-10}{RGB}{250 235 227}

\definecolor{bordeaux}   {RGB}{161  16  53}
\definecolor{bordeaux-75}{RGB}{182  82  86}
\definecolor{bordeaux-50}{RGB}{205 139 135}
\definecolor{bordeaux-25}{RGB}{229 197 192}
\definecolor{bordeaux-10}{RGB}{245 232 229}

\definecolor{violett}   {RGB}{ 97  33  88}
\definecolor{violett-75}{RGB}{131  78 117}
\definecolor{violett-50}{RGB}{168 133 158}
\definecolor{violett-25}{RGB}{210 192 205}
\definecolor{violett-10}{RGB}{237 229 234}

\definecolor{lila}   {RGB}{122 111 172}
\definecolor{lila-75}{RGB}{155 145 193}
\definecolor{lila-50}{RGB}{188 181 215}
\definecolor{lila-25}{RGB}{222 218 235}
\definecolor{lila-10}{RGB}{242 240 247}

%% file: tikz_tensor_node_style.tex
\tikzstyle{alphacontr}=[gray,line width=0.04cm]
\tikzstyle{tensornode}=[fill=turkis-50,draw=none,circle,minimum width = 0.8cm, minimum height = 0.8cm]
\tikzstyle{partialcontraction}=[fill=turkis-25,draw=none,circle,minimum width = 1cm, minimum height = 1cm]
\tikzstyle{tensornodeshadow}=[fill=white,draw=none,circle,minimum width = 0.9cm, minimum height = 0.9cm]
\tikzstyle{widetensornode}=[fill=turkis-50,draw=none,minimum width = 2.4cm, minimum height = 0.8cm, rounded corners = 0.3cm] 
\tikzstyle{verywidetensornode}=[fill=turkis-50,draw=none,minimum width = 4cm, minimum height = 0.8cm, rounded corners = 0.3cm] 
\tikzstyle{widesvnode}=[fill=turkis-50,draw=none,diamond,minimum width = 3cm, minimum height = 1.5cm, rounded corners = 0.5cm] 
\tikzstyle{greytensornode}=[fill=black-50,draw=none,circle,minimum width = 0.8cm, minimum height = 0.8cm]
\tikzstyle{largetensornode}=[fill=turkis-50,draw=none,circle,minimum width = 1.2cm, minimum height = 1.2cm]
\tikzstyle{matrixnode}=[font={\footnotesize},fill=turkis-50,draw=none,regular polygon,regular polygon sides=4,inner sep = 0.1cm,minimum width = 0.6cm, minimum height = 0.6cm]
\tikzstyle{svnode}=[font={\footnotesize},fill=turkis-50,draw=none,diamond,minimum width = 0.6cm, minimum height = 0.6cm]
\tikzstyle{tinynode}=[font={\footnotesize},fill=turkis-50,draw=none,circle,minimum width = 0.3cm, minimum height = 0.3cm]
\tikzstyle{centernode}=[fill=turkis-50,draw,circle,minimum width = 0.8cm, minimum height = 0.8cm]
\tikzstyle{virtnode}=[draw=none]
\tikzstyle{modename}=[midway, above, inner sep=0.05cm, sloped, font={\footnotesize}]
\tikzstyle{modenamebelow}=[midway, below, inner sep=0.05cm, sloped, font={\footnotesize}]
\tikzstyle{network}=[semithick,scale=0.75,transform shape,x=1cm,y=1cm]
\tikzstyle{orthofrom}=[{Latex[length=0.8mm, width=1.45mm]}-,shorten <=-0.1cm]
\tikzstyle{orthoto}=[-{Latex[length=0.8mm, width=1.45mm]},shorten >=-0.1cm]
\tikzstyle{Grouporthofrom}=[{Latex[length=0.8mm, width=1.45mm]}-,shorten <=0.056cm]
\tikzstyle{Grouporthoto}=[-{Latex[length=0.8mm, width=1.45mm]},shorten >=0.056cm]
\tikzstyle{grouporthofrom}=[{Latex[length=0.8mm, width=1.45mm]}-,shorten <=0.168cm]
\tikzstyle{grouporthoto}=[-{Latex[length=0.8mm, width=1.45mm]},shorten >=0.168cm]
\newlength{\nodedist}
\newlength{\shortvirtnodedist}
\newlength{\virtnodedist}
\newlength{\longvirtnodedist}
\newcommand{\convexpath}[2]{
[   
    create hullnodes/.code={
        \global\edef\namelist{#1}
        \foreach [count=\counter] \nodename in \namelist {
            \global\edef\numberofnodes{\counter}
            \node at (\nodename) [draw=none,name=hullnode\counter] {};
        }
        \node at (hullnode\numberofnodes) [name=hullnode0,draw=none] {};
        \pgfmathtruncatemacro\lastnumber{\numberofnodes+1}
        \node at (hullnode1) [name=hullnode\lastnumber,draw=none] {};
    },
    create hullnodes
]
($(hullnode1)!#2!-90:(hullnode0)$)
\foreach [
    evaluate=\currentnode as \previousnode using \currentnode-1,
    evaluate=\currentnode as \nextnode using \currentnode+1
    ] \currentnode in {1,...,\numberofnodes} {
  let
    \p1 = ($(hullnode\currentnode)!#2!-90:(hullnode\previousnode)$),
    \p2 = ($(hullnode\currentnode)!#2!90:(hullnode\nextnode)$),
    \p3 = ($(\p1) - (hullnode\currentnode)$),
    \n1 = {atan2(\y3,\x3)},
    \p4 = ($(\p2) - (hullnode\currentnode)$),
    \n2 = {atan2(\y4,\x4)},
    \n{delta} = {-Mod(\n1-\n2,360)}
  in 
    {-- (\p1) arc[start angle=\n1, delta angle=\n{delta}, radius=#2] -- (\p2)}
}
-- cycle
}

%% file: IRLS_tensor_abstract.tex
Affine sum-of-ranks minimization (ASRM) generalizes the affine rank minimization (ARM) problem from matrices to tensors. Here, the interest lies in the ranks of a family $\mathcal{K}$ of different matricizations. Transferring our priorly discussed results on asymptotic log-det rank minimization, we show that iteratively reweighted least squares with weight strength $p = 0$ remains a, theoretically and practically, particularly viable method denoted as \IRLS{}-$0\mathcal{K}$. As in the matrix case, we prove global convergence of asymptotic minimizers of the log-det sum-of-ranks function to desired solutions. Further, we show local convergence of
\IRLS{}-$0\mathcal{K}$ in dependence of the rate of decline of the therein appearing regularization parameter $\gamma \searrow 0$. For hierarchical families $\mathcal{K}$, we show how an alternating version (\AIRLS{}-$0\mathcal{K}$, related to prior work under the name \SALSA{}) can be evaluated solely through tensor tree network based operations. The method can thereby be applied to high dimensions through the avoidance of exponential computational complexity. Further, the otherwise crucial rank adaption process becomes essentially superfluous even for completion problems. In numerical experiments, we show that the therefor required subspace restrictions and relaxation of the affine constraint cause only a marginal loss of approximation quality. On the other hand, we demonstrate that \IRLS{}-$0\mathcal{K}$ allows to observe the theoretical phase transition also for generic tensor recoverability in practice. Concludingly, we apply \AIRLS{}-$0\mathcal{K}$ to larger scale problems.

%% file: IRLS_tensor_content.tex
\section{Introduction}\label{sec:intro}
The setting of affine sum-of-ranks minimization (ASRM) is a generalization of the
affine rank minimization (ARM) problem for matrices to tensors. While \textit{the} tensor rank refers to the
minimal number of elementary tensors required for a decomposition into a sum, we are here interested in the ranks of so called matricizations. Let $[d] = \{1,\ldots,d\}$, $d \in \N$, as well as $n_\mu \in \N$, $\mu = 1,\ldots,d$.
For $\emptyset \neq J \subsetneq [d]$ and $J^{\mathsf{c}} := [d] \setminus J$, we define such matricizations (cf. \cite{Gr10_Hie})
\begin{align*}
 (\cdot)^{[J]}: \R^{n_1} \otimes \ldots \otimes \R^{n_d} \rightarrow \R^{n_J \times n_{J^{\mathsf{c}}}}, \quad n_S := \prod_{\mu \in S} n_\mu,
\end{align*}
as the simple reshaping isomorphisms induced via
\begin{align*}
 (v_1 \otimes \ldots \otimes v_d)^{[J]} := \mathrm{vec}(\bigotimes_{j \in J} v_j) \cdot \mathrm{vec}(\bigotimes_{j \in J^{\mathsf{c}}} v_j)^T, \quad v_i \in \R^{n_i},\ i = 1,\ldots,d,
 \end{align*}
 where $\mathrm{vec}(\cdot): \R^{\bigtimes_{\mu \in S} n_\mu} \rightarrow \R^{n_S}$ denotes the vectorization in co-lexicographic (column-wise) order. As usual, 
 we identify $\R^{n_1} \otimes \ldots \otimes \R^{n_d} \cong \R^{n_1 \times \ldots \times n_d}$.
For a (not necessarily hierarchical) family of subsets $\mathcal{K} \subseteq \{ J \subsetneq [d] \mid J \neq \emptyset \}$ 
and a surjective linear operator $\mathcal{L}: \R^{n_1 \times \ldots \times n_d} \rightarrow \R^\ell$, $\ell < n_{[d]}$,
as well as measurements $y \in \mathrm{image}(\mathcal{L})$, we then define \ASRMK{} to refer to the problem of finding
\begin{align}\label{trueobj}
  \underset{X \in \R^{n_1 \times \ldots \times n_d}}{\mathrm{argmin}} \sum_{J \in \mathcal{K}} \mathrm{rank}(X^{[J]}) \quad \mbox{subject to } \mathcal{L}(X) = y.
\end{align}
This setting is not only of particular interest due to its regularizing properties, 
but its close relation to so called hierarchical (or tensor tree) decompositions (cf. \cite{Gr10_Hie,Kr20_Tre}).
We are here however mainly interested in the problem itself, and only secondarily in the possibility
to recover an eventual ground truth tensor from its measurements.
As large parts of this work rely on our preceding article \cite{Kr21_Asy}, which in turn is based on \cite{MoFa12_Ite,FoRaWa11_Low,CaWaBo08_Enh,DaDeFoGu10_Ite}, we strongly recommend to take notice of such.
\subsection{Approaches to \ASRMK{} and tensor recovery}
Affine rank minimization (ARM), as theoretical origin of \ASRMK{}, is included as such for the dimension $d = 2$ and $\mathcal{K} = \{ \{1\} \}$ \cite{Kr21_Asy}, and consequently defined as the problem to find a matrix 
\begin{align*}
  X^\ast \in \underset{X \in \R^{n \times m}}{\mathrm{argmin}}\ \mathrm{rank}(X) \quad \mbox{subject to } \mathcal{L}(X) = y.
\end{align*}
This setting in turn is based on the affine cardinality minimization problem (ACM), that is to find a vector
\begin{align*}
 x^\ast \in \underset{x \in \R^n}{\mathrm{argmin}} \ \mathrm{card}(x) \quad \mbox{subject to} \quad \mathcal{L}(x) = y.
\end{align*}
A short overview over recovery methods as well as the role of iteratively reweighted least squares (\IRLS{}, cf. \cite{MoFa12_Ite,FoRaWa11_Low} for ARM and \cite{CaWaBo08_Enh,DaDeFoGu10_Ite} for ACM) for these two problems can be found in our preceding article \cite{Kr21_Asy}.
To the best of our knowledge, \IRLS{} has only priorly been considered with regard to the \ASRMK{} problem for tensors in the thesis \cite{Kr20_Tre}, from which also the related, so called stable ALS approximation algorithm \cite{Kr19_Sta} stems. Relaxations of \ASRMK{} itself however have been considered before, including the minimization of the sum of nuclear norms \cite{GaReYa11_Ten,LiSh13_AnE,SiTrLaSu14_Lea}. The tensor rank as outlined in the introduction, \cref{sec:intro}, however, is hard to calculate, and usually not the direct target of minimization. Though \cite{SoSiLa19_Can,SoLa19_Fib} utilize the canonical polyadic decomposition to a certain fiber completion problem.
Other algorithm rely on the explicit, separate adaption of unknown ranks such as low rank manifold \cite{St16_Rie,KrStVa14_Low,SiHe15_Opt} or a-priorly representation or subspace based optimization \cite{HaNoPe21_Act,GrNoCh19_Lea,RaScSt15_Ten}. 
However, non-intrusive rank adaption schemes, even if elaborate, tend to be problematic \cite{Kr19_Sta}.
The \AIRLS{} related method presented therein, as well as \cite{GoGoRoSwKuEi20_Ten,BaEiSaTr21_Pri,GoScTr21_Abl} based thereon, contrarily consider an intrusive regularization related to reweighting that circumvents the instability and overfitting problems otherwise caused.
Another class of algorithm requires to choose specific sampling points, prominently cross approximation based methods \cite{BaGrKl13_Bla,OsTy10_TTc,KaOsBu20_Ten}. Though such are preferable in that setting, we here however assume the affine measurement operator to be a priorly given.

\subsection{Contributions and organization of this paper} 

The novel aspects of this paper are organized as follows.
\begin{itemize}
 \item In \cref{sec:asymmin,sec:irls}, we generalize the optimization as well as reweighting process from the matrix to the tensor case in an introductory manner. \Cref{sec:datasparseintro}
 contains a preliminary description of hierarchical decompositions and the thereto related data sparse optimization.
 \item In \cref{sec:understr}, we interpret the tensor log-det approach as successive minimization scheme and thereby prove the convergence of global optima to the desired solution, as analogously done for the matrix case \cite{Kr21_Asy}.
 \item \Cref{sec:logdetirls} provides global convergence results for the adjusted tensor \IRLS{}-$0\mathcal{K}$ algorithm with respect to sequences of complementary weights, under consideration of the rate of decline of the regularization parameter $\gamma$.
 \item In \cref{sec:tensor}, we discuss the relaxation of the affine constraint together with the restriction to iteratively defined sequences of admissible subspaces.
 \item \Cref{sec:reloptHT} concisely reintroduces hierarchical formats as non-rooted tree tensor networks with an emphasis on its graph theoretical foundation. It contains several fundamental statements required for the subsequently introduced A(lternating)\IRLS{}-$0\mathcal{K}$ algorithm.
 \item In \cref{AIRLS}, we utilize tree tensor networks to derive the \AIRLS{}-$0\mathcal{K}$ algorithm which allows a non-exponentially scaling realization of the relaxed \IRLS{}-$0\mathcal{K}$ method introduced in \cref{sec:tensor} through an evaluation within given low rank representations. 
 \item \Cref{sec:numexp} contains a comprehensive series of numerical experiments. Firstly, we
 demonstrate that \IRLS{}-$0\mathcal{K}$ allows to observe the theoretical phase transition \cite{BrGeMiVa21_Alg} regarding the required number of measurements for recoveries. Secondly, we follow the relaxations laid out in this work made from \IRLS{}-$0\mathcal{K}$ up to the \AIRLS{}-$0\mathcal{K}$ approach. We demonstrate the improvement, but likewise common ground towards our priorly introduced, so called SALSA algorithm \cite{Kr19_Sta}, as well as superiority over conventional ALS. We conclude with an application of \AIRLS{}-$0\mathcal{K}$ to large scale problems in higher dimensions. 
 \item \Cref{sec:minorproofs} contains a postponed proof.
 The supplementary \cref{altasrm} includes a further numerical experiment.
 \Cref{sec:sens} contains extended visualization of results as explained in
 \cref{sec:visres}. Technical proofs concerning branch evaluations and therefor partially necessary notation can be found in \cref{technicalproof,app:sec:reloptHT}.  
 The \AIRLS{}-$0\mathcal{K}$ method is summarized in \cref{AIRLS-0K},
 whereas \cref{sec:practheurasp} discusses viable heuristics.
\end{itemize}

\subsection{Asymptotic minimization}\label{sec:asymmin}
We have priorly discussed in \cite{Kr21_Asy} as based on \cite{MoFa12_Ite,FoRaWa11_Low,CaWaBo08_Enh,DaDeFoGu10_Ite} in which way the ARM problem for matrices can be
approached via the asymptotic minimization (cf. \cref{Xastdef}) of the family
\begin{align} \label{matrixsur}
 f_\gamma(A) := \log \prod_{i = 1}^{k_1} (\sigma_i^2(A) + \gamma) = \log \det(A A^T + \gamma I), \quad \gamma \searrow 0,
\end{align}
for which $\sigma_i(A)$, $i = 1,\ldots,r$, are defined as the singular values of $A \in \R^{k_1 \times k_2}$
and $\sigma_i(A) = 0$, $i > r$, $r = \mathrm{rank}(A)$. Plainly analogous,
its tensor version for the minimization of a sum of ranks is defined as (see \cref{sec:understr})
\begin{align}\label{tensorf0}
 f^{\mathcal{K}}_{\gamma}(X) := \sum_{J \in \mathcal{K}} f_{\gamma}(X^{[J]}) 
 = \log \prod_{J \in \mathcal{K}} \prod_{i=1}^{n_J} (\sigma_i^{(J)}(X)^2 + \gamma),
\end{align}
where $\sigma_i^{(J)}(X) = \sigma_i(X^{[J]})$ is the $i$-th singular value of the matrix $X^{[J]} \in \R^{n_J \times n_{J^{\mathsf{c}}}}$.
Thus the matrix version corresponds to $\mathcal{K} = \{\{1\}\}$, whereas for the alternating \IRLS{} method, 
we have also considered the complementary $\mathcal{K} = \{\{2\}\}$. In \cite{Kr21_Asy}, we have 
already reasoned the choice $p = 0$ of the therein appearing weight strength parameter $p \in [0,1]$. Thus, we here only regard\footnote{Most formulas are however easily adaptable to $p \in [0,1]$.}
the thereto corresponding log-det approach laid out above, as opposed to the
other extreme $p = 1$ associated to nuclear norm minimization. 
This leads us to the following, potential solutions to the \ASRMK{} problem.
\begin{definition} \label{Xastdef}
We define
 \begin{align*}
 \mathcal{X}^{\ast} := \{ X^\ast \mid \exists (X_\gamma)_{\gamma > 0} \subset \mathcal{L}^{-1}(y), \ X^\ast = \lim_{\gamma \searrow 0} X_\gamma, \ f^{\mathcal{K}}_{\gamma}(X_\gamma) = \min_{X \in \mathcal{L}^{-1}(y)} f^{\mathcal{K}}_{\gamma}(X)\}.
\end{align*}
\end{definition}
This set of asymptotic, global minimizers indeed yields the desired solutions as we prove in \cref{convergencelemmatensor}. The decline of the parameter $\gamma$ is no less important here as more detailly remarked
on in the predecessor \cite{Kr21_Asy}. It should further be noted that neither the ranks $r^{(J)} := \mathrm{rank}(X^{[J]})$ (cf. \cref{sec:understr}),
nor the families of singular values $\sigma^{(J)}$, $J \in \mathcal{K}$, 
are independent of each other \cite{Kr19_The}, though not prohibitively so in regard of aboves approach.

\subsection{Iteratively reweighted least squares (\texorpdfstring{\IRLS{}}{IRLS})}\label{sec:irls}
In line with the overall generalization, also iteratively reweighted least squares (\IRLS{}) allows to be
applied to the minimization of a sum of ranks of a tensor.
For the matrix case, one version (cf. \cite{Kr21_Asy,MoFa12_Ite}) defines ($\|\cdot\|_F$ being the Frobenius norm)
\begin{align*}
 X^{(i)} := \underset{X \in \mathcal{L}^{-1}(y)}{\mathrm{argmin}} \ \|W_{\gamma^{(i-1)},X^{(i-1)}}^{1/2} X\|_F, \quad W_{\gamma,X} := (X X^T + \gamma I)^{-1},
\end{align*}
for a monotonically decreasing sequence $\{\gamma^{(i)}\}_{i \geq 0} \subset \R_{>0}$.
The tensor variant straightforwardly is given by (see \cref{sec:logdetirls})
\begin{align}\label{IRLS-0K}
 X^{(i)} & := 
 \underset{X \in \mathcal{L}^{-1}(y)}{\mathrm{argmin}} \sum_{J \in \mathcal{K}} \big\|(W^{(J)}_{\gamma^{(i-1)},X^{(i-1)}})^{1/2} (X^{(i)})^{[J]} \big\|_F^2,
 \end{align}
where the weight matrices\footnote{Though certainly interrelated, such are not matricizations of some common tensor.} follow the same generalization with
 \begin{align*}
 W^{(J)}_{\gamma,X} & := W_{\gamma,X^{[J]}} = \big(X^{[J]} (X^{[J]})^T  + \gamma I \big)^{-1}, \quad J \in \mathcal{K}.
\end{align*}
Continued from the vector as well as matrix case, it also here holds true that for
a sequence $X_\gamma \rightarrow \overline{X}$ with \textit{sufficiently fast} declining singular values
\begin{align*}
 \sum_{J \in \mathcal{K}} \big\|W^{(J)}_{\gamma,X_\gamma} X^{[J]}_\gamma \big\|^2_F = 
 \sum_{J \in \mathcal{K}} \sum_{i = 1}^{n_J}
 \frac{\sigma^{(J)}_{i} (X_\gamma)^2}{\sigma^{(J)}_i(X_\gamma)^2 + \gamma} \underset{\gamma \searrow 0}{\longrightarrow} \sum_{J \in \mathcal{K}} \mathrm{rank}(\overline{X}^{[J]}).
\end{align*}
Though largely similar to the matrix case, there is however at least one difference as we discuss in \cref{sec:understr}. Due to its dependence
on $p = 0$ and the family $\mathcal{K}$, we abbreviate aboves algorithm \cref{IRLS-0K} as \IRLS{}-$0\mathcal{K}$.

\subsection{Data sparse optimization}\label{sec:datasparseintro}
With increasing dimensions $d$, the size of the space $\R^{n_1 \times \ldots \times n_d}$
quickly becomes prohibitively large. While for smaller instances, \IRLS{}-$0\mathcal{K}$ is by all
means a viable algorithm, it otherwise remains a theoretical ideal. However, for hierarchical families $\mathcal{K}$,
that is if
 \begin{align}\label{eq:hiercond}
  (J \subset S \quad \vee \quad S \subset J \quad \vee \quad J \cap S = \emptyset) \quad \wedge \quad J \neq S^{\mathsf{c}}, \qquad \forall J, S \in \mathcal{K},
 \end{align}
so called hierarchical decompositions \cite{Gr10_Hie} or, basically synonymously, tensor tree networks (cf. \cite{FaHaNo18_Tre,Kr20_Tre})
provide remedy in the same way the ordinary low rank matrix decomposition does (cf. \cite{Kr21_Asy}). In the
latter case, the data space $\mathcal{D}_r := \{ (Y,Z) \mid Y \in \R^{k_1 \times r},\ Z \in \R^{r \times k_2} \}$
represents the low rank variety 
\begin{align*}
 V^{k_1,k_2}_{\leq r} := \{ A \in \R^{k_1 \times k_2} \mid \mathrm{rank}(A) \leq r \}
\end{align*}
via the surjective (but not injective) bilinear map 
\begin{align} \label{lowrankdec}
 \tau_r: \mathcal{D}_r \rightarrow V_{\leq r}, \quad \tau_r(Y,Z):=Y Z \in \R^{n_1 \times n_2}.
\end{align}
The alternating method \AIRLS{} then only requires to operate on $\mathcal{D}_r$, while directly minimizing $f_\gamma$ subject to relaxed affine constraints (see \cref{sec:tensor}).
In the tensor case, where the rank becomes $r = \{r^{(J)}\}_{J \in \mathcal{K}} \in \N^{\mathcal{K}}$,
the variety
\begin{align} \label{HTtensors}
 V^\mathcal{K}_{\leq r} := \{ X \in \R^{n_1 \times \ldots \times n_d} \mid X^{[J]} \in V^{n_J,n_{J^{\mathsf{c}}}}_{\leq r^{(J)}}, \ J \in \mathcal{K} \},
\end{align}
has a logarithmicly lower dimension and is likewise represented by a data space $\mathcal{D}_r$ together with a simple, 
surjective and multilinear contraction map $\tau_r: \mathcal{D}_r \rightarrow V^\mathcal{K}_{\leq r}$ (see \cref{sec:reloptHT}). Thereby, 
a sparse optimization as for matrices is also possible in higher dimensions.
Ultimately, also \AIRLS{}-$0\mathcal{K}$ (see \cref{AIRLS}) distinguishes itself from well known \textit{unregularized} alternating least squares (ALS)  \cite{HoRowSc12_The}
only through an additional penalty term. However, it thereby not only becomes stable by means of \cite{Kr19_Sta},
but it is derived from and directly minimizes the objective function $f^{\mathcal{K}}_\gamma$ restricted to $V^\mathcal{K}_{\leq r}$.

\section{Underlying structure and global behavior}\label{sec:understr}
Phrased more generalized, we in principle desire to solve the problem (cf. \cite{Kr21_Asy}) of finding
\begin{align}\label{Xastgeneral}
 X^\ast \in \underset{X \in \mathcal{L}^{-1}(y)}{\mathrm{argmin}}\ \mathcal{C}_{\mathcal{V}}(X), \quad \mathcal{C}_{\mathcal{V}}(v) := \min_{V \in \mathcal{V}:\ X \in V} \mathrm{dim}(V),
\end{align}
where in this setting the family of varieties $\mathcal{V}$ is
\begin{align*}
\mathcal{V}_d^{\mathcal{K}} := \{ V^{\mathcal{K}}_{\leq r} \subset \R^{n_1 \times \ldots \times n_d} \mid r = \{r^{(J)}\}_{J \in \mathcal{K}} \in \N_0^{\mathcal{K}} \},
\end{align*}
for $V^{\mathcal{K}}_{\leq r}$ as defined in \cref{HTtensors}.
In general however, the dimension of $V^\mathcal{K}_{\leq r}$
does not equal $\sum_{J \in \mathcal{K}} r^{(J)}$, and is thus not directly represented by the sum of ranks as in \cref{trueobj}.
While
\begin{align*}
V^\mathcal{K}_{\leq \widetilde{r}} \subsetneq V^\mathcal{K}_{\leq r} \quad \Rightarrow \quad \widetilde{r}^{(J)} \leq r^{(J)},\ J \in \mathcal{K},\ \widetilde{r} \neq r \quad \Rightarrow \quad \dim(\mathcal{V}^\mathcal{K}_{\leq \widetilde{r}}) < \dim(\mathcal{V}^\mathcal{K}_{\leq r}),
\end{align*}
neither of the converse implications holds true in general. 
Firstly, some differently indexed varieties are equal since some constellations $r \in \N_0^{\mathcal{K}}$ are \textit{unfeasible} \cite{Kr20_Tre}.
\begin{definition}\label{def:feasible}
 The values $r = \{r^{(J)}\}_{J \in \mathcal{K}}$ are called (un)feasible (for $n \in \N^d$), if there 
 exists (not) at least one tensor $X \in \R^{n_1 \times \ldots \times n_d}$ 
 with $\mathrm{rank}(X^{[J]}) = r^{(J)}$, $J \in \mathcal{K}$.
\end{definition}
For hierarchical sets $\mathcal{K}$, these bounds are (cf. \cite{Kr20_Tre})
  $r^{(J_{\hat{e}})} \leq n_v \prod_{e \in E_v \setminus \{\hat{e}\}} r^{(J_e)}$ for $\hat{e} \in E_v$, $v \in V$.
This natural interrelation of ranks is somewhat beneficial to the simplified sum-of-ranks approach as it excludes some extremal cases. The sum-of-ranks minimization is itself a necessary
relaxation of the (arguably) more desirable objective function $\mathcal{C}_{\mathcal{V}^\mathcal{K}_d}$, yet
it is closer than it might first seem.
What remains however is that, contrarily to the matrix case, the varieties are only partially nested.

\subsection{Determinant expansion and convergence of (global) minimizers}\label{detexptensor}
Following from the matrix case, one can likewise expand the function $f_\gamma^\mathcal{K}$
into squared sums of minors defined as
\begin{align*}
  \mathrm{det}^2_k(A) & := \sum_{I \in \mathcal{P}_k([n_J])} \sum_{J \in \mathcal{P}_k([n_{J^{\mathsf{c}}}]) } \det(A_{I,J})^2, \quad k = 1,\ldots,n_J,
\end{align*}
for $A_{I,J} := \{A_{i,j}\}_{i \in I,j \in J} \in \R^{|I| \times |J|}$ and $\mathcal{P}_k([\ell]) := \{ I \subseteq \{1,\ldots,\ell\} \mid |I| = k\}$.
For simplicity of notation, we further define $\mathrm{det}_0^2(A) := 1$.
\begin{corollary}\label{cor:fgKexp}
Let $X \in \R^{n_1 \times \ldots \times n_d}$ and $\gamma \geq 0$. Then
 \begin{align*}
 \prod_{J \in \mathcal{K}} \prod_{i=1}^{n_J} (\sigma_i^{(J)}(X)^2 + \gamma)
 = \prod_{J \in \mathcal{K}} \sum_{{k^J} = 0}^{n_J} \gamma^{(n_J - {k^J})} \cdot \mathrm{det}_{{k^J}}^2(X^{[J]}) 
  & = \sum_{s = 0}^{\sum_{J \in \mathcal{K}} n_J} \gamma^{\sum_{J \in \mathcal{K}} n_J - s} g_s(X)
\end{align*}
with
\begin{align}
 g_s(X) := \sum_{\{k^J\}_{J \in \mathcal{K}} \in \Xi_s} \prod_{J \in \mathcal{K}} \mathrm{det}_{k^J}^2(X^{[J]}), \label{eq:g_ktensor}
\end{align}
for $\Xi_s := \{ \{k^J\}_{J \in \mathcal{K}} \mid \ 0 \leq k^J \leq n_J,\ J \in \mathcal{K},\ \sum_{J \in \mathcal{K}} k^J = s \}$.
\end{corollary}
\begin{proof}
 As $\prod_{i=1}^{n_J} (\sigma_i^{(J)}(X)^2 + \gamma) = \det(X^{[J]} (X^{[J]})^T + \gamma I)$,
 the first equality follows by \cite{Kr21_Asy}. The third term is merely a restructured version.
\end{proof}
The minimizers of these functions are nested in the sense of the following Lemma.
\begin{lemma}\label{nestedmin}
 For $g_s(X)$, $s = 0,\ldots,\sum_{J \in \mathcal{K}} n_J$, as in \cref{cor:fgKexp}, we have
 \begin{align*}
 g_s(X) = 0  \quad \Leftrightarrow \quad \sum_{J \in \mathcal{K}} \mathrm{rank}(X^{[J]}) < s
 \end{align*}
 for all $X \in \R^{n_1 \times \ldots \times n_d}$. 
\end{lemma}
\begin{proof}
By definition of $g_s(X)$, we have
 \begin{align*}
 g_s(X) \neq 0 \quad & \Leftrightarrow \quad \exists \{k^J\}_{J \in \mathcal{K}} : \sum_{J \in \mathcal{K}} k^J = s \ \forall J \in \mathcal{K} : \mathrm{rank}(X^{[J]}) \geq k^J \\
 \quad & \Leftrightarrow \quad \sum_{J \in \mathcal{K}} \mathrm{rank}(X^{[J]}) \geq s.
\end{align*}
\end{proof}
By \cref{nestedmin}, it directly follows that each $g_s(X) = 0$ implies $g_{s+1}(X) = 0$.
With this structure, we can apply the nested minimization scheme as in \cite{Kr21_Asy} to conclude the following \cref{convergencelemmatensor}.
\begin{theorem} \label{convergencelemmatensor}
Let
\begin{align*}
 s^\ast = \min_{X \in \mathcal{L}^{-1}(y)} \sum_{J \in \mathcal{K}} \rank(X^{[J]}).
 \end{align*}
 Then for any convergent sequence of (global) minimizers $X_\gamma$ of $f^{\mathcal{K}}_{\gamma}(X)$ subject to $\mathcal{L}(X) = y$, we have 
 \begin{align*}
  X^\ast := \lim_{\gamma \rightarrow 0} X_\gamma \in \underset{X \in \mathcal{L}^{-1}(y), \ \sum_{J \in \mathcal{K}} \rank(X^{[J]}) = s^\ast}{\mathrm{argmin}}\ \prod_{J \in \mathcal{K}} \prod_{i = 1}^{\mathrm{rank}(X^{[J]})} \sigma_i^{(J)}(X)
 \end{align*}
with
\begin{align}
 \sigma^{(J)}_{\mathrm{rank}((X^\ast)^{[J]})+1}(X_\gamma)^2 \in \mathcal{O}(\gamma), \quad J \in \mathcal{K}. \label{sigmajboundtensor}
\end{align}
If there is only one $X_{s^\ast} \in \mathcal{L}^{-1}(y)$ with $\sum_{J \in \mathcal{K}} \mathrm{rank}(X_{s^\ast}^{[J]}) = s^\ast$, then $X_\gamma \rightarrow X_{s^\ast}$.
\end{theorem}
\begin{proof}
Since 
$
\mathrm{argmin}_{X \in \mathcal{L}^{-1}} \ g_s(X) \subset \mathrm{argmin}_{X \in \mathcal{L}^{-1}} \ g_{s+1}(X)
$
due to \cref{nestedmin},
the proof is analogous to the corresponding one in \cite{Kr21_Asy}.
\end{proof}
\section{Log-det tensor iteratively reweighted least squares (\texorpdfstring{\IRLS{}-$0\mathcal{K}$}{IRLS-0K})}\label{sec:logdetirls}
Although the global minimizers of $f_\gamma^\mathcal{K}$ yield the sought solution,
it is not practicable to directly minimize these functions or to find its extremal points. As in the matrix case, the map is augmented.
While one here requires to introduce one weight for each $J \in \mathcal{K}$,
most results for the matrix case transfer directly due to the similar structure.

\subsection{Minimization of an augmented function}\label{tensoraugf}
The augmented map\footnote{Due to the distinguishable roles of $J \in \mathcal{K}$ and the map $J^\mathcal{K}_\gamma$
, we here remain faithful to prior literature as for both the letter $J$ has been used before.} analogous to $f_\gamma$ corresponding to the tensor function $f^{\mathcal{K}}_{\gamma}$ is
\begin{align*}
 J^{\mathcal{K}}_{\gamma}(X,\{W^{(J)}\}_{J \in \mathcal{K}}) := \sum_{J \in \mathcal{K}} J_{\gamma,n_J}(X^{[J]},W^{(J)})
 \end{align*}
for
\begin{align*}
J_{\gamma,m}(A,H) & = \mathrm{trace}(H (A A^T  + \gamma I)) - \log \det(H) - m \\
 & = \sum_{J \in \mathcal{K}} \|H^{1/2} A\|_F^2 + \gamma \|H^{1/2}\|_F^2 - \log \det(H) - m,
\end{align*}
where each $W^{(J)} \in \R^{n_J \times n_J}$ ranges over $W^{(J)} = (W^{(J)})^T \succ 0$ (symmetric positive definite).
Consequently, with the same argumentation as in \cite{Kr21_Asy,FoRaWa11_Low}, it is 
 \begin{align*}
  \frac{\partial}{\partial W^{(J)}}\ J^{\mathcal{K}}_{\gamma}(X,\{W^{(J)}\}_{J \in \mathcal{K}}) = X^{[J]} (X^{[J]})^T  + \gamma I - (W^{(J)})^{-1}
 \end{align*}
 and thus
\begin{align}\label{tensorupdateW}
 W^{(J)}_{\gamma,X} :=&\ \underset{W^{(J)} = (W^{(J)})^T \succ 0}{\mathrm{argmin}} J^{\mathcal{K}}_{\gamma}(X,\{W^{(J)}\}_{J \in \mathcal{K}}) = (X^{[J]} (X^{[J]})^T  + \gamma I)^{-1}.
\end{align}
It likewise holds true that
\begin{align}\label{returntoffct}
 f^{\mathcal{K}}_{\gamma}(X) = J^{\mathcal{K}}_{\gamma}(X,\{W^{(J)}_{\gamma,X}\}_{J \in \mathcal{K}}).
\end{align}
Further, the minimizer in $X$ is determined by an ordinary least squares problem.
In order to derive the closed form solution for the minimizer, we note that each $W^{(J)}$, $J \in \mathcal{K}$, defines linear operations
\begin{align*}
(\mathcal{W}^{(J)})^\alpha: \R^{n_1 \times \ldots \times n_d} \rightarrow \R^{n_1 \times \ldots \times n_d}, \quad ((\mathcal{W}^{(J)})^\alpha (X))^{[J]} := (W^{(J)})^\alpha X^{[J]}, \quad \alpha > 0.
\end{align*}
We can thereby write
\begin{align*}
 \sum_{J \in \mathcal{K}} \|(W^{(J)})^{1/2} X^{[J]} \|_F^2 & = \sum_{J \in \mathcal{K}} \| (\mathcal{W}^{(J)})^{1/2} (X) \|_F^2 = \| \overline{\mathcal{W}}^\mathcal{K}(X) \|_F^2,
\end{align*}
where
 $\overline{\mathcal{W}}^{\mathcal{K}}(X) := \{(\mathcal{W}^{(J)})^{1/2}(X)\}_{J \in \mathcal{K}} \in \R^{n_1 \times \ldots \times n_d \times |\mathcal{K}|}$.
Based on the operator $\overline{\mathcal{W}}^{\mathcal{K}}$ (cf. \cite{Kr21_Asy}), the sought minimizer is given by
\begin{align} \label{tensorX_W}
 X_W^\mathcal{K} := \underset{X \in \mathcal{L}^{-1}(y)}{\mathrm{argmin}} J^{\mathcal{K}}_{\gamma}(X,\{W^{(J)}\}_{J \in \mathcal{K}})
 = \widehat{\mathcal{W}}^{-1} \circ \mathcal{L}^\ast \circ (\mathcal{L} \circ \widehat{\mathcal{W}}^{-1} \circ \mathcal{L}^\ast)^{-1}(y)
\end{align}
for 
\begin{align}\label{widehatdef}
 \widehat{\mathcal{W}}^\mathcal{K}(X) := (\overline{\mathcal{W}}^{\mathcal{K}})^\ast \circ \overline{\mathcal{W}}^{\mathcal{K}}(X) = \sum_{J \in \mathcal{K}} \mathcal{W}^{(J)}(X),
\end{align}
where $(\cdot)^\ast$ denotes adjoint operators.
Further, following \cite{Kr21_Asy,FoRaWa11_Low,MoFa12_Ite}, we have
\begin{align}
  \widehat{\mathcal{W}}^\mathcal{K}(X^{\mathcal{K}}_W) \perp \mathrm{kernel}(\mathcal{L}). \label{tensorkernelL}
\end{align}
Vice versa, $X^{\mathcal{K}}_W$ is the unique solution to \cref{tensorkernelL} subject to $\mathcal{L}(X^{\mathcal{K}}_W) = y$.
A more stable update formula is provided by \cite{Kr21_Asy} through
\begin{align}\label{eq:kernelupdate}
 X^{\mathcal{K}}_W = X_0 - \mathcal{K} \circ (\mathcal{K}^\ast \circ \widehat{\mathcal{W}}^\mathcal{K} \circ \mathcal{K})^{-1} \circ \mathcal{K}^\ast \circ \widehat{\mathcal{W}}^\mathcal{K}(X_0),
\end{align}
where $X_0$ is one arbitrary solution to $\mathcal{L}(X_0) = y$ and $\mathcal{K}: \R^{\prod_{i=1}^d n_i - \ell} \rightarrow \R^{n_1 \times \ldots \times n_d}$ is a kernel representation of $\mathcal{L}$, whereby $\mathrm{image}(\mathcal{K}) = \mathrm{kernel}(\mathcal{L})$.
Due to the sum structure, also the gradient properties generalize to the tensor case.
\begin{corollary} \label{stabilizertensor}
 It is
\begin{align}\label{returntofdiv}
 \nabla_X f^{\mathcal{K}}_{\gamma}(X) = \nabla_X J^{\mathcal{K}}_{\gamma}(X,\{W^{(J)}\}_{J \in \mathcal{K}})|_{W^{(J)} = W^{(J)}_{\gamma,X},\ J \in \mathcal{K}}.
\end{align}
 Thus $X$ is a stationary point of $f^{\mathcal{K}}_{\gamma}$ if and only if $X = X^{\mathcal{K}}_W$ for $W^{(J)} = W^{(J)}_{\gamma,X}, \ J \in \mathcal{K}$,
 which means that $(X,\{W^{(J)}_{\gamma,X}\}_{J \in \mathcal{K}})$ is a stationary point of $J^{\mathcal{K}}_{\gamma}$.
\end{corollary}%
As in the matrix case, $\gamma \rightarrow \infty$ provides a unique, canonical starting value.
\begin{corollary}\label{gammainf}
 Independently of $X^{(0)} \in \mathcal{L}^{-1}(y)$, it holds
 \begin{align*}
  \lim_{\gamma \rightarrow \infty} \underset{X \in \mathcal{L}^{-1}(y)}{\mathrm{argmin}} f^\mathcal{K}_\gamma(X) = \lim_{\gamma \rightarrow \infty} X_{\{W^{(J)}_{\gamma,X^{(0)}}\}_{J \in \mathcal{K}}}
  = \underset{X \in \mathcal{L}^{-1}(y)}{\mathrm{argmin}} \|X\|_F,
 \end{align*}
 where the first limit is possibly a set convergence.
\end{corollary}

\subsection{Complementary weights}\label{sec:complwe}
In the matrix case \cite{Kr21_Asy}, there is one more equitable choice $f^{(2)}(A) = \log \det(A^T A + \gamma I)$
as opposed to $f^{(1)}_\gamma(A) = f_\gamma(A) = \log \det(A A^T + \gamma I)$. 
For families $\mathcal{K}$ containing more subsets, each set $J \in \mathcal{K}$ may be replaced by its
complement. For a subset $\mathcal{S} \subset \mathcal{K}$, let therefor
\begin{align}\label{KS}
  \mathcal{K}^{\mathcal{S}} := (\mathcal{K} \setminus \mathcal{S}) \cup \{ J^{\mathsf{c}} \mid J \in \mathcal{S}\}, \quad J^{\mathsf{c}} := [d] \setminus J,
\end{align}
for $W^{(J)} = W^{(J)}_{\gamma,X}$, $J \in \mathcal{K}^{\mathcal{S}}$.
Although the updates $X_W^{(\mathcal{K})}$ and $X_W^{(\mathcal{K}^\mathcal{S})}$ in general differ, the overall properties outlined in 
\cref{sec:logdetirls} are not influenced as
 \begin{align*}
 f^{\mathcal{K}^{\mathcal{S}}}_{\gamma}(X) 
 & =   \sum_{J \in \mathcal{S}} \sum_{i=1}^{n_{J^{\mathsf{c}}}} \log(\sigma_i^{(J^{\mathsf{c}})}(X)^2 + \gamma) + \sum_{J \in \mathcal{K} \setminus \mathcal{S}} \sum_{i=1}^{n_J} \log (\sigma_i^{(J)}(X)^2 + \gamma) \\
 & = f^{\mathcal{K}}_{\gamma}(X)  + \sum_{J \in \mathcal{S}} (n_{J^{\mathsf{c}}} - n_J) \log \gamma.
\end{align*}
While the weights are in that sense interchangeable, switching between complementary weights becomes essential for \AIRLS{}-$0\mathcal{K}$ as captured in \cref{switchlemma}. 

\subsection{Adjusted \texorpdfstring{\IRLS{}-$0\mathcal{K}$}{IRLS-0K} algorithm}
Based on a monotonically declining sequence $\{\gamma^{(i)}\}_{i \geq 0} \subset \R_{> 0}$ (cf. \cref{Xastdef}),
and (optionally) a sequence $\mathcal{S}_i \subset \mathcal{K}$ (cf. \cref{sec:complwe}),
\cref{alg:irlsmrreintws} defines the sequence $\{(X^{(i)}, \{W^{(i,J)}\}_{J \in \mathcal{K}}) \}_{i \geq 0}$ 
with $\mathcal{L}(X^{(i)}) = y$ and $(W^{(i,J)})^T = W^{(i,J)} \succ 0$, $J \in \mathcal{K}$, $i \geq 0$. 
These iterates behave largely analogously to the matrix version \cite{Kr21_Asy} (cf. \cite{MoFa12_Ite,FoRaWa11_Low,CaWaBo08_Enh,DaDeFoGu10_Ite}).
In particular, that case is included in \cref{declinelemmatensor} for $d = 2$ and $\mathcal{K} = \{ \{1\} \}$.
 \begin{algorithm}
  \caption{Iteratively reweighted least squares with switching weights}
  \begin{algorithmic}[1] \label{alg:irlsmrreintws}
  \STATE set $X^{(0)} \in \mathcal{L}^{-1}(y)$, $\gamma^{(0)} > 0$  
  \FOR{$i = 1,2,\ldots$}
  \STATE{set $\mathcal{S}_{i-1} \subset \mathcal{K}$ (cf. \cref{sec:complwe})}
  \STATE{$\{W^{(i-1,J)}\}_{J \in \mathcal{K}^{\mathcal{S}_{i-1}}} \sgets \{W^{(J)}_{\gamma^{(i-1)},X^{(i-1)}}\}_{J \in \mathcal{K}^{\mathcal{S}_{i-1}}}$ (cf. \cref{tensorupdateW})}
  \STATE{$X^{(i)} \sgets X^{\mathcal{K}^{S_{i-1}}}_{W^{(i-1)}}$ \label{line1swtwo} (cf. \cref{tensorX_W})}
  \STATE{set $\gamma^{(i)} \leq \gamma^{(i-1)}$ \label{line3sw}}  
  \ENDFOR
  \end{algorithmic}
\end{algorithm}
\begin{theorem} \label{declinelemmatensor}
 Let $\{(X^{(i)})\}_{i \geq 0}$ be generated by \cref{alg:irlsmrreintws}
 for $\{\mathcal{S}_i\}_{i \in \N_0}$ and the weakly decreasing sequence $\{\gamma_i\}_{i \geq 0} \subset \R_{>0}$. 
 Let further $\mathcal{S}^\ast_\gamma \subset \mathcal{L}^{-1}(y)$ be the stationary points of $f^\mathcal{K}_\gamma|_{\mathcal{L}^{-1}(y)}$ for $\gamma > 0$,
 as well as $\gamma^\ast := \lim_{i \rightarrow \infty} \gamma^{(i)}$.
 \begin{enumerate}[label=(\roman*)]
  \item For each $i \in \N$ and each $\mathcal{S} \subset \mathcal{K}$, it holds 
  \begin{align*}
	  f^{\mathcal{K}^\mathcal{S}}_{\gamma^{(i)}}(X^{(i)}) \leq f^{\mathcal{K}^\mathcal{S}}_{\gamma^{(i-1)}}(X^{(i-1)}). 
	\end{align*}
 \item If $\gamma^\ast > 0$, then the sequences $X^{(i)}$ and $|f^{\mathcal{K}^\mathcal{S}}_{\gamma^{(i)}}(X^{(i)})|$, $\mathcal{S} \subset \mathcal{K}$, remain bounded.
  \item Further, if $\gamma^\ast > 0$, then
      \begin{align} \label{difftozero2}
        \lim_{i \rightarrow \infty} \|X^{(i)} - X^{(i-1)}\|_F = 0
       \end{align}
 and each accumulation point of $X^{(i)}$ is in $\mathcal{S}^\ast_{\gamma^\ast}$.
 \item (See \cref{ivremark}) Let $\Theta \subset \R_{> 0}$ be an arbitrary, infinite, bounded set with its only accumulation point at $\inf(\Theta) = 0$, and let
 \begin{align*}
  \delta_i := \inf_{S \in \mathcal{S}_{\gamma^{(i)}}^\ast} \| X^{(i)} - S \|, \quad i \in \N.
 \end{align*}
 For an arbitrary, bounded sequence $A = \{\alpha_i\}_{i \in \N_0}$ with $\inf(A) > 0$ (e.g. $\alpha_i = 1$, $i \in \N_0$) and for $\gamma^{(0)} = \max(\Theta)$, we recursively define
 \begin{align*}
 \gamma^{(i+1)} = \begin{cases}
		   \theta_i & \mbox{ if } \alpha_i \delta_i < \theta_i \\
		    \gamma^{(i)} & \mbox{ otherwise }
                  \end{cases}, \quad \theta_i := \max \{ z \in \Theta \mid z < \gamma^{(i)} \}, \quad i \in \N_0.
\end{align*}
 Then $\lim_{i \rightarrow \infty} \delta_i = \gamma^\ast = 0$ and for at least one subsequence $\{X^{(i_\ell)}\}_{\ell \in \N}$,
 there exists a sequence of stationary points $\{S_{\ell}\}_{\ell \in \N}$, $S_{\ell} \in \mathcal{S}^\ast_{\gamma^{(i_\ell)}}$, with $\|S_\ell - X^{(i_\ell)}\| \rightarrow 0$.
 \end{enumerate}
\end{theorem}
\begin{remark}\label{ivremark}
 Part $(iv)$ of \cref{declinelemmatensor} as well as its proof are literally the same as in the matrix case \cite{Kr21_Asy}. Roughly, if the sequence $\{\gamma^{(i)}\}_{i \in \N}$ is decreased to $\gamma^\ast = 0$ slowly enough,
then $X^{(i)}$ can only converge to a limit of stationary points of $f_\gamma|_{\mathcal{L}^{-1}(y)}$ for $\gamma \searrow 0$.
The contrary case of too fast decline has been covered in \cite{Kr21_Asy} as well.
\end{remark}
\begin{proof}
See \cref{sec:minorproofs}.
\end{proof}

\section{Relaxed iteratively reweighted least squares}\label{sec:tensor}
Too large mode sizes $n$ or high dimensions $d$ in practice prohibit to even operate on the spaces $\mathcal{L}^{-1}(y)$ or $\R^{n_1 \times \ldots \times n_d}$ directly. 
As hinted on in \cref{sec:datasparseintro}, so called hierarchical decompositions can provide
remedy in the same way low rank matrix decompositions do. This however first requires to relax
the affine constraint $\mathcal{L}(X) = y$.

\subsection{Relaxation of affine constraint}\label{sec:relaxed}
Let 
$a_\gamma(s) := s - \sum_{J \in \mathcal{K}} n_J \log(\gamma)$, $\gamma > 0$. 
As each of these function is monotonically increasing, a composition with such does not change minimizers. We correspondingly define 
\begin{align}\label{tensorfa0}
 f^{a,\mathcal{K}}_{\gamma}(X) := a_\gamma \circ f^{\mathcal{K}}_{\gamma}(X) = \log \prod_{J \in \mathcal{K}} \prod_{i=1}^{\infty} (1 + \frac{\sigma_i^{(J)}(X)^2}{\gamma}),
\end{align}
with $\sigma_i^{(J)}(X) := 0$ for $i > n_J$, $J \in \mathcal{K}$. Likewise, let
$
 J^{a,\mathcal{K}}_{\gamma}(X,\{W^{(J)}\}_{J \in \mathcal{K}}) := a_\gamma \circ J^{\mathcal{K}}_{\gamma}(X,\{W^{(J)}\}_{J \in \mathcal{K}})$.
With the same reasoning as in \cite{Kr21_Asy}, one then defines
\begin{align}
 F_{\gamma,\omega}^{a,\mathcal{K}}(X) & := \| \mathcal{L}(X) - y \|_F^2 + c_{\mathcal{L}}\cdot \omega^2 \cdot f_\gamma^{a,\mathcal{K}}(X), \nonumber \\
 \mathcal{J}^{a,\mathcal{K}}_{\gamma,\omega}(X,W) & := \| \mathcal{L}(X) - y \|_F^2 + c_{\mathcal{L}}\cdot \omega^2 \cdot J^{a}_{\gamma,\mathcal{K}}(X,W). \label{relaxedJ}
\end{align}
for an appropriate scaling constant $c_\mathcal{L}$. As
$\frac{\partial}{\partial \gamma} F_{\gamma,\sqrt{\gamma}}^{a,\mathcal{K}}(X) = c_{\mathcal{L}} \cdot \frac{\partial}{\partial \gamma} (\gamma \cdot f_{\gamma}^{a,\mathcal{K}}(X)) \geq 0$,
the choice $\omega = \sqrt{\gamma}$ seems suitable. In that case, we skip the index $\omega$.
\subsection{Subspace dependent, relaxed optimization algorithm}\label{sec:subdep}
To later incorporate the alternating optimization, we here also consider an additional sequence of subspaces
$\{\mathcal{T}_i\}_{i \in \N_0}$ with $\mathcal{T}_i \subseteq \R^{n_1 \times \ldots \times n_d}$, as well as $\mathcal{T}_i \cap \mathcal{T}_{i-1} \ni X^{(i)}$, $i \in \N_0$.
The latter condition ensures that the previous iterate remains admissible. This then yields the modified \cref{alg:relT}.
 \begin{algorithm}
  \caption{Subspace restricted \IRLS{} with switching weights}
  \begin{algorithmic}[1] \label{alg:relT}
  \STATE set $X^{(0)} \in \mathcal{L}^{-1}(y)$, $\gamma^{(0)} > 0$  
  \FOR{$i = 1,2,\ldots$}
  \STATE{set $\mathcal{S}_{i-1} \subset \mathcal{K}$ (cf. \cref{sec:complwe})}
  \STATE{$\{W^{(i-1,J)}\}_{J \in \mathcal{K}^{(\mathcal{S}_{i-1})}} \sgets \{W^{(J)}_{\gamma^{(i-1)},X^{(i-1)}}\}_{J \in \mathcal{K}^{(\mathcal{S}_{i-1})}}$ (cf. \cref{tensorupdateW})}
\STATE{set a subspace $\mathcal{T}_{i-1} \subset \R^{n_1 \times \ldots \times n_d}$ with $\mathcal{T}_{i-1} \ni X^{(i-1)}$}  
  \STATE{$X^{(i)} \sgets \mathrm{argmin}_{X \in \mathcal{T}_{i-1}}\  \mathcal{J}^{a,\mathcal{K}^{(S_{i-1})}}_{\gamma^{(i-1)}}(X,\{W^{(i-1,J)}\}_{J \in \mathcal{K}^{(S_{i-1})}})$ (cf. \cref{relaxedJ})}
  \STATE{set $\gamma^{(i)} \leq \gamma^{(i-1)}$}  
  \ENDFOR
  \end{algorithmic}
\end{algorithm}
While the objective function is still monotonically decreased as provided by \cref{mondeclt}, to show the remaining parts of \cref{declinelemmatensor}
as far as possible for now remains subject to future research.
\begin{corollary}\label{mondeclt}
 For $X^{(i)}$ as defined by \cref{alg:relT} it holds
\begin{align*}
 F_{\gamma^{(i)}}^{a,\mathcal{K}^\mathcal{S}}(X^{(i)}) \leq F_{\gamma^{(i-1)}}^{a,\mathcal{K}^\mathcal{S}}(X^{(i-1)}),
\end{align*}
for all $i \in \N$ and all $\mathcal{S} \subset \mathcal{K}$.
\end{corollary}
\begin{proof}
 The argumentation is the same as in \cref{declinelemmatensor} part $(i)$ as steps $(a)$ to $(g)$ analogously hold true (cf. \cref{sec:relaxed}).
\end{proof}

\section{Hierarchical decomposition}\label{sec:reloptHT}
%
\def\mlcr{\\}
\def\deltatensor{\delta}
\def\AJtensor{A^{({\hat{J}})}_{\alpha_c',\{{\beta^e}'\}_{e \in E_c};\, \alpha_c,\{\beta^e\}_{e \in E_c}}}
\def\MJmatrix{M^{(\hat{J})}_{{\beta^{e_1}}',\beta^{e_1}}}
\def\Nvprime#1{(N_{#1})_{\alpha_{#1}, \{{\beta^e}'\}_{e \in E_{#1}}}}
\def\Nvtensor#1{(N_{#1})_{\alpha_{#1}, \{{\beta^e}\}_{e \in E_{#1}}}}
\def\HJmatrix#1#2{H^{(#1)}_{{\beta^{#2}}',\beta^{#2}}}
\def\HJinv#1#2{(H^{(#1)} + \gamma I)^{-1}_{{\beta^{#2}}',\beta^{#2}}}
\def\GJprime{G^{({\hat{J}})}_{\{\alpha_{J_e}\}_{e \in E_Y},\alpha_c',\{{\beta^e}'\}_{e \in E_c}}}
\def\GJtensor{G^{({\hat{J}})}_{\{\alpha_{J_e}\}_{e \in E_Y},\alpha_c,\{\beta^e\}_{e \in E_c}}}
\def\WJmatrix{W^{({\hat{J}})}_{\alpha_{{\hat{J}}}',\alpha_{{\hat{J}}}}}
\def\YJmatrix#1{Y^{(J_{#1})}_{\alpha_{J_{#1}}, \beta^{#1}}}
\def\PJprime{P^{(\hat{J})}_{\{\alpha_v\}_{v \in p},\{{\beta^{e}}'\}_{e \in \partial{E}_p}}}
\def\PJtensor{P^{(\hat{J})}_{\{\alpha_v\}_{v \in p},\{{\beta^{e}}\}_{e \in \partial{E}_p}}}
\def\PcJprime{P^{(+c,\hat{J})}_{\alpha_p,\{{\beta^e}'\}_{e \in E_Y}}}
\def\PcJtensor{P^{(+c,\hat{J})}_{\alpha_p,\{\beta^e\}_{e \in E_Y}}}
\def\BJ#1#2{B^{(J_{#1})}_{{\beta^{#2}}',\beta^{#2}}}
\def\BJtilde#1#2{\widetilde{B}^{(J_{#1})}_{{\beta^{#2}}',\beta^{#2}}}
\def\netprod{\prod}
\def\netsum#1{\sum_{#1}}
We briefly reintroduce hierarchical
tensor decompositions \cite{Gr10_Hie} as tensor tree networks with reference to the introductory \cref{sec:datasparseintro}.
For further reading, we recommend \cite{GrKrTo13_Ali,No17_Low,FaHaNo18_Tre,Gr10_Hie,KoBa09_Ten,La09_Asu,Kr20_Tre}. 
\subsection{Notational deviation}\label{sec:notdev}
In the following, $G = (V,E)$ denotes a tree graph with vertices $V \supseteq [d]$ and edges $E \subseteq \{ \{v,w\} \mid v\neq w \in V\}$.
Due to the complex description of general tensor (tree) networks, we 
require a certain minimum of notational deviation.
That is, we dismiss the order of modes when indexing tensors.
Instead, in order to avoid ambiguity, each specific object is consistently referenced 
with the same, distinctly assigned \textit{labels}, based on the graph $G = (V,E)$. The first group is given by
$\alpha_S = \{\alpha_\mu\}_{\mu \in S}$, for $\alpha_S \in [n_S]$,
$n_S = \prod_{\mu \in S} n_\mu$, $S \subseteq V$. 
We set $n_\mu = 1$ for $\mu > d$, but
any such $\alpha_\mu$ is only denoted when required for notational simplicity.
Further, the second group is given by $\beta = \{\beta^e\}_{e \in E}$ with $\beta^e \in [r^{(J_e)}]$, $J_e \in \mathcal{K}$ (see \cref{sec:KtoG}), whereas the measurement index is denoted by $\zeta \in [\ell]$.
For each such label, we correspondingly define the spaces
\begin{align*}
 \mathfrak{H}_{\alpha_\mu} := \R^{[n_\mu]},\ v \in V, \quad \mathfrak{H}_{\beta^e} := \R^{[r^{(J_e)}]},\ e \in E, \quad \mathfrak{H}_\zeta := \R^{[\ell]}.
\end{align*}
The entirety of labels is formally required to be ordered, but the exact ordering is irrelevant. To each collection $\Gamma$ of such labels, we consequently assign the space 
\begin{align}\label{HGamma}
 \mathfrak{H}_\Gamma := \bigotimes_{\gamma \in \Gamma} \mathfrak{H}_\gamma.
\end{align}
Some, in particular labels corresponding to edges also appear as unequally treated, so called \textit{primed labels} ${\beta^e}' \neq \beta^e$, $e \in E$. Each is however still thought to refer to the same, implicitly declared positions of its unprimed twin.
Throughout this section, it shall become apparent that it is in fact mostly redundant to explicitly denote these labels. While we nevertheless here hold on to indices, \cref{app:sec:reloptHT} does make use of this fact to more compactly repeat some of following statements and lay out their proofs.
What is here introduced as notation, is the basis to the formalized arithmetic introduced in \cite{Kr20_Tre}. For the \textsc{Matlab} toolbox that realizes the latter through automated contractions, on which the implementation of \AandIRLS{}-$0\mathcal{K}$ is based on, please contact the author.
\subsection{Graph notation}\label{sec:graphtheory}
We denote each the path from excluding $c \in V$ to excluding $v \in V \setminus \{v\}$ within a tree $G = (V,E)$ as the unique ordered set
\begin{align}\label{eq:pathdef}
 c \mathring{\rightharpoondown} v := (p_1,\ldots,p_{-1}) = p \subset V,
\end{align}
for which $\{c,p_1\} \in E$, $\{p_i,p_{i+1}\} \in E$, $i = 1,\ldots,|p|-1$
as well as $\{p_{-1},v\} \in E$. We further define the neighbors of $v \in V$, as well as the predecessor and set of descendants of $v \in V \setminus \{c\}$ relative to $c \in V$ as
\begin{align*}
\mathrm{neigh}(v) := \{ h \in V \mid \{ h,v \} \in E \}, \quad
 \mathrm{pred}_c(v) := p_{-1}, \quad \mathrm{desc}_c(v) := \mathrm{neigh}(v) \setminus \{p_{-1}\}.
\end{align*}
We define the branches relative to $c \in V$ as
\begin{align*}
 \mathrm{branch}_c(v) := \{v\} \cup \{ b \in V \setminus \{c,v\} \mid v \in c \mathring{\rightharpoondown} b \}.
\end{align*}
Each root $c \in V$ splits the graph into the multiple connected components of $V \setminus \{c\}$,
\begin{align*}
 \dot{\bigcup}_{h \in \mathrm{neigh}(c)} \mathrm{branch}_c(h) = V \setminus \{c\}.
\end{align*}
For any $v \neq w \in V$, we further define the sets $J_w(v) := \mathrm{branch}_w(v) \cap [d]$.
Thus if $e = \{v,w\} \in E$ is an edge, then $J_w(v)\ \dot{\cup}\ J_v(w) = [d]$. 
\subsection{Tree corresponding to hierarchical family}\label{sec:KtoG}
Without loss of generality, we from here on postulate that hierarchical families $\mathcal{K}$ (cf. \cref{sec:datasparseintro}) are by definition also \textit{dimension separating}. That is, we assume that there does not exist a map $\pi: [d] \rightarrow [d-1]$, for which $\pi(J) \notin \{ \pi(\hat{J}),\, [d-1] \setminus \pi(\hat{J}) \}$ for all $J,\hat{J} \in \mathcal{K}$.
\begin{lemma}\label{Kcorlem}
 Each (dimension separating) hierarchical family $\mathcal{K}$ defines an, up to equivalence, unique tree $G_\mathcal{K} = (V,E)$, $V \supseteq [d]$ and root $c \in V$, for which $|E| = |V| - 1 = |\mathcal{K}|$ and $\mathcal{K} = \{ J_c(v) \}_{v \in V \setminus \{c\}}$
 ---
 and vice versa.
\end{lemma}
\begin{proof}
 See for example \cite{Kr20_Tre,Gr10_Hie}.
\end{proof}
\begin{definition}
 Let $G_\mathcal{K}$ correspond to the hierarchical family $\mathcal{K}$.
 We define $J_e \in \{J_w(v),J_v(w)\}$, $e = \{v,w\} \in E$, as each the one set that is contained in $\mathcal{K}$.
\end{definition}
This convention implies a bijection $\mathcal{K} = \{ J_e \mid e \in E\}$ to $E$.
The simple graph that corresponds to the matrix case $\mathcal{K}_2 = \{\{1\}\}$ for $d = 2$ is for instance given by
the tree 
\begin{align*}
    G_{\mathcal{K}_2} = ( V, E ),\quad V = \{1,2\},\quad E = \{\{1,2\}\}, 
\end{align*}
whereby $J_{\{1,2\}} = \{ 1 \}$.
For $\mathcal{K}_{\mathrm{Tucker}} = \{\{1\},\ldots,\{d\}\}$ (cf. \cref{example:Tucker}), we have
 \begin{align}\label{Tuckergraph}
  G_{\mathcal{K}_{\mathrm{Tucker}}} = (V,E), \quad V = \{1,\ldots,d+1\}, \quad E = \{ \{1,d+1\},\ldots,\{d,d+1\} \},
 \end{align}
 and $J_{\{\mu,d+1\}} = \{ \mu \}$, $\mu \in [d]$.
 As required later, for subsets $S \subset V$, we further define
 \begin{align}
 E_S & := \{ \{ v,w \} \subset E \mid v \in S, \ w \in \mathrm{neigh}(v) \}, \label{ES} \\
 \mathring{E}_S & := \{ \{ v,w \} \subset E \mid v,w \in S \}, \ \partial{E}_S := E_S \setminus \mathring{E}_S. \nonumber
 \end{align}
For $S = \{v\}$, $v \in V$, we may skip set brackets. Thus, $E_v = \{ \{v,h\} \}_{h \in \mathrm{neigh}(v)}$.

\subsection{Representation map corresponding to tree}\label{sec:repmap}
Whereas each hierarchical family $\mathcal{K}$ defines a tree $G_\mathcal{K} = (V,E)$, 
each such (not necessarily rooted) graph together with $r \in \N^\mathcal{K}$ in turn defines
a certain data space $\mathcal{D}_r$
and a representation map $\tau_r: \mathcal{D}_r \rightarrow \R^{n_1 \times \ldots \times n_d}$ for values $r \in \N^\mathcal{K}$.
\begin{definition}\label{def:mv}
With reference to \cref{sec:notdev}, let
\begin{align*}
 \mathcal{D}_r := \bigtimes_{v \in V} \mathfrak{H}_{\mathfrak{m}_v},
 \quad \mathfrak{m}_v := \{ \beta^e \}_{e \in E_v} \cup 
 \begin{cases}
     \{\alpha_v\} & \mbox{if } v \in [d], \\
    \emptyset & \mbox{otherwise}.
\end{cases}  
\end{align*}
\end{definition}
The dimension of each node $N_v \in \mathfrak{H}_{\mathfrak{m}_v}$, $\{N_v\}_{v \in V} \in \mathcal{D}_r$, is thus the degree of $v \in V$, plus one if $v \in [d]$.
The representation map $\tau_r$ is now defined as the map that proceeds each
a contraction over modes with common labels.
With the notation declared in \cref{sec:notdev}, we may write
\begin{align}\label{eq:taur}
 \tau_r(N)_{\alpha_1,\ldots,\alpha_d} := \sum_{\beta^e \in E} \prod_{\mu \in [d]} \Nvtensor{v} \prod_{v \in V \setminus [d]} (N_v)_{\{\beta^e\}_{e \in E_v}},
\end{align}
where $\alpha_\mu \in [n_\mu]$, $\mu \in [d]$.
\begin{example}
 In the matrix case with $r = r^{(J_{\{1,2\}})} \in \N$, we simply have an ordinary matrix multiplication (cf. \cref{sec:datasparseintro})
$\tau_r(Y,Z)_{\alpha_1,\alpha_2} = \sum_{\beta = 1}^{r} Y_{\alpha_1,\beta} Z_{\beta,\alpha_2}$, where the summation ranges over $\alpha_1 \in [n_1]$ and $\alpha_2 \in [n_2]$.
Here, $\beta = \beta^{\{1,2\}} \in [r]$ is the label assigned to the only edge.
\end{example}
\begin{example}\label{example:Tucker}
 For $d \in \N$, the Tucker format \cite{Tucker66_Som} or MLSVD\footnote{subject to further orthonormality constraints (cf. \cref{rootedtrees})} \cite{LaMoVa00_AMu} is defined through the graph $\mathcal{K}_{\mathrm{Tucker}}$ \cref{Tuckergraph}
 and consists of the components $\{N_\mu\}_{\mu = 1}^{d+1} \in \mathcal{D}_r$ of sizes $N_\mu \in \R^{n_\mu \times r^{(J_{\{\mu,d+1\}})}}$
 and $N_{d+1} \in \R^{r^{(J_{\{1,d+1\}})} \times \ldots \times r^{(J_{\{d,d+1\}})}}$.
 The corresponding contraction map is given by (though less convenient when written out in particular cases)
 \begin{multline*}
  X_{\alpha_1,\ldots,\alpha_d} = \tau_r(N_1,\ldots,N_d,N_{d+1})_{\alpha_1,\ldots,\alpha_d} \\
= \sum_{\beta^{\{1,d+1\}} = 1}^{r^{(J_{\{1,d+1\}})}} \ldots \sum_{\beta^{\{d,d+1\}} = 1}^{r^{(J_{\{d,d+1\}})}} (N_1)_{\alpha_1,\beta^{\{1,d+1\}}} \ldots (N_d)_{\alpha_d,\beta^{\{d,d+1\}}} (N_{d+1})_{\beta^{\{1,d+1\}},\ldots,\beta^{\{d,d+1\}}},
 \end{multline*}
 for $\alpha_\mu = 1,\ldots,n_\mu$, $\mu = 1,\ldots,d$ as visualized in \cref{Tucker_graph}.
\tikzfigure{htb}{Tucker_graph}{\label{Tucker_graph}{\normalfont[Left]} The contraction diagram for the Tucker representation in \cref{example:Tucker} for $d = 4$.
The dotted line indicates the part which for $J = \{1\}$ yields $Z^{(J)}$,
whereas $Y^{(J)} = N_1$ (cf. \cref{YJZJ}). {\normalfont[Right]} A balanced binary hierarchical Tucker (HT) representation (cf. \cref{sec:exhaustive}) for the exhaustive family $\mathcal{K} = \{\{1,2\},\{1\},\ldots,\{4\}\}$, $Y^{(\{1,2\})} = \tau_r(\{N_1,N_2,\widetilde{N}_{d+1}\})$ (cf. \cref{eq:partcont}). In contrast to conventional literature (cf. \cite{Gr10_Hie}), the root node has been omitted as it is redundant here (cf. \cite{Kr20_Tre}).}%
\end{example}
While initially defined on the whole network, we can also extend the map $\tau_r$
to contract nodes over any subset $S \subset V$ via
\begin{align}\label{eq:partcont}
 \tau_r(\{N_s\}_{s \in S})_{\{\alpha_s\}_{s \in S},\{\beta^e\}_{e \in \partial E_S}}               := \sum_{\beta^e\,:\,e \in \mathring{E}_S} \prod_{v \in S} \Nvtensor{v},
\end{align}
for $\alpha_s \in [n_s], \ s \in S$, and $\partial E_S$ as well as $\mathring{E}_S$ as defined by \cref{ES}. Here, some $\alpha_v$, that is for $v > d$, are redundant (cf. \cref{sec:notdev}).
\subsection{Decomposition theorem}
The following theorem is fundamental to hierarchical tensor approximation theory.
\begin{theorem}[\mbox{\cite{Gr10_Hie}}]\label{reprthm}
 Let $G_\mathcal{K}$ be the tree corresponding to the hierarchical family $\mathcal{K}$ (cf. \cref{Kcorlem}).
 Then for each $r = \{r^{(J)}\}_{J \in \mathcal{K}} \in \N^\mathcal{K}$,
 the according multilinear representation map $\tau_r: \mathcal{D}_r \rightarrow \R^{n_1 \times \ldots \times n_d}$
 is non-injective with $\mathrm{image}(\tau_r) = V^\mathcal{K}_{\leq r}$.
\end{theorem}
In other words, for each tensor $X \in \R^{n_1 \times \ldots \times n_d}$ with $\mathrm{rank}(X^{[J]}) \leq r^{(J)}$, $J \in \mathcal{K}$,
there exists a (non-unique) decomposition $N \in \mathcal{D}_r$ with $X = \tau_r(N)$.
Each edge $e = \{v,w\} \in E$, assuming $J = J_e = J_w(v) \in \mathcal{K}$, splits the tree into two disconnected subgraphs and yields a corresponding matrix decomposition
\begin{align}\label{YJZJ}
 X^{[J]} = Y^{(J)} Z^{(J)}, \quad Y^{(J)} \in \R^{[n_J] \times r^{(J)}}, \quad Z^{(J)} \in \R^{r^{(J)} \times [n_{J^\mathsf{c}}]}.
\end{align}
The matrices $Y^{(J)}$ and $Z^{(J)}$ are obtained by contractions over each $(N_h)_{h \in \mathrm{branch}_w(v)}$
and $(N_h)_{h \in \mathrm{branch}_v(w)}$, respectively. In explicit, abbreviating $S = \mathrm{branch}_w(v)$, we have
\begin{align*}
 Y^{(J)}_{\alpha_J,\beta^e} = \tau_r(\{N_s\}_{s \in S})_{\{\alpha_\mu\}_{\mu \in J},\beta^e} = \sum_{\beta^e\,:\,e \in \mathring{E}_S} \prod_{v \in S} \Nvtensor{v},
\end{align*}
for $\mathring{E}_S$ as defined in \cref{sec:KtoG}. 
Given \cref{YJZJ}, it is easy to see that indeed $\mathrm{image}(\tau_r) \subseteq V^\mathcal{K}_{\leq r}$,
whereas the other direction requires some more work (cf. \cite{Gr10_Hie,Kr20_Tre}).
\begin{lemma}\label{manidim}
 The dimension of the variety corresponding to a feasible $r \in \N^\mathcal{K}$ for a hierarchical family $\mathcal{K}$ is
 \begin{align*}
  \dim(V^\mathcal{K}_{\leq r}) = \big(\sum_{\mu \in [d]} n_v \prod_{e \in E_\mu} r^{(J_e)}\big) + \big(\sum_{v \in V \setminus [d]} \prod_{e \in E_v} r^{(J_e)}\big) - \sum_{e \in E} (r^{(J_e)})^2,
 \end{align*}
where $G_\mathcal{K} = (V,E)$ is the corresponding graph. The set $V^\mathcal{K}_{= r}$ in turn is a manifold of equal dimension.
\end{lemma}
\begin{proof}
 Follows by a generalization of the argumentation in \cite{HoThRe12_Onm,UsVa13_The}\footnote{The rank considered therein is implicitly assumed to be feasible.}.
\end{proof}

\subsection{Exhaustive hierarchical families}\label{sec:exhaustive}
The larger the family $\mathcal{K}$, the more regularizing the \IRLS{} approach. Thus, one may desire such to be exhaustive in the following sense.
\begin{definition}
 Let $\mathcal{K}$ be a hierarchical family. We say $\mathcal{K}$ is exhaustive
 if there does not exist another hierarchical family $\widetilde{\mathcal{K}}$ with $\widetilde{\mathcal{K}} \supsetneq \mathcal{K}$.
\end{definition}

Exhaustive hierarchical families in a certain sense yield particularly data sparse formats as specified in the following \cref{maxlemma}.
For any such family, it further holds $|\mathcal{K}| = 2d - 3 = |E|$ and $|V| = 2d - 2$ (cf. \cref{Kcorlem}).

\begin{lemma}\label{maxlemma}
 Let $\mathcal{K}$ be an exhaustive hierarchical family. Then $G_\mathcal{K}$ consists only of inner vertices $v \in V \setminus [d]$ of
 degree $3$ and leafs $v \in [d] \subset V$ of degree $1$.
\end{lemma}
\begin{proof}
 See for instance \cite{Gr10_Hie,Kr20_Tre}.
\end{proof}

The Tucker family $\mathcal{K}_{\mathrm{Tucker}}$ for example is not exhaustive (for $d \geq 4$). The degree of the
vertex $d + 1 \in V$ is $d$, whereby the dimension of the node $N_{d+1}$ is $d$ as well.
For $d = 4$, all exhaustive families are equivalent (up to permutation of modes) to $\mathcal{K} = \{\{1,2\},\{1\},\{2\},\{3\},\{4\}\}$ (see \cref{Tucker_graph}).
In general, exhaustive hierarchical families
correspond to so called binary hierarchical Tucker formats (cf. \cite{Gr10_Hie,Kr20_Tre}).

\subsection{Rooted trees and orthonormalization}\label{rootedtrees}
A root $c \in V$, if at all, may be chosen freely, leading us back to the choice of complementary weights in \cref{sec:complwe}.
\begin{lemma}\label{Svlem}
 For each $c \in V$, there exists a unique subset $\mathcal{S}_c \subset \mathcal{K}$
for which $\mathcal{K}^{\mathcal{S}_c} = \{ J_c(v) \mid v \in V \setminus \{c\} \}$ (cf. \cref{sec:graphtheory}).
\end{lemma}
\begin{proof}
 Follows directly with $\mathcal{S}_c = \{ J_c(v)^{\mathsf{c}} \mid J_c(v) \notin \mathcal{K}, \ v \in V \setminus \{c\}\} \subseteq \mathcal{K}$.
\end{proof}
The set equality in \cref{Svlem} implies that for each $J \in \mathcal{K}^{\mathcal{S}_c}$, there is a unique vertex $v =: v_{c,J} \in V \setminus \{c\}$
with $J = J_c(v)$. Note that only the sets $\mathcal{S}_c$, $c \in V$, again lead to hierarchical families $\mathcal{K}^{\mathcal{S}_c}$ as opposed to the $2^{|\mathcal{K}|}$ generally possible subsets $\mathcal{S} \subset \mathcal{K}$.
One can utilize the non-injectivity of the map $\tau_r$ to orthonormalize the representation in the sense of the following \cref{canonN}, yet
without the need to calculate the represented, full tensor. 
\begin{theorem}\label{canonN}
 Let $r^{(J)} = \mathrm{rank}(X^{[J]})$, $J \in \mathcal{K}$.
 Then there exists a representation $X = \tau_r(N)$, $N \in \mathcal{D}_r$, such that
$Y^{(J)} \in \R^{[n_J] \times r^{(J)}}$, $J \in \mathcal{K}^{\mathcal{S}_c}$,
as defined in \cref{YJZJ}, are orthonormal matrices. 
\end{theorem}
\begin{proof}
 Can for example be found in \cite{Kr20_Tre}.
\end{proof}

Note that the matrices $Y^{(J)}$ in \cref{canonN} are defined via the representation $N$. If may further be achieved that
these matrices each consist of the left singular vectors, $Y^{(J)} = U^{(J)}$, of the compact matrix SVDs $X^{[J]} = U^{(J)} \Sigma^{(J)} (V^{(J)})^T$, $J \in \mathcal{K}^{\mathcal{S}_c}$. Thereby, the decomposition in fact becomes essentially unique\footnote{Essentially here refers to the same weak uniqueness as for the conventional matrix SVD.} \cite{Gr10_Hie,Kr20_Tre}.
However, mere orthonormality is in general sufficient and can be ensured
with significantly less effort in an alternating optimization.
In case of the Tucker format (\cref{example:Tucker}), if indeed $Y^{(J)} = U^{(J)}$, this canonical form is specifically known as MLSVD \cite{LaMoVa00_AMu}, while for the tensor train format \cite{Os11_Ten}, it is known as canonical MPS \cite{Vi03_Eff}. General canonical forms of tensor tree networks and their properties are further discussed in \cite{Kr20_Tre}.

\section{Alternating iteratively reweighted least squares (\texorpdfstring{\AIRLS{}-$0\mathcal{K}$}{AIRLS-0K})}\label{AIRLS}
In this section, let $\mathcal{K}$ be a hierarchical family,
$G_\mathcal{K} = (V,E)$ the corresponding tree as well as $\tau_r: \mathcal{D}_r \rightarrow V^\mathcal{K}_{\leq r}$, with $\mathcal{D}_r = \bigtimes_{v \in V} \mathfrak{H}_{\mathfrak{m}_v}$, the representation map for $r \in \N^\mathcal{K}$ as described in \cref{sec:reloptHT}.
The idea of alternating least squares (ALS) is to in each step fixate $N = \{N_v\}_{v \in V} \in \mathcal{D}_r$
but the one component $N_c$, where the root $c \in V$ iteratively cycles through all vertices. 
We therefor define the linear map
\begin{align*}
 \mathcal{N}_{\neq c}: \mathfrak{H}_{\mathfrak{m}_c} \rightarrow V^\mathcal{K}_{\leq r}, \quad \mathcal{N}_{\neq c}(\hat{N}_c) = \tau_r(\{N_v\}_{v \in V \setminus \{c\}} \cup \{\hat{N}_c\}).
\end{align*}
As the image of that map is independent of the specific, chosen representation,
we obtain the (well defined) subspace (cf. \cref{sec:subdep})
\begin{align}\label{subsp}
 \mathcal{T}_c(\tau_r(N)) := \mathrm{image}(\mathcal{N}_{\neq c}). 
\end{align}
Though one avoids to ever calculate the full tensor $X = \tau_r(\{N_v\}_{v \in V}) \in \R^{n_1 \times \ldots \times n_d}$,
we define the resulting update as
\begin{align}\label{defXwceq}
 X_{\gamma,\omega,W}^{N,c} & := \mathrm{argmin}_{X \in \mathcal{T}_c(N)}\  \mathcal{J}^{a,\mathcal{K}^{\mathcal{S}_c}}_{\gamma,\omega}(X,\{W^{(J)}\}_{J \in \mathcal{K}^{\mathcal{S}_c}}),
\end{align}
as well as $(N_c)_{\gamma,\omega,W}^{N,c} \in \mathfrak{H}_{\mathfrak{m}_c}$ via $\mathcal{N}_{\neq c}((N_c)_{\gamma,\omega,W}^{N,c}) := X_{\gamma,\omega,W}^{N,c}$. The subset $\mathcal{S}_c \in \mathcal{K}$
is defined as by \cref{Svlem}, the objective functions in \cref{relaxedJ}.
The subsequent sections are summarized in \cref{AIRLS-0K}, though it generates the
same iterates $X^{(i)} = \tau_r(N^{(i)})$, $i \in \N_0$, as \cref{alg:relT} when the subspaces are chosen according to \cref{subsp}.

\subsection{Sweeps, micro steps and stability}\label{repbascal}
The update $X_{\gamma,\omega,W}^{N,c}$ (cf. \cref{defXwceq}) is independent of the specific, chosen representation $N$ of the previous iterate $X$ (cf. \cref{reprthm}). 
Thus, for each $c \in V$, the updating maps
\begin{align*}
  \mathcal{M}^{(c)}_{r}: \mathcal{D}_r \rightarrow \mathcal{D}_r, \quad  \mathcal{M}^{(c)}_{r}(N) := \{N_v\}_{v \in V \setminus \{c\}} \cup \{(N_c)_{\gamma,\omega,W}^{N,c}\}, \quad W = W_{\gamma,\tau_r(N)},
  \end{align*}
  operating on the data space, as well as the one operating on the full tensor space,
  \begin{align*}
 \zeta_{\mathcal{M}^{(c)}}: \R^{n_1 \times \ldots \times n_d} \rightarrow \R^{n_1 \times \ldots \times n_d}, \quad \zeta_{\mathcal{M}^{(c)}}(X) := \tau_{r(X)} \circ \mathcal{M}^{(c)}_{\tau_{r(X)}} \circ \tau_{r(X)}^{-1}(X),
\end{align*}
are well defined. Here, $r(X) \in \N^\mathcal{K}$ denotes the ranks of each $X$ and $N = \tau_r^{-1}(X)$ is each
an arbitrary representation. A whole sweep (for fixed $\gamma$ and $\omega$) is defined as
\begin{align*}
 \mathcal{M}_r := \bigcirc_{c \in V}\ \mathcal{M}_r^{(c)}, \quad \zeta_{\mathcal{M}} := \bigcirc_{c \in V}\ \zeta_{\mathcal{M}^{(c)}},
\end{align*}
where the order of composition may be chosen as most suitable.
Issues around these functions in particular concerning stable rank adaptivity 
have been discussed in \cite{Kr19_Sta}.

\subsection{Representation based evaluation}
In order to obtain a practically viable algorithm, 
it remains to show that each next iterate $X_{\gamma,\omega,W}^{N,c}$ (cf. \cref{defXwceq}) 
given $W = W_{\gamma,\tau_r(N)}$ and $\omega = \sqrt{\gamma} > 0$ (cf. \cref{sec:relaxed}) can indeed be calculated through its representation, that is, without 
the need to construct full tensors in $\R^{n_1 \times \ldots \times n_d}$.
The updated node $(N_c)_{\gamma,\omega,W}^{N,c}\in \mathfrak{H}_{\mathfrak{m}_c}$, for which $X_{\gamma,\omega,W}^{N,c} = \mathcal{N}_{\neq c}((N_c)_{\gamma,\omega,W}^{N,c})$, is given by the lineare least squares problem
(cf. \cref{tensoraugf,sec:complwe,sec:relaxed})
\begin{align*}
 (N_c)_{\gamma,\omega,W}^{N,c} = \underset{\widetilde{N}_c \in \mathfrak{H}_{\mathfrak{m}_c}}{\mathrm{argmin}}\ \| \mathcal{L} \circ \mathcal{N}_{\neq c}(\widetilde{N}_c) - y \|^2 + c_{\mathcal{L}} \gamma \sum_{J \in \mathcal{K}^{\mathcal{S}_c}} \|(\mathcal{W}^{(J)})^{1/2} \mathcal{N}_{\neq c}(\widetilde{N}_c) \|_F^2,
\end{align*}
for $\mathcal{W}^{(J)}$ as in \cref{widehatdef}. The minimizer is thus given as solution $(N_c)_{\gamma,\omega,W}^{N,c} := \widetilde{N}_c$ to 
\begin{align}\label{updateformula}
 \mathcal{N}_{\neq c}^\ast \circ \mathcal{L}^\ast \circ \mathcal{L} \circ \mathcal{N}_{\neq c}(\widetilde{N}_c)
 + c_{\mathcal{L}} \gamma \sum_{J \in \mathcal{K}^{\mathcal{S}_c}} \mathcal{N}_{\neq c}^\ast \circ \mathcal{W}^{(J)} \circ \mathcal{N}_{\neq c}(\widetilde{N}_c) = \mathcal{N}_{\neq c}^\ast \circ \mathcal{L}^\ast(y).
\end{align}
Following are two aspects that are required for a representation based evaluation.
The first one in \cref{sec:Ldecomp} depends on the operator $\mathcal{L}$ itself and
can in that sense not be influenced. The second one in \cref{sec:lrW} in turn merely
asks for the right choices of $\mathcal{S}_c$, namely the one in \cref{Svlem}, 
and can thus always be achieved.
\subsection{Decomposition of measurement operator}\label{sec:Ldecomp}
Like each linear operator, $\mathcal{L}: \R^{n_1 \times \ldots \times n_d} \rightarrow \R^\ell$ has a tensor description 
$L \in \R^{\ell \times n_1 \times \ldots \times n_d}$
in terms of
\begin{align}
 \label{LXtensor} \mathcal{L}(X)_\zeta = \sum_{\alpha_1 = 1}^{n_1} \ldots \sum_{\alpha_d = 1}^{n_d} L_{\zeta,\alpha_1,\ldots,\alpha_d} X_{\alpha_1,\ldots,\alpha_d}.
\end{align}
This tensor $L$ must itself somehow allow for an efficient handling.
Similar to \cref{reprthm}, each $L \in \R^{\ell \times n_1 \times \ldots \times n_d}$
can for some $r_L \in \N^\mathcal{K}$ (assumed to be low) be decomposed into lower dimensional components. We therefor define
\begin{align*}
 \mathcal{D}_{r_L}^L := \bigtimes_{v \in V} \mathfrak{H}_{\mathfrak{m}^L_v},
 \quad \mathfrak{m}^L_v := \{ \varepsilon^e \}_{e \in E_v} \cup 
 \begin{cases}
     \{\zeta, \alpha_v\} & \mbox{if } v \in [d], \\
    \emptyset & \mbox{otherwise}.
\end{cases}  
\end{align*}
The symbols $\varepsilon^e$, $e \in E$, are additional labels with range $\varepsilon^e \in [r^{(J_e)}_L]$, whereas $\zeta \in [\ell]$.
The assigned multilinear representation map is
\begin{align*}
 \rho_{r_L}(L)_{\zeta,\alpha_1,\ldots,\alpha_d} := \sum_{\varepsilon^e\,:\, e \in E} \prod_{\mu \in [d]} (L_\mu)_{\zeta,\alpha_\mu,\{\varepsilon^e\}_{e \in E_\mu}} \prod_{v \in V \setminus [d]} (L_v)_{\{\varepsilon^e\}_{e \in E_v}},
\end{align*}
For simplicity of notation, as with the representation $N$, we will also denote indices $\zeta$ and $\alpha_v$ in nodes $L_v$ with $v > d$. While this is formally compatible as long as $(L_v)_{\zeta,\alpha_v,\{\varepsilon^e\}_{e \in E_v}}$, $v > d$, is constant in $\zeta$, these indices can likewise be omitted.
Sampling operators for instance can be decomposed for $r_L^{(J)} \equiv 1$, $J \in \mathcal{K}$.
\begin{example}\label{example:TuckerL}
 For $d \in \N$, the operator decomposition corresponding to the Tucker graph $\mathcal{K}_{\mathrm{Tucker}}$ \cref{Tuckergraph}
 consists of the components $\{L_v\}_{v \in V}$ of sizes $L_\mu \in \R^{\ell \times n_\mu \times r^{(\{\mu\})}}$
 and $L_{d+1} \in \R^{r^{(J_{\{1,d+1\}})} \times \ldots \times r^{(J_{\{d,d+1\}})}}$.
 The corresponding contraction map $\rho_{r_L}$ is
 \begin{multline*}
  L_{\zeta,\alpha_1,\ldots,\alpha_d}  = \rho_{r_L}(L_1,\ldots,L_d,L_{d+1})_{\zeta,\alpha_1,\ldots,\alpha_d} \\
   = \sum_{\beta^{\{1,d+1\}} = 1}^{r^{(J_{\{1,d+1\}})}_L} \ldots \sum_{\beta^{\{d,d+1\}} = 1}^{r^{(J_{\{d,d+1\}})}_L} (L_1)_{\zeta,\alpha_1,\beta^{\{1,d+1\}}} \ldots (L_d)_{\zeta,\alpha_d,\beta^{\{d,d+1\}}} (L_{d+1})_{\{\beta^{\{\mu,d+1\}}\}_{\mu \in [d]}},
 \end{multline*}
 for $\alpha_\mu = 1,\ldots,n_\mu$, $\mu = 1,\ldots,d$ and $\zeta = 1,\ldots,\ell$ as visualized in \cref{Tucker_graphL}.
\tikzfigure{htb}{Tucker_graphL}{\label{Tucker_graphL}The contraction diagram for the Tucker-like decomposition of $L$ as in \cref{example:TuckerL} for $d = 4$ (cf. \cref{example:Tucker}).}%
\end{example}
 In general, when all summations over $\alpha_v$, $v \in [d]$, are proceeded first, then
 $\mathcal{L}(X)$ can be efficiently evaluated by means of the tree structure of (cf. \cref{brevalLN})
 \begin{align}\label{eq:LNprod}
  \mathcal{L}(X)_\zeta = \sum_{\varepsilon^e, \beta^e \,:\, e \in E}
  \prod_{v \in V} \big( \sum_{\alpha_v} (L_v)_{\zeta, \alpha_v, \{\beta^e\}_{e \in E_v}}
  \Nvtensor{v} \big), \quad \zeta \in [\ell].
 \end{align}
 Note that we have here again made use of the redundant additional indices $\zeta$ and $\alpha_v$ for $v > d$.
Likewise, the composition of $\mathcal{L}$ and $\mathcal{N}_{\neq c}$ can be proceeded efficiently.
\subsection{Equivalent low rank weights}\label{sec:lrW}
The switching between each complementary weights introduced in \cref{sec:complwe} has the following motivation. 

\begin{lemma}\label{switchlemma}
 Let $c \in V$ and $\mathcal{S}_c$ be as in \cref{Svlem}, and let $N$ be a representation 
 for which $Y^{(J)}$, $J \in \mathcal{K}^{\mathcal{S}_c}$, are orthonormal (cf. \cref{canonN}).
 Then the update $X_{\gamma,\omega,W}^{N,c}$ as defined in \cref{defXwceq} for the rank $n_J$ matrices
 $W^{(J)} = W^{(J)}_{\gamma,X} = (X^{[J]} (X^{[J]})^T + \gamma I)^{-1}$, 
 $J \in \mathcal{K}^{\mathcal{S}_c}$, $X = \tau_r(N)$, is the same as for the rank
 $r^{(J)}$ matrices
 \begin{align*}
  W^{(J)} = W^{(J)}_{\gamma,N,c} := Y^{(J)} (H^{(J)} + \gamma I)^{-1} (Y^{(J)})^T, \quad H^{(J)} := Z^{(J)} (Z^{(J)})^T, \quad J \in \mathcal{K}^{\mathcal{S}_c}.
 \end{align*}
\end{lemma}
\begin{proof}
 It suffices to show that for every $\widetilde{N}_c \in \mathfrak{H}_{\mathfrak{m}_c}$, we have
 \begin{align*}
   (W^{(J)}_{\gamma,X})^{1/2} \mathcal{N}_{\neq c}(\widetilde{N}_c)^{[J]} = (W^{(J)}_{\gamma,N,c})^{1/2} \mathcal{N}_{\neq c}(\widetilde{N}_c)^{[J]}, \quad J \in \mathcal{K}^{\mathcal{S}_c}.
  \end{align*}
 Let the orthonormal matrix $U^{J,\perp} \in \R^{n_J \times n_J - r^{(J)}}$ span the orthogonal complement
 of the $r^{(J)}$ dimensional space $\mathrm{range}(Y^{(J)})$. Then
 \begin{align*}
  (W^{(J)}_{\gamma,X})^{1/2} = (W^{(J)}_{\gamma,N,c} + \gamma^{-1} U^{J,\perp} (U^{J,\perp})^T)^{1/2} = (W^{(J)}_{\gamma,N,c})^{1/2} + \gamma^{-1/2} U^{J,\perp} (U^{J,\perp})^T.
 \end{align*}
 It thus remains to show that $\mathrm{range}(\mathcal{N}_{\neq c}(\widetilde{N}_c)^{[J]}) \perp \mathrm{range}(U^{J,\perp})$ for all $\widetilde{N}_c$.
 As by construction of $\mathcal{S}_c$, the matrix $Y^{(J)}$ does not depend on the vertex $c \in V$, we have
  \begin{align*}
  \mathrm{range}(\mathcal{N}_{\neq c}(\widetilde{N}_c)^{[J]}) \subseteq \mathrm{range}(Y^{(J)}).
 \end{align*}
\end{proof}

\subsection{Path evaluations}
\Cref{switchlemma} allows for a further, significant simplification in the evaluation of $(N_c)_{\gamma,\omega,W}^{N,c}$
as defined in \cref{updateformula}. Let in the following $c \in V$ and ${\hat{J}} \in \mathcal{K}^{\mathcal{S}_c}$ be fixed.
Further, let $v_{c,{\hat{J}}} \in V \setminus \{c\}$ be the uniquely determined vertex with
 $J_c(v_{c,{\hat{J}}}) \cap [d] = {\hat{J}}$ (cf. \cref{sec:graphtheory}),
as in \cref{rootedtrees}. 
Without explicit indication of the dependence on the above, we denote (cf. \cref{eq:pathdef,ES})
\begin{align*}
p := c\mathring{\rightharpoondown}v_{c,{\hat{J}}} \subseteq V \setminus \{c,v_{c,{\hat{J}}}\}, \quad E_Y & := \partial{E}_{\{c\} \cup p}= E_c \setminus \{e_1\} \cup \partial{E}_p.
\end{align*}
For empty $p$, we set $p_1 = v_{c,{\hat{J}}}$ and $p_{-1} = c$ for convenience.
We further define the edges $e_1 = \{ c, p_1 \} \notin E_Y$ and $\hat{e} = \{ p_{-1}, v_{c,\hat{J}} \} \in E_Y$, such that ${\hat{J}} = J_{\hat{e}}$. 
\begin{theorem}[cf. \cref{sec:thmalt}]\label{thm:patheval}
Let $c \in V$ and $\mathcal{S}_c$ be as in \cref{Svlem}, and let $N$ be a representation 
 for which $Y^{(J)}$, $J \in \mathcal{K}^{\mathcal{S}_c}$, are orthonormal (as in \cref{canonN}).
Further, let the operator
$\mathcal{N}_{\neq c}^\ast \circ \mathcal{W}^{({\hat{J}})} \circ \mathcal{N}_{\neq c}: \mathfrak{H}_{\mathfrak{m}_c} \rightarrow \mathfrak{H}_{\mathfrak{m}_c}$
be described by the matrix $A^{({\hat{J}})} \in \mathfrak{H}_{\mathfrak{m}_c} \otimes \mathfrak{H}_{\mathfrak{m}_c}$.
Then (cf. \cref{NWJN})
\begin{align} \label{AtoM}
 \AJtensor
  = \deltatensor_{\alpha_c',\alpha_c} \big( \netprod_{e \in E_c \setminus \{ e_1 \}} \deltatensor_{{\beta^e}',\beta^e}\big) \ \MJmatrix,
\end{align}
where each $\delta_{\gamma',\gamma} \in \{0,1\}$ is a Kronecker delta, as well as 
 \begin{multline} \label{theMJ}
 \MJmatrix = 
    \netsum{\begin{array}{c} 
            \scriptstyle {\beta^e}',\beta^{e} \, :\, e \in E_p \setminus \{e_1\},\\
            \scriptstyle \alpha_v\,:\, v \in p
            \end{array}} 
 \big( \netprod_{e \in \partial{E}_p \setminus \{e_1,\hat{e}\}} \deltatensor_{{\beta^{e}}',\beta^{e}} \big)
 \mlcr
 \left( \netprod_{v \in p}  \Nvprime{v} \Nvtensor{v} \right) \HJinv{\hat{J}}{\hat{e}},
\end{multline}
and further (cf. \cref{ZZT})
\begin{multline} \label{theHJ}
 \HJmatrix{\hat{J}}{\hat{e}} = \mlcr
    \netsum{\begin{array}{c} 
            \scriptstyle {\beta^e}',\beta^{e} \, :\, e \in E_{\{c\} \cup p} \setminus \{\hat{e}\},\\
            \scriptstyle \alpha_v\,:\, v \in \{c\} \cup p
            \end{array}  }
 \big( \netprod_{e \in E_Y \setminus \{\hat{e}\}} \deltatensor_{{\beta^{e}}',\beta^{e}} \big) \netprod_{v \in \{c\} \cup p} \left( \Nvprime{v} \Nvtensor{v} \right),
\end{multline}
for each $\alpha_v \in [n_v]$, $v \in V$, and ${\beta^e}',\beta^e \in [r^{(J_e)}]$, $e \in E$.
\end{theorem}
\begin{proof}
See \cref{NWJN,ZZT}. For the rigorous, though exceedingly technical proof, see \cref{technicalproof}. A more elegant version can be found in \cref{sec:thmalt}.
\end{proof}
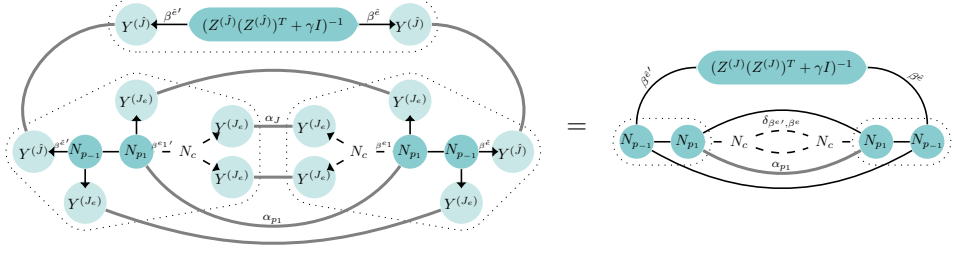
\begin{figure}[htb]
 \begin{align*}
  \subimport{tikz_base_files/HT_graph/}{HT_graph.tex} = \subimport{tikz_base_files/HT_graph_deltas/}{HT_graph_deltas.tex}
 \end{align*}
\caption{\label{NWJN}Network diagram for $A^{(\hat{J})}$ representing $\mathcal{N}_{\neq c}^\ast \circ \mathcal{W}^{({\hat{J}})} \circ \mathcal{N}_{\neq c}$ (cf. \cref{AtoM,theMJ}) for a particular case of a certain $\mathcal{K}$, network $N$ and a path $p = (p_1,p_{-1})$ of length $|p| = 2$. Contractions over labels $\alpha_S$, $S \subset [d]$, are in gray, whereas uncontracted modes are visualized via dashed lines. Here, it is $c,p_{-1} \notin [d]$, but $p_1 \in [d]$.
We recommend to view the digital version for better readability.
{\normalfont[Lefthand]} Emphasized are the segments $G^{(\hat{J})}$ ($\mathcal{N}_{\neq c}^\ast$ at south-west and $\mathcal{N}_{\neq c}$ south-east) and $W^{(\hat{J})}_{\gamma,N,c}$ (north)
as in \cref{thm:patheval}. The lighter shaded nodes are the partial contractions $Y^{(J_e)}$
for $e \in E_Y$, the orthogonality constraints of which are indicated with corresponding arrows.
{\normalfont[Righthand]} The contracted version
in which only the nodes $\{N_v\}_{v \in p}$ and their copies (as encircled) as well as the matrix $(H^{(\hat{J})} + \gamma I)^{-1} = (Z^{(\hat{J})} (Z^{(\hat{J})})^T + \gamma I)^{-1}$ remain as well as some delta tensors.}
\end{figure}
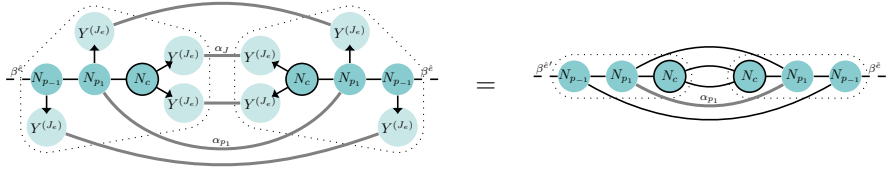
\begin{figure}[htb]
 \begin{align*}
  \subimport{tikz_base_files/HT_graph_Z/}{HT_graph_Z.tex} = \subimport{tikz_base_files/HT_graph_Z_deltas/}{HT_graph_Z_deltas.tex}
 \end{align*}
\caption{\label{ZZT}Network diagram for $H^{(\hat{J})} = Z^{(\hat{J})} (Z^{(\hat{J})})^T$ (cf. \cref{theHJ}) for the same particular case 
as in \cref{NWJN}. 
{\normalfont[Lefthand]} The lighter shaded nodes are the partial contractions $Y^{(J_e)}$
for $e \in E_Y$, the orthogonality constraints of which are indicated with corresponding arrows.
{\normalfont[Righthand]} The contracted version
in which only the nodes $\{N_v\}_{v \in \{c\} \cup p}$ and their copies (as encircled) remain.}
\end{figure}%
The formula for $A^{(\hat{J})}$ simplifies whenever $\hat{e} \in E_c$ as follows.
\begin{corollary}\label{parthle0}
 In \cref{thm:patheval}, if $\hat{e} = e_1 = \{ c,v \}$ for $v \in \mathrm{neigh}(c)$, 
 then $p = \emptyset$. Thus, we have $M^{(\hat{J})} = (H^{(\hat{J})} + \gamma I)^{-1}$ and
  \begin{align*}
 \HJmatrix{\hat{J}}{\hat{e}} 
  = 
    \netsum{\begin{array}{c} 
            \scriptstyle \beta^{e} \, :\, e \in E_c \setminus \{\hat{e}\},\\
            \scriptstyle \alpha_c
            \end{array}  }
  (N_c)_{\alpha_c,\{\beta^e\}_{e \in E_c \setminus \{\hat{e}\}},{\beta^{\hat{e}}}'} \cdot  \Nvtensor{c},
\end{align*}
for $\alpha_v \in [n_v]$, and ${\beta^e}',\beta^e \in [r^{(J_e)}]$, $e \in E_c$.
\end{corollary}

\subsection{Branch evaluations}\label{sec:brancheval}

As described in the following,
each expression in the update formula of $(N_c)_{\gamma,\omega,W}^{N,c}$ (cf. \cref{updateformula}) can be rewritten, such that reevaluations of identical terms are avoided. This is particularly
useful (\cref{brevalLN}) for the first measurement related summand $(\mathcal{L} \circ \mathcal{N}_{\neq c})^\ast \circ \mathcal{L} \circ \mathcal{N}_{\neq c}: \mathfrak{H}_{\mathfrak{m}_c} \rightarrow \mathfrak{H}_{\mathfrak{m}_c}$ and the righthand side $(\mathcal{L} \circ \mathcal{N}_{\neq c})^\ast: \R^\ell \rightarrow \mathfrak{H}_{\mathfrak{m}_c}$ since only few branch-wise evaluations
change after each micro-step during a sweep. While this is not true for 
the weight related terms or the sum of such, the computational complexity may (depending on $c \in V$, $\mathcal{K}$ and $d$) still be reduced through the recursive, branch-wise evaluation of the entire term $\sum_{J \in \mathcal{K}^{\mathcal{S}_c}} \mathcal{N}_{\neq c}^\ast \circ \mathcal{W}^{(J)} \circ \mathcal{N}_{\neq c}$ (\cref{brevalA,brevalHJ}). 
\def\Lv#1{(L_{#1})_{ \zeta,\alpha_{#1},\{\varepsilon^e\}_{e \in E_{#1}} }}%
\def\Stensor#1{S^{(J_{#1})}_{ \zeta,\beta^{#1},\varepsilon^{#1} }}%
\def\hadzeta{}%
\begin{proposition}\label{brevalLN}
Let $c \in V$ and each $J_e \in \mathcal{K}^{\mathcal{S}_c}$, $e \in E$. 
Further, let $\mathcal{L} \circ \mathcal{N}_{\neq c}: \mathfrak{H}_{\mathfrak{m}_c} \rightarrow \R^\ell$ be represented by the tensor $F_c \in \R^{[\ell]} \otimes \mathfrak{H}_{\mathfrak{m}_c}$. Then
 \begin{align*}
  (F_c)_{ \zeta,\alpha_c,\{\beta^e\}_{e \in E_c} } =
  \netsum{ \varepsilon^e \,:\, e \in E_c }\hadzeta
  \Lv{c}
  \netprod_{v \in \mathrm{neigh}(c)} 
  \Stensor{\{c,v\}},
 \end{align*}
for $\zeta \in [\ell]$ (not being contracted), as well as
\begin{align*}
 \Stensor{\hat{e}} 
 = 
 \netsum{\begin{array}{c} 
            \scriptstyle \varepsilon^{e},\beta^{e}\, :\, e \in E_v \setminus \{\hat{e}\} \\
            \scriptstyle \alpha_v
            \end{array} 
} 
 \Lv{v} \Nvtensor{v} 
 \netprod_{b \in \mathrm{desc}_c(v)} \Stensor{\{v,b\}},
\end{align*}
for $\hat{e} = \{\mathrm{pred}_c(v),v\}$, $v \in V \setminus \{c\}$ and $\zeta \in [\ell]$ (not being contracted).
\end{proposition}
\begin{proof}
 See \cref{app:sec:brancheval}.
\end{proof}
The paths appearing in the evaluation of $A^{(\hat{J})} \in \mathfrak{H}_{\mathfrak{m}_c} \otimes \mathfrak{H}_{\mathfrak{m}_c}$ representing $\mathcal{N}_{\neq c}^\ast \circ \mathcal{W}^{({\hat{J}})} \circ \mathcal{N}_{\neq c}: \mathfrak{H}_{\mathfrak{m}_c} \rightarrow \mathfrak{H}_{\mathfrak{m}_c}$, $J \in \mathcal{K}^{\mathcal{S}_c}$ (cf. \cref{thm:patheval}) naturally overlap. In the evaluation of
$A := \sum_{J \in \mathcal{K}} A^{(\hat{J})}$ (cf. \cref{updateformula}),
this can be utilized. 

\begin{proposition}\label{brevalA}
Let $c \in V$ and each $J_e \in \mathcal{K}^{\mathcal{S}_c}$, $e \in E$. It is
 \begin{align*}
  \sum_{\hat{J} \in \mathcal{K}} \AJtensor = 
  \deltatensor_{\alpha_c',\alpha_c} \sum_{v \in \mathrm{neigh}(c)} \big( \netprod_{e \in E_c \setminus \{\{c,v\}\}} \deltatensor_{{\beta^e}',\beta^e}\big) \BJ{\{c,v\}}{\{c,v\}}
 \end{align*} 
 with $B^{(J_{\hat{e}})} = (H^{(J_{\hat{e}})} + \gamma I)^{-1} +  \sum_{b \in \mathrm{desc}_c(v)}
 \widetilde{B}^{(J_{\{v,b\}})}$,
 $\hat{e} = \{\mathrm{pred}_c(v),v\}$, $v \in V \setminus \{c\}$, as well as, for $b \in \mathrm{desc}_c(v)$,
  \begin{multline*}
 \BJtilde{\{v,b\}}{\hat{e}} = \mlcr
 \netsum{\begin{array}{c} 
            \scriptstyle \beta^{e}\, :\, e \in E_v \setminus \{\hat{e}\},\\
            \scriptstyle {\beta^{\{v,b\}}}',\, \alpha_v
            \end{array} 
} 
(N_v)_{\alpha_v,{\beta^{\hat{e}}}',\{\beta^e\}_{e \in E_v \setminus \{\{v,b\},\hat{e}\}},{\beta^{\{v,b\}}}'} \cdot
\BJ{\{v,b\}}{\{v,b\}} \cdot \Nvtensor{v}.
 \end{multline*}
\end{proposition}
\begin{proof}
 See \cref{app:sec:brancheval}.
\end{proof}
Due to the recursive structure in \cref{brevalA}, the evaluation is to be proceeded in order leaves to root.
In turn, also the matrices $H^{(J_{\hat{e}})}$, $\hat{e} \in E$, (cf. \cref{theHJ}) can be simplified, but in the opposing root to leaves order. The starting points for this recursion are given by \cref{parthle0}.
\begin{proposition}\label{brevalHJ}
Let $c \in V$ and each $J_e \in \mathcal{K}^{\mathcal{S}_c}$, $e \in E$. 
For $\hat{e} = \{\mathrm{pred}_c(v),v\}$, $v \in V \setminus \{c\}$, and $b \in \mathrm{desc}_c(v)$, it is
\begin{multline*}
\HJmatrix{J_{\{v,b\}}}{\{v,b\}} = \mlcr
\netsum{\begin{array}{c} 
            \scriptstyle \beta^{e}\, :\, e \in E_v \setminus \{\{v,b\}\}\mlcr
            \scriptstyle {\beta^{\hat{e}}}',\, \alpha_v
            \end{array} 
} 
(N_v)_{\alpha_v,{\beta^{\{v,b\}}}',\{\beta^e\}_{e \in E_v \setminus \{\{v,b\},\hat{e}\}},{\beta^{\hat{e}}}'} \cdot
\HJmatrix{J_{\hat{e}}}{\hat{e}} \cdot \Nvtensor{v}.
\end{multline*}
\end{proposition}
\begin{proof}
 See \cref{app:sec:brancheval}.
\end{proof}
\section{Numerical Experiments}\label{sec:numexp}
The following \cref{sec:sols,sec:oper,sec:solmeth,sec:expsetup} specify terminology and configurations referred to in the subsequent \cref{exp00,exp0,exp0_,exp1,exp2} in \cref{aasrm,altasrm,sec:lsasrm}. The presentation of results is further laid out in \cref{sec:presres}.
For simplicity, the mode sizes $\{n_\mu\}_{\mu \in [d]}$ are chosen uniformly as $\overline{n} \in \N$ in all experiments. For the corresponding \textsc{Matlab} code, please contact the author. 

\subsection{Reference solutions, measurements vectors and family \texorpdfstring{$\mathcal{K}$}{K}}\label{sec:sols}
Each measurement vector is constructed via a (not necessarily sought for) reference solution with ranks $r_\rs \in \N^{\mathcal{K}}$, 
which in turn relies on a randomly generated representation, 
\begin{align*}
 y = \mathcal{L}(X^\rs) \in \R^\ell, \quad X^\rs  = \tau_{r_\rs}(N^\rs) \in \mathcal{L}^{-1}(y) \cap V_{\leq r_\rs}^\mathcal{K}.
\end{align*} 
All entries of the components $\{N^\rs_v\}_{v \in V^\rs}$ are assigned independent, normally distributed entries. For simplicity\footnote{Our considered \IRLS{} algorithms neither use uniform ranks nor are provided any information on $r_\rs$. For further related, extensive tests on rank adaptivity, we refer to \cite{Kr19_Sta,Kr20_Tre}.} and to limit the amount of randomness, we also choose the components $\{r_\rs^{(J)}\}_{J \in \mathcal{K}}$ uniformly, as $\overline{r}_\rs \in \N$. We distinguish between four different types.

\paragraph{Tucker format}
With $\mathcal{K} = \mathcal{K}_{\mathrm{Tucker}} = \{\{1\},\ldots,\{d\}\}$, the components of the representation $\{N^\rs\}_{v \in V^\rs}$ follow the scheme in \cref{example:Tucker}.

\paragraph{Balanced, binary hierarchical Tucker format (bbHT)}
A balanced, binary hierarchical Tucker format can be defined by the property
of $\mathcal{K} = \mathcal{K}_{\mathrm{bbHT}}$ to be exhaustive (cf. \cref{sec:exhaustive})
and to minimize the maximal distance of any two vertices $v,w \in [d]$
within $G_{\mathcal{K}}$ (that is, the depth of the rooted tree, cf. \cite{Gr10_Hie}).

\paragraph{Exponentially declining singular values}
Firstly, a bbHT representation as defined above is generated.
As second step, all singular values $\sigma^{(J)}$, $J \in \mathcal{K}$,
are manipulated such they decline exponentially.
In explicit, for a constant $s_{(\mathrm{expfac})} \in (0,1)$,
it is $\sigma^{(J)}_i \approx \max(\sigma_{\min},s_{(\mathrm{expfac})}^{x})$, $i = 1,\ldots,\overline{r}_\rs$, $J \in \mathcal{K}$,
where each $x$ is an independent random, normally distributed value
and $\sigma_{\min} > \varepsilon > 0$ (cf. \cref{sec:expsetup}) is a lower bound.
We denote such reference solutions by the abbreviated \textit{exp.dec.bbHT}. 

\paragraph{Canonical polyadic (CP) decomposition}\label{para:CP}
For $\overline{r}_\rs \in \N$, the 
reference solution does here not rely on $\mathcal{K}$, but is generated as sum of $\overline{r}_\rs$ elementary tensors,
$X^\rs_{\alpha_1,\ldots,\alpha_d} := \tau_{\overline{r}_\rs}(\phi^\rs) = \sum_{\gamma = 1}^{\overline{r}_\rs} (\phi^\rs_1)_{\alpha_1,\gamma} \ldots (\phi^\rs_d)_{\alpha_d,\gamma}$,
where $(\phi^\rs_\mu) \in \R^{[n_\mu] \times [\overline{r}_\rs]}$, for $\mu = 1,\ldots,d$.
The corresponding graph is a hypertree, and the set $\mathrm{image}(\tau_{\overline{r}_\rs})$ of at most rank $\overline{r}_\rs$ tensors is a semi-algebraic subset of $V^{\mathcal{K}_{\max}}_{\leq r_\rs}$ (cf. \cref{HTtensors})
for the in that case defined, non-hierarchical family $\mathcal{K}_{\max} := \{ J \subsetneq [d] \mid J \neq \emptyset \}$, given $r_\rs^{(J)} \equiv {\overline{r}_\rs}$, $J \in \mathcal{K}_{\max}$.

\subsection{Operators}\label{sec:oper}
We consider three types of operators $\mathcal{L}$, where in each case
$\mathcal{L}(X) := L\ \mathrm{vec}(X)$ is based on the tensor $L \in \R^{\ell \times n_1 \ldots n_d}$.
\paragraph{(Full) Gaussian operator}\label{gaussop}
With a Gaussian operator, we refer to a randomly generated tensor $L \in \R^{\ell \times n_1 \ldots n_d}$ with independent, normally distributed entries.

\paragraph{Gaussian low rank operator}\label{gaussop2}
For (low) uniform ranks $r_L^{(J)} \equiv \overline{r}_L \in \N$, $J \in \mathcal{K}$,
the operator is defined through the representation of $L := \rho_{r_L}(\{L_v\}_{v \in V^\rs})$ (cf. \cref{sec:Ldecomp}). Each component therein are assigned independent, normally distributed entries.
\paragraph{Random sampling operator}\label{samplop}
As sampling operator, we denote
  $\mathcal{L}(X) := \{ X_{p_i} \}_{i = 1}^\ell$,
 for uniformly randomly drawn indices $\{p_1,\ldots,p_\ell\} \subset \bigtimes_{\mu = 1}^d [n_\mu]$.
 Note that sampling operators can trivially be decomposed, for $r^{(J)}_L = 1$, $J \in \mathcal{K}$.

\subsection{Solution methods}\label{sec:solmeth}

Based on a sufficiently large starting value $\gamma^{(0)} > 0$, we choose $\gamma^{(i)} = \nu \gamma^{(i-1)}$, where $\nu < 1$ remains constant throughout each single run of an algorithm. 
We consider the following types of optimization.

\paragraph{Full, image based (\texorpdfstring{\IRLS{}-$0\mathcal{K}$}{IRLS-0K})}

As in \cref{IRLS-0K}, the full tensor is optimized based on the (literally interpreted)
image update formula \cref{tensorX_W} without further modification (\cref{alg:irlsmrreintws} for $\mathcal{S}_i \equiv \emptyset$, $i \in \N_0$). When instability threatens to occur, 
the equivalent kernel based update \cref{eq:kernelupdate}
for $X_0 = X^{(0)}$ is applied, with $X^{(0)}$ as in \cref{gammainf}.

\paragraph{Full, relaxed}

The relaxed constraints described in \cref{sec:relaxed} are utilized, but without subspace restrictions or weight switching (\cref{alg:relT} with $\mathcal{T}_i \equiv \mathcal{L}^{-1}(y)$, $\mathcal{S}_i \equiv \emptyset$, $i \in \N_0$).
In this case, the residual $\|\mathcal{L}(X) - y\|$ is expected to converge to $0$ parallel to the decline of $\gamma$, but this is not guaranteed.

\paragraph{Alternating (\texorpdfstring{\AIRLS{}-$0\mathcal{K}$}{AIRLS-0K})}

We apply the representation based, necessarily relaxed, alternating optimization (\cref{alg:relT} for $\mathcal{T}_i = \mathcal{T}_{c_i}(\tau_r(N^{(i)}))$ and $\mathcal{S}_i = \mathcal{S}_{c_i}$, $i \in \N_0$) further discussed in \cref{AIRLS}. 
The update formulas for the single components make use of the branch-wise evaluations as derived in \cref{sec:brancheval}. Whether the (maximal) ranks $r \in \N^\mathcal{K}$ of the iterate, that is, the sizes of $\{N_v\}_{v \in V}$, are fixed
or adapted, as well as the potential use of other heuristics laid out in \cref{sec:practheurasp}, is specified in the respective experiments.

\paragraph{Neigh}
The same algorithm as aboves \AIRLS{}-$0\mathcal{K}$ is applied, but
in each update of the node $N_c$, $c \in V$, only weights corresponding to $J_e \in \mathcal{K}^{\mathcal{S}_c}$, $e \in E_c$, are included (cf. \cref{parthle0}) in order to reduce the computational complexity. 
This reduction of paths yields the variant closest to our priorly introduced algorithm \textsc{SALSA},
and in particular the minimal number of weights in each the update of $N_c$, $c \in V$, for which the rank adaption stability property, as further introduced in \cite{Kr19_Sta}, still holds true. 
\paragraph{Plain ALS without reweighting}
In one instance in \cref{exp1}, we also compare to the plain alternating least squares (ALS) residual minimization (\cref{defXwceq} for $\omega = \gamma = 0$) for fixed ranks $r = r_\rs \in \N^\mathcal{K}$.
This algorithm is thus granted additional, in practice generally unavailable information and does not adapt ranks.
\subsection{Experimental setup and evaluation}\label{sec:expsetup}%
In order to evaluate each output $X^\alg$, we compare its \textit{non-neglectable} singular values to those of $X^\rs$.
We define
\begin{align*}
 \mathrm{det}^\mathcal{K}_{n,\gamma,\varepsilon}(X) := \prod_{J \in \mathcal{K}}
 \mathrm{det}_{n_J,\gamma,\varepsilon}(X^{[J]}), \quad X \in \R^{n_1 \times \ldots \times n_d},
\end{align*}
where the matrix version is as in \cite{Kr21_Asy} given by
\begin{align*}
 \mathrm{det}^2_{m,\gamma,\varepsilon}(A) & := \gamma^{m-\mathrm{rank}_{\varepsilon}(A)} \prod_{i = 1}^{\mathrm{rank}_{\varepsilon}(A)} (\sigma_i(A)^2 + \gamma), 
\end{align*}
for $\mathrm{rank}_{\varepsilon}(A) := \max \{i \in [m] \mid \sigma_i(A) > \epsilon \cdot \|A\|_F \}$. Therein, we choose $\epsilon := 10^{-6}$.
We firstly examine the residual norm $\|\mathcal{L}(X^\rs) - y\|_F$, secondly compare the approximate ranks, and lastly compare the products of singular values. 
The latter two aspects are reflected by the limit of the quotient
\begin{align*}
 \mathcal{Q}_\varepsilon(X^\alg,X^\rs) := 
 \lim_{\gamma \searrow 0} \frac{\mathrm{det}^\mathcal{K}_{n,\gamma,\varepsilon}(X^\alg)}{\mathrm{det}^\mathcal{K}_{n,\gamma,\varepsilon}(X^\rs)} 
  \in [0,0.98] \cup (0.98,1.005) \cup [1.005,\infty].
\end{align*}
The three intervals are related to the categorization into improvements, successes or the two types of failures as outlined below, where
the limits $0$ or $\infty$ are reached if and only if $\sum_{J \in \mathcal{K}}\mathrm{rank}_{\varepsilon}((X^\alg)^{[J]})$ and $\sum_{J \in \mathcal{K}} \mathrm{rank}_{\varepsilon}((X^\rs)^{[J]})$ differ.
\paragraph{Post iteration}\label{sec:postiter}
In order to avoid misjudgment, in cases where the tensor $X^\alg$ may be an improving solution (though that seldomly happens here),
we apply a post iteration analogous to the one discussed in the matrix case \cite{Kr21_Asy} in order to allow the parameter $\varepsilon$ to be reduced to machine precision.
\paragraph{Details of comparison}
As in \cite{Kr21_Asy}, if $\|\mathcal{L}(X^{\alg}) - y\| > 10^{-6} \|y\|$ or if for the quotient, it holds $\mathcal{Q}_\varepsilon(X^{\alg},X^\rs) = \infty$, then the result is considered a \textit{strong failure}.
If $\|\mathcal{L}(X^{\alg}) - y\| \leq 10^{-6} \|y\|$, then on the one hand we refer to $1.005 \leq \mathcal{Q}_\varepsilon(X^{\alg},X^\rs) < \infty$ as \textit{weak failure}.
On the other, for $0.98 < \mathcal{Q}_\varepsilon(X^{\alg},X^\rs) < 1.005$, we consider the result \textit{successful}, while for $\mathcal{Q}_\varepsilon(X^{\alg},X^\rs) \leq 0.98$,
 we say the result is an \textit{improvement}, subject to the consideration above.
\paragraph{Sensitivity analysis}
 With the exception of \cref{exp1}, we lower the meta parameter $\nu = \nu_k = \sqrt{\nu_{k-1}}$ (cf. \cref{sec:solmeth}), starting with $\nu_0 = 1.2$, and rerun the respective algorithm from the start until the result is not a \textit{failure}.
However, after too many reruns $k > k_{\max}$, we give up and thus either achieve a \textit{weak} or \textit{strong} \textit{failure} depending on the result for $k = k_{\max}$.
All other meta parameters for each algorithm are common to all respective experiments.
\subsection{Presentation of results}\label{sec:presres}
Each but \cref{exp00,exp1} is reflected upon in three different ways as summarized in \cref{metatable}.
\paragraph{ASRM/recovery tables}
For each instance, we list the percentual numbers of ASRM improvements, successes or fails as defined in \cref{sec:expsetup}. Successes are further distinguished regarding recoveries, whether $\|X^\alg - X^\rs\|_F \leq 10^{-4} \|X^\rs\|_F$. In near all cases where this is fulfilled, the relative residual even falls below $10^{-6}$ (see \cref{sec:visres}), in which case the algorithm stops automatically\footnote{Needless to say, this is the only point at which the reference solution itself is used within the algorithm, and only done in order to save a considerable amount of unnecessary computation time.}.
Note that both improvements as well as fails with respect to ASRM naturally nearly exclude recoveries with accuracy $10^{-4}$, and always so for $10^{-6}$.
\paragraph{ASRM/recovery figures}
More distinguished visualizations of the results underlying the above mentioned tables  can be found in \cref{sec:visres} as described therein.
\paragraph{$\gamma$-decline sensitivity}
A depiction of results regarding the sensitivity analysis outlined in \cref{sec:expsetup} is covered in \cref{sec:visres} as well.

\subsection{Observing the theoretical phase transition for generic recoveries}

\begin{experiment}\label{exp00}
 For $d = 4$, $\overline{n} = 5$ and $\overline{r}_\rs = 3$,
 we consider the \ASRMK{}-$\mathcal{K}_{\mathrm{bbHT}}$ problem based on \textit{Gaussian measurements} for
 reference solutions given via \textit{bbHT} representations for $\ell \in \{68,69,70\}$.
 The solution method in both cases utilizes \textit{full, image based} updates (cf. \cref{sec:sols}).
 Each constellation is repeated $100$ times, for a comparatively large value of $k_{\max} = 10$. The results are covered in \cref{tab00}.
\end{experiment}%
The dimension of the given bbHT variety is $\dim(V^{\mathcal{K}_{\mathrm{bbHT}}}_{\leq r_\rs}) = 4 \overline{n} \overline{r}_\rs + 2 \overline{r}_\rs^3 - 5 \overline{r}_\rs^2 = 69$ (cf. \cref{manidim}). 
The value $\ell = \dim(V_{\leq \overline{r}_\rs}^{\mathcal{K}_{\mathrm{bbHT}}}) + 1$ (which here is $\ell = 70$) in turn provides
the minimal sufficient number of \textit{generic}\footnote{To be more precise, \textit{generic} in that context is an algebraic property that is stronger than the ones that stem from analysis or probability theory, but roughly similar.} measurements (thus not including sampling) to provide $\mathcal{L}^{-1}(\mathcal{L}(X^\rs)) \cap V_{\leq \overline{r}_\rs}^{\mathcal{K}_{\mathrm{bbHT}}} = \{ X^\rs \}$ for \textit{generic} $X^\rs \in V_{\leq \overline{r}_\rs}^{\mathcal{K}_{\mathrm{bbHT}}}$, as more generally proven in \cite{BrGeMiVa21_Alg}.
We can indeed observe (see \cref{tab00}) that for $\ell = 69$, multiple solutions are found as verified through the post iteration process up to machine accuracy. For the value $\ell = 70$ in turn, no duplicate solutions seem to exist. The one improving solution
as well as the two weak failures are not the reference solution, though in fact neither within $V_{\leq r}^{\mathcal{K}_{\mathrm{bbHT}}}$ for $r = r_\rs$ but $r = \tilde{r}$, $\tilde{r}^{(\{1,2\})} = 3$, $(\tilde{r}^{(\{1\})},\ldots,\tilde{r}^{(\{4\})}) = (2,2,4,4)$. From the perspective of a dimension minimization (cf. \cref{Xastgeneral}) in turn, not even the improving result would be preferable as $\mathrm{dim}(V_{\leq \hat{r}}^{\mathcal{K}_{\mathrm{bbHT}}}) = 71$ ($V_{\leq \hat{r}}^{\mathcal{K}_{\mathrm{bbHT}}} \nsupseteq V_{\leq r_\rs}^{\mathcal{K}_{\mathrm{bbHT}}}$).
\def\externaltablecaption#1#2#3{\label{#1}table as specified in \cref{sec:presres} for \cref{#2} (see \cref{#3} for more details)}
\externaltable{htb!}{table_n2_tensor_conj_new}{\externaltablecaption{tab00}{exp00}{res00b}}%
\subsection{Affine sum-of-ranks minimization}\label{aasrm}
\begin{experiment}\label{exp0}
 For $d = 4$, $\overline{n} = 5$ and $\overline{r}_\rs = 3$,
 we consider the \ASRMK{}-$\mathcal{K}$ problem based on \textit{samples} or \textit{Gaussian measurements} and
 reference solutions given via \textit{bbHT} representations for $\ell \in \{83,111,138\}$ and  $\mathcal{K} = \mathcal{K}_{\mathrm{bbHT}}$ or by \textit{CP} decompositions for $\ell \in \{62,82,102\}$ and $\mathcal{K} = \mathcal{K}_{\max}$.
 The solution method in both cases utilizes \textit{full, image based} updates
based on the respective families $\mathcal{K}$ (cf. \cref{sec:sols}).
 Each constellation is repeated $100$ times, for $k_{\max} = 8$. The results are covered in \cref{tab0,res0,res0b}.
\end{experiment}
\externaltable{htb!}{table_7_tensor_by_setting_new}{\externaltablecaption{tab0}{exp0}{res0b}}%
The dimension or even the more particular structure of the variety $V_{\leq r_\rs}^{\mathcal{K}_{\max}}$, for $\mathcal{K}_{\max} = \{ J \subsetneq [d] \mid J \neq \emptyset \}$, $d \geq 4$, as applied in the CP case (cf. \cref{para:CP}), is unknown to the best of our knowledge. 
While real tensors of at most rank $\overline{r}_\rs$ do not form varieties, 
complex ones with at most this border rank do, here with a dimension of
$\mathrm{dim}(V_{\leq \overline{r}_\rs,\mathbb{C}}) = \overline{r}_\rs (d (\overline{n}-1) + 1) = 51$ (cf. \cite{BrCeDuHeMaSaVeYu21_Non,QiCoLi16_Sem,ChOtVa14_AnA}). Though we assume this dimension to be lower than the one for $\mathcal{K}_{\max}$, we take this smaller value as reference.
In that sense, the considered values $\ell$ are each (rounded) multiples $c_{\mathrm{mf}} \in \{1.2,1.6,2\}$.
To our surprise, if successeful, the CP reference solution is (near perfectly) recovered even for $\ell = 62$, considering that this value is smaller than $69 \equiv \mathrm{dim}(V_{\leq r_\rs}^{\mathcal{K}_{\mathrm{bbHT}}})$ for every exhaustive hierarchical family $K_{\mathrm{bbHT}}$. 
One possible explanation would be that $\mathrm{dim}(V_{\leq r_\rs}^{\mathcal{K}_{\max}})$ is lower or equal to $61$, but further investigation remains subject to future work.
While we can not, as theory provides, expect generic completions in case of sampling problems,
the failures with respect to \ASRMK{} are subject of \IRLS{}-$0\mathcal{K}$ itself.
In particular, slower rates of decline $\nu$ (cf. \cref{res0}) may be required, 
and allow for better results for both sampling and Gaussian measurements as suggested
by \cref{res00}.
Though already for $c_{\mathrm{mf}} = 1.2$, the rate of decline seems to suffice.
\subsection{Alternating, affine sum-of-ranks minimization}%
\begin{experiment}\label{exp1}
 For $d = 4$, $\overline{n} = 5$, $\overline{r}_\rs = 3$ and $\ell \in \{126,168,210\}$
 we consider the \ASRMK{}-$\mathcal{K}_{\mathrm{Tucker}}$ problem based on \textit{samples} or rank $\overline{r}_L = 1$ \textit{Gaussian measurements} for reference solutions given through \textit{Tucker} representations.
 We compare the following four solution methods: 
 \begin{enumerate}
  \item[(a)] \textit{full, image based}
  \item[(c)] \textit{alternating}, based on fixed, maximally feasible ranks $r^{(J)} = 5$, $J \in \mathcal{K}_{\mathrm{Tucker}}$ 
  \item[(d)] \textit{alternating}, with adaptive ranks $r^{(J)} \in [5]$, $J \in \mathcal{K}_{\mathrm{Tucker}}$ (cf. \cref{sec:explra})
  \item[(e)] \textit{plain, ALS without reweighting}, based on the, a-priorly provided, fixed ranks  $r^{(J)} = \overline{r}_\rs = 3$, $J \in \mathcal{K}_{\mathrm{Tucker}}$.
 \end{enumerate}
A fixed rate $\nu = 1.002^{-1}$ of decline is used (applicable to the first three methods).
 Each constellation is repeated $100$ times, for which the results are covered in \cref{tab1,res1b}.
\end{experiment}%
\externaltable{htb!}{table_8_tensor_comparison}{\externaltablecaption{tab1}{exp1}{res1b}}%
The degrees of freedom within a Tucker decompositions for $d = 4$ in this setting is $\dim(V^\mathcal{K_{\mathrm{Tucker}}}_{\leq r_\rs}) =  \overline{r}_\rs^4 + 4 \overline{n} \overline{r}_\rs - 4 \overline{r}_\rs^2 = 105$ (cf. \cref{manidim}). The number of measurements $\ell$ are (rounded) multiples $c_{\mathrm{mf}} \in \{1.2,1.6,2\}$ of such.
As in \cref{exp0_}, there is nearly no difference between the version using fixed ranks or adaptive ranks, but both instances are slightly worse than the full version using unrelaxed constraints (note that here, these methods use the same, fixed rate of decay $\nu = 1.002^{-1}$). Plain alternating least squares on the other hand (even though only in that case, the ranks of the reference solution are provided) is significantly worse than the other methods, also for larger numbers of measurements. 
\subsection{Large scale, alternating \ASRMK{}}\label{sec:lsasrm}
\begin{experiment}\label{exp2}
For $d = 8$, $\overline{n} = 20$, $\overline{r}_\rs = 5$ and $\ell \in \{6500,13000,19500,26000\}$, we consider the \ASRMK{}-$\mathcal{K}_{\mathrm{bbHT}}$ problem based on \textit{samples}, rank $\overline{r}_L = 1$ or rank $\overline{r}_L = 2$ \textit{Gaussian measurements} for
reference solutions given via \textit{bbHT} representations with \textit{exponentially declining} singular values, $s_{(\mathrm{expfac})} = \frac{1}{3}$. For Gaussian measurements, we also consider unmodified singular values.
As solution method, we apply \textit{alternating} optimization with explicit rank adaption (limited only by $r^{(J)} \leq 8$, $J \in \mathcal{K}_{\mathrm{bbHT}}$) as
well as the applicable heuristics laid out in \cref{sec:practheurasp}. The maximal length of paths is either unrestricted, or limited to \textit{neighbors}.
Each constellation is repeated $100$ times, for $k_{\max} = 5$. The results are covered in \cref{tab2,tab2c,res2,res2b,res2c,res2d}.
\end{experiment}%
\externaltable{htb!}{table_9_alternating_tensor_by_setting}{\externaltablecaption{tab2}{exp2}{res2b}}%
\externaltable{htb!}{table_9b_alternating_tensor_by_setting_expdec}{\externaltablecaption{tab2c}{exp2}{res2d}}%
The degrees of freedom within $8$-dimensional bbHT decompositions in this setting is $\dim(V^{\mathcal{K}_{\mathrm{bbHT}}}_{\leq r_\rs}) = 8 \overline{n} \overline{r}_\rs + 6 \overline{r}_\rs^3 - 13 \overline{r}_\rs^2 = 1225$ (cf. \cref{manidim}),
while $\ell = 6500$ constitutes a fraction of about $2.5 \cdot 10^{-7}$ of the total size $\overline{n}^d = 2.56 \cdot 10^{10}$ of the tensor. 
Due to the long runtime for values $k > 5$, it yet remains speculation whether the restriction of paths to neighboring nodes does result in a loss of approximation quality or, as in other cases, rather a need for a lower parameter $\nu$ (cf. \cref{res2,res2c}). The same might hold true for the completion problem considered here.
Rank $\overline{r}_L = 2$ Gaussian operators seem in fact to generate easier problems than rank $\overline{r}_L = 1$ ones, at least judging from the given results. 
On the other hand, it becomes clear that exponentially decaying singular values pose significantly easier problems.

%% file: tikz_base_files/HT_graph/HT_graph.tex
\begin{tikzpicture}[network,scale=0.75]

  \pgfdeclarelayer{bg}    
  \pgfsetlayers{bg,main}  
  
  \node[tensornode,draw,color=white,label={center:$N_c$}] (Nc) at (0,0) {};
  \node[tensornode,label={center:$N_{p_1}$}] (P1) at (1.2,0) {};
  \node[tensornode,label={center:$N_{p_{-1}}$}] (P2) at (2.4,0) {};
  \node[partialcontraction,label={center:$Y^{(\hat{J})}$}] (QJ) at (3.6,0) {};
  
  \node[partialcontraction,label={center:$Y^{(J_e)}$}] (Q1) at ($(Nc) + (150:1.2)$) {};
  \node[partialcontraction,label={center:$Y^{(J_e)}$}] (Q2) at ($(Nc) + (210:1.2)$) {};
  
  \node[partialcontraction,label={center:$Y^{(J_e)}$}] (Q3) at ($(P1) + (90:1.2)$) {};
  \node[partialcontraction,label={center:$Y^{(J_e)}$}] (Q4) at ($(P2) + (-90:1.2)$) {};
    
  \path (Nc) edge[dashed] node[modename] {$\scriptscriptstyle \beta^{e_1}$} (P1) 
  (P1) edge (P2) 
  (P2) edge[orthoto] node[modename] {$\scriptscriptstyle \beta^{\hat{e}}$} (QJ); 
  \path (P1) edge[orthoto] (Q3);
  \path (P2) edge[orthoto] (Q4);
  \path[dashed] (Nc) edge[orthoto] (Q1);
  \path[dashed] (Nc) edge[orthoto] (Q2);

  \begin{scope}[shift={(-4,0)}]
    \node[tensornode,draw,color=white,label={center:$N_c$}] (Nc_mir) at (0,0) {};
  \node[tensornode,label={center:$N_{p_1}$}] (P1_mir) at (-1.2,0) {};
  \node[tensornode,label={center:$N_{p_{-1}}$}] (P2_mir) at (-2.4,0) {};
  \node[partialcontraction,label={center:$Y^{(\hat{J})}$}] (QJ_mir) at (-3.6,0) {};
  
  \node[partialcontraction,label={center:$Y^{(J_e)}$}] (Q1_mir) at ($(Nc_mir) + (30:1.2)$) {};
  \node[partialcontraction,label={center:$Y^{(J_e)}$}] (Q2_mir) at ($(Nc_mir) + (-30:1.2)$) {};
  
  \node[partialcontraction,label={center:$Y^{(J_e)}$}] (Q3_mir) at ($(P1_mir) + (90:1.2)$) {};
  \node[partialcontraction,label={center:$Y^{(J_e)}$}] (Q4_mir) at ($(P2_mir) + (-90:1.2)$) {};

  \path (Nc_mir) edge[dashed] node[modename] {$\scriptscriptstyle {\beta^{e_1}}'$} (P1_mir) (P1_mir) edge (P2_mir) (P2_mir) edge[orthoto] node[modename] {$\scriptscriptstyle {\beta^{\hat{e}}}'$} (QJ_mir); 
  \path (P1_mir) edge[orthoto] (Q3_mir);
  \path (P2_mir) edge[orthoto] (Q4_mir);
  \path[dashed] (Nc_mir) edge[orthoto] (Q1_mir);
  \path[dashed] (Nc_mir) edge[orthoto] (Q2_mir);
  \end{scope}
  
  \path[alphacontr] (Q1) edge node[modename] {\color{black} $\alpha_{J}$} (Q1_mir);
  \path[alphacontr] (Q2) edge (Q2_mir);
  \draw[alphacontr] (Q3_mir) to[out=20,in=160] (Q3);
  \draw[alphacontr] (Q4_mir) to[out=-20,in=-160] (Q4);
  \draw[alphacontr] (P1_mir) to[out=-55,in=-125] node[modename] {\color{black} $\alpha_{p_1}$} (P1);

  \begin{scope}[shift={(-2,3)}]
     \node[verywidetensornode,label={center:$(Z^{(\hat{J})} (Z^{(\hat{J})})^T + \gamma I)^{-1}$}] (ZJ) at (0,0) {};
     \node[partialcontraction,label={center:$Y^{(\hat{J})}$}] (QJW_mir) at (-3.2,0) {};
     \node[partialcontraction,label={center:$Y^{(\hat{J})}$}] (QJW) at (3.2,0) {};
   
   \path (ZJ) edge[orthoto] node[modename] {${\beta^{\hat{e}}}'$} (QJW_mir);
   \path (ZJ) edge[orthoto] node[modename] {$\beta^{\hat{e}}$}  (QJW);
  \end{scope}

  \def\al{2mm}
  
  \draw[alphacontr] (QJ) to[out=70,in=0] (QJW);
  \draw[alphacontr] (QJ_mir) to[out=110,in=180] (QJW_mir);

    \begin{pgfonlayer}{bg}    

    \draw[dotted] \convexpath{Q1,Q3,QJ,Q4,Q2}{0.65cm};
    \draw[dotted] \convexpath{QJ_mir,Q3_mir,Q1_mir,Q2_mir,Q4_mir}{0.65cm};
    \draw[dotted] \convexpath{QJW_mir,ZJ,QJW}{0.65cm};

  \end{pgfonlayer}

\end{tikzpicture}

%% file: tikz_base_files/HT_graph_deltas/HT_graph_deltas.tex
\begin{tikzpicture}[network,scale=0.75]

  \pgfdeclarelayer{bg}    
  \pgfsetlayers{bg,main}  
  
  \node[tensornode,draw,color=white,label={center:$N_c$}] (Nc) at (0,0) {};
  \node[tensornode,label={center:$N_{p_1}$}] (P1) at (1.2,0) {};
  \node[tensornode,label={center:$N_{p_{-1}}$}] (P2) at (2.4,0) {};
    
  \path (Nc) edge[dashed] (P1) (P1) edge (P2);

  \begin{scope}[shift={(-2,0)}]
    \node[tensornode,draw,color=white,label={center:$N_c$}] (Nc_mir) at (0,0) {};
  \node[tensornode,label={center:$N_{p_1}$}] (P1_mir) at (-1.2,0) {};
  \node[tensornode,label={center:$N_{p_{-1}}$}] (P2_mir) at (-2.4,0) {};
  
  \path (Nc_mir) edge[dashed] (P1_mir) (P1_mir) edge (P2_mir); 
  \end{scope}
  
  \draw[dashed] (Nc_mir) to[out=20,in=160] node[modename] {$\delta_{{\beta^e}',\beta^e}$} (Nc);
  \draw[dashed] (Nc_mir) to[out=-20,in=-160] (Nc);
  \draw (P1_mir) to[out=30,in=150] (P1);
  \draw (P2_mir) to[out=-30,in=-150] (P2);
  
  \draw[alphacontr] (P1_mir) to[out=-30,in=-150] node[modename] {\color{black} $\alpha_{p_1}$} (P1);

  \begin{scope}[shift={(-1,1.75)}]
     \node[verywidetensornode,label={center:$(Z^{(J)} (Z^{(J)})^T + \gamma I)^{-1}$}] (ZJ) at (0,0) {};
  \end{scope}

  \def\al{2mm}
  
  \draw (P2) to[out=90,in=0] node[modename] {$\beta^{\hat{e}}$} (ZJ);
  \draw (P2_mir) to[out=90,in=180] node[modename] {${\beta^{\hat{e}}}'$} (ZJ);

    \begin{pgfonlayer}{bg}    
 \draw[dotted] \convexpath{P1,P2}{0.55cm};
\draw[dotted] \convexpath{P1_mir,P2_mir}{0.55cm};

  \end{pgfonlayer}

\end{tikzpicture}

%% file: tikz_base_files/HT_graph_Z/HT_graph_Z.tex
\begin{tikzpicture}[network,scale=0.7]

  \pgfdeclarelayer{bg}    
  \pgfsetlayers{bg,main}  
  
  \node[tensornode,draw,label={center:$N_c$}] (Nc) at (0,0) {};
  \node[tensornode,label={center:$N_{p_1}$}] (P1) at (1.2,0) {};
  \node[tensornode,label={center:$N_{p_{-1}}$}] (P2) at (2.4,0) {};
  \node[virtnode] (QJ) at (3.6,0) {};
  
  \node[partialcontraction,label={center:$Y^{(J_e)}$}] (Q1) at ($(Nc) + (150:1.2)$) {};
  \node[partialcontraction,label={center:$Y^{(J_e)}$}] (Q2) at ($(Nc) + (210:1.2)$) {};
  
  \node[partialcontraction,label={center:$Y^{(J_e)}$}] (Q3) at ($(P1) + (90:1.2)$) {};
  \node[partialcontraction,label={center:$Y^{(J_e)}$}] (Q4) at ($(P2) + (-90:1.2)$) {};
    
  \path (Nc) edge (P1) (P1) edge (P2) (P2) edge[dashed] node[modename] {$\beta^{\hat{e}}$} (QJ); 
  \path (P1) edge[orthoto] (Q3);
  \path (P2) edge[orthoto] (Q4);
  \path (Nc) edge[orthoto] (Q1);
  \path (Nc) edge[orthoto] (Q2);

  \begin{scope}[shift={(-4,0)}]
    \node[tensornode,draw,label={center:$N_c$}] (Nc_mir) at (0,0) {};
  \node[tensornode,label={center:$N_{p_1}$}] (P1_mir) at (-1.2,0) {};
  \node[tensornode,label={center:$N_{p_{-1}}$}] (P2_mir) at (-2.4,0) {};
  \node[virtnode] (QJ_mir) at (-3.6,0) {};
  
  \node[partialcontraction,label={center:$Y^{(J_e)}$}] (Q1_mir) at ($(Nc_mir) + (30:1.2)$) {};
  \node[partialcontraction,label={center:$Y^{(J_e)}$}] (Q2_mir) at ($(Nc_mir) + (-30:1.2)$) {};
  
  \node[partialcontraction,label={center:$Y^{(J_e)}$}] (Q3_mir) at ($(P1_mir) + (90:1.2)$) {};
  \node[partialcontraction,label={center:$Y^{(J_e)}$}] (Q4_mir) at ($(P2_mir) + (-90:1.2)$) {};

  \path (Nc_mir) edge (P1_mir) (P1_mir) edge (P2_mir) (P2_mir) edge[dashed] node[modename] {$\beta^{\hat{e}}$}  (QJ_mir); 
  \path (P1_mir) edge[orthoto] (Q3_mir);
  \path (P2_mir) edge[orthoto] (Q4_mir);
  \path (Nc_mir) edge[orthoto] (Q1_mir);
  \path (Nc_mir) edge[orthoto] (Q2_mir);
  \end{scope}
  
  \path[alphacontr] (Q1) edge node[modename] {\color{black} $\alpha_{J}$} (Q1_mir);
  \path[alphacontr] (Q2) edge (Q2_mir);
  \draw[alphacontr] (Q3_mir) to[out=20,in=160] (Q3);
  \draw[alphacontr] (Q4_mir) to[out=-20,in=-160] (Q4);
  \draw[alphacontr] (P1_mir) to[out=-55,in=-125] node[modename] {\color{black} $\alpha_{p_1}$} (P1);
  
    \begin{pgfonlayer}{bg}    
    \draw[dotted] \convexpath{Q1,Q3,P2,Q4,Q2}{0.65cm};
    \draw[dotted] \convexpath{P2_mir,Q3_mir,Q1_mir,Q2_mir,Q4_mir}{0.65cm};
  \end{pgfonlayer}

\end{tikzpicture}

%% file: tikz_base_files/HT_graph_Z_deltas/HT_graph_Z_deltas.tex
\begin{tikzpicture}[network,scale=0.7]

  \pgfdeclarelayer{bg}    
  \pgfsetlayers{bg,main}  
  
  \node[tensornode,draw,label={center:$N_c$}] (Nc) at (0,0) {};
  \node[tensornode,label={center:$N_{p_1}$}] (P1) at (1.2,0) {};
  \node[tensornode,label={center:$N_{p_{-1}}$}] (P2) at (2.4,0) {};
  \node[virtnode] (QJ) at (3.6,0) {};
    
  \path (Nc) edge (P1) (P1) edge (P2) (P2) edge[dashed] node[modename] {$\beta^{\hat{e}}$}  (QJ);  
  
  \begin{scope}[shift={(-2,0)}]
    \node[tensornode,draw,label={center:$N_c$}] (Nc_mir) at (0,0) {};
  \node[tensornode,label={center:$N_{p_1}$}] (P1_mir) at (-1.2,0) {};
  \node[tensornode,label={center:$N_{p_{-1}}$}] (P2_mir) at (-2.4,0) {};
  \node[virtnode] (QJ_mir) at (-3.6,0) {};
  
  \path (Nc_mir) edge (P1_mir) (P1_mir) edge (P2_mir) (P2_mir) edge[dashed] node[modename] {${\beta^{\hat{e}}}'$}  (QJ_mir); 
  \end{scope}
  
    \draw (Nc_mir) to[out=20,in=160] (Nc);
  \draw (Nc_mir) to[out=-20,in=-160] (Nc);
  \draw (P1_mir) to[out=30,in=150] (P1);
  \draw (P2_mir) to[out=-30,in=-150] (P2);
  \draw[alphacontr] (P1_mir) to[out=-30,in=-150] node[modename] {\color{black} $\alpha_{p_1}$} (P1);
  
    \begin{pgfonlayer}{bg}    
 \draw[dotted] \convexpath{Nc,P2}{0.55cm};
\draw[dotted] \convexpath{Nc_mir,P2_mir}{0.55cm};

  \end{pgfonlayer}

\end{tikzpicture}

%% file: IRLS_tensor_conclusions.tex
We have shown that despite subtle differences, the overall structure of the log-det approach towards ARM can be generalized to the ASRM tensor setting. The global convergence of minimizers of the log-det sum-of-ranks function can likewise be concluded via the priorly applied nested minimization scheme. Even subject to the additionally considered switching between complementary subsets in $\mathcal{K}$, the \IRLS{}-$0\mathcal{K}$ algorithm inherits analogous local convergence properties, in particular with respect to the decline of the regularization parameter $\gamma \searrow 0$.
Thereafter, we have laid out that despite the relaxation of the affine constraint, as well as the iterative restriction to admissible subspaces, \IRLS{}-$0\mathcal{K}$ remains faithful to a monotone minimization of the corresponding objective function. In particular, these modifications allow a tree tensor network based, alternating evaluation \AIRLS{}-$0\mathcal{K}$, with a non-exponential, low computational complexity based on branch-wise evaluations. 
In numerical experiments, we have demonstrated that it can also practically suffice if only the number of Gaussian measurements exceeds the dimension of the lowest rank variety, the reference solution truth is contained in, by one. Further, we have shown that \AIRLS{}-$0\mathcal{K}$ is only marginally less successful than its non-alternating version \IRLS{}-$0\mathcal{K}$, while cleary superior towards ordinary, unregularized ALS. In moderately large cases, we could observe that $1.2$ times the minimally necessary number of measurement in near all cases suffices to recover the reference solution. For large scale problems, it may yet show that a slower decline of $\gamma$ could allow to further reduce the number of required measurements.

%% file: IRLS_tensor_appendix.tex
\FloatBarrier
\section{(Remaining proof of \texorpdfstring{\cref{declinelemmatensor}}{theorem})} \label{sec:minorproofs}
\begin{proof} (of \cref{declinelemmatensor}) 
Throughout the proof, we abbreviate $W^\mathcal{K} := \{W^{(J)}\}_{J \in \mathcal{K}}$ as well as the iterates $W^{(i)} := \{W^{(i,J)}\}_{J \in \mathcal{K}^{\mathcal{S}_i}}$.\\
 $(i)$: Let $\Delta^{(i)} = \gamma^{(i)} ( \sum_{J \in \mathcal{K}^{\mathcal{S}}} n_J - \sum_{J \in \mathcal{K}^{\mathcal{S}_i}} n_J)$.
 Independent of $\mathcal{S} \subset \mathcal{K}$, we have
 \begin{align*}
 f^{\mathcal{K}^\mathcal{S}}_{\gamma^{(i)}}(X^{(i)}) & \overset{(a)}{=} f^{\mathcal{K}^{\mathcal{S}_i}}_{\gamma^{(i)}}(X^{(i)}) + \Delta^{(i)} 
  \overset{(b)}{=} J^{\mathcal{K}^{\mathcal{S}_i}}_{\gamma^{(i)}}(X^{(i)},W^{(i)}) + \Delta^{(i)} \\
 & \overset{(c)}{\geq} J^{\mathcal{K}^{\mathcal{S}_i}}_{\gamma^{(i)}}(X^{(i+1)},W^{(i)}) + \Delta^{(i)} 
  \overset{(d)}{\geq} J^{\mathcal{K}^{\mathcal{S}_i}}_{\gamma^{(i)}}(X^{(i+1)},W^{\mathcal{K}^{\mathcal{S}_i}}_{\gamma^{(i)},X^{(i+1)}}) + \Delta^{(i)} \\
 & \overset{(e)}{=} f^{\mathcal{K}^{\mathcal{S}_i}}_{\gamma^{(i)}}(X^{(i+1)}) + \Delta^{(i)} \overset{(f)}{=} f^{\mathcal{K}^\mathcal{S}}_{\gamma^{(i)}}(X^{(i+1)}) \overset{(g)}{\geq} f^{\mathcal{K}^\mathcal{S}}_{\gamma^{(i+1)}}(X^{(i+1)}).
 \end{align*}
The steps $(a)$ to $(g)$ are provided by: $(a)$ \cref{sec:complwe}, $(b)$ \cref{returntoffct}, $(c)$ $X^{(i+1)} = X^{\mathcal{K}^{\mathcal{S}_i}}_{W^{(i)}}$
is optimum in $X$ \cref{tensorX_W}, $(d)$ $W^{\mathcal{K}^{\mathcal{S}_i}}_{\gamma^{(i)},X^{(i+1)}}$ is the respective optimum in $W$ \cref{tensorupdateW},
$(e)$ \cref{returntoffct}, $(f)$ \cref{sec:complwe}, $(g)$ $\frac{\partial}{\partial \gamma} f^{\mathcal{K}^\mathcal{S}}_\gamma(X) \geq 0$, $\mathcal{S} \subset \mathcal{K}$, for all $X$.\\
 $(ii)$: Since (cf. \cref{detexptensor})
    $|\mathcal{K}| \gamma^{(\sum_{J \in \mathcal{K}} n_J)  - 1} \|X\|_F^{2} \leq \prod_{J \in \mathcal{K}} \prod_{i=1}^{n_J} (\sigma_i^{(J)}(X)^2 + \gamma)  \leq  \exp(f^{\mathcal{K}}_{\gamma}(X))$,
 it follows due to $(i)$ that 
  $|\mathcal{K}| \|X^{(i)}\|_F^{2} \leq (\gamma^{(i)})^{1 - \sum_{J \in \mathcal{K}} n_J} \exp(f^\mathcal{K}_{\gamma^{(1)}}(X^{(1)}))$.
 As $\gamma^{(i)}$ does not converge to zero, the sequence $X^{(i)}$ remains bounded.\\
 $(iii/1)$: 
For $\mathcal{S} = \mathcal{S}_i$ (and thus $\Delta^{(i)} = 0$), the steps $(d)$ to $(g)$ in $(i)$
provide that $J^{\mathcal{K}^{\mathcal{S}_i}}_{\gamma^{(i)}}(X^{(i+1)},W^{(i)}) \geq f^{\mathcal{K}^{\mathcal{S}_i}}_{\gamma^{(i+1)}}(X^{(i+1)})$.
 With $\widehat{\mathcal{W}}^{(i)}$ as defined in \cref{widehatdef}, we then have
 \begin{align*}
   f^{\mathcal{K}^{\mathcal{S}_i}}_{\gamma^{(i)}}(X^{(i)}) - f^{\mathcal{K}^{\mathcal{S}_i}}_{\gamma^{(i+1)}}(X^{(i+1)}) 
  \geq &\ J^{\mathcal{K}^{\mathcal{S}_i}}_{\gamma^{(i)}}(X^{(i)},W^{(i)}) - J^{\mathcal{K}^{\mathcal{S}_i}}_{\gamma^{(i)}}(X^{(i+1)},W^{(i)}) \\
  = &\ \langle X^{(i)}, \widehat{\mathcal{W}}^{(i)}(X^{(i)}) \rangle - \langle X^{(i+1)}, \widehat{\mathcal{W}}^{(i)}(X^{(i+1)}) \rangle \\
  = &\ \langle X^{(i)} - X^{(i+1)}, \widehat{\mathcal{W}}^{(i)}(X^{(i)} + X^{(i+1)}) \rangle.
 \end{align*}
  As $ \widehat{\mathcal{W}}^{(i)}(X^{(i+1)}) \perp X^{(i)} - X^{(i+1)} \in \mathrm{kernel}(\mathcal{L})$ (as provided by \cref{tensorkernelL}) we have
   \begin{align*}
  \langle X^{(i)} - X^{(i+1)}, \widehat{\mathcal{W}}^{(i)}(X^{(i)} + X^{(i+1)}) \rangle & = \langle X^{(i)} - X^{(i+1)}, \widehat{\mathcal{W}}^{(i)}(X^{(i)} - X^{(i+1)}) \rangle \\
  & \geq \| (X^{(i)} - X^{(i+1)}) \|^2_F\  \lambda_{\min}({\widehat{\mathcal{W}}^{(i)}}).
 \end{align*}
 Since $\mathcal{W}^{(i,J)} \succ 0$, $J \in \mathcal{K}^{\mathcal{S}_i}$, the eigenvalue can be bounded via
 \begin{align*}
  \lambda_{\min}({\widehat{\mathcal{W}}^{(i)}}) & = \lambda_{\min}(\sum_{J \in \mathcal{K}^{\mathcal{S}_i}} \mathcal{W}^{(i,J)}) \geq
  \sum_{J \in \mathcal{K}^{\mathcal{S}_i}} \lambda_{\min}(W^{(i,J)}) \\
  & =  \sum_{J \in \mathcal{K}^{\mathcal{S}_i}} \lambda_{\min}(({X^{(i)}}^{[J]} ({X^{(i)}}^{[J]})^T  + \gamma I)^{-1}) \\ 
  & = \sum_{J \in \mathcal{K}^{\mathcal{S}_i}} (\sigma_1^{(J)}({X^{(i)}})^2 + \gamma)^{-1} \geq |\mathcal{K}| \ (\|X^{(i)}\|_F^2 + \gamma)^{-1}.
 \end{align*}
 Thereby, as $\|X\|_F^2$ remains bounded due to $\gamma^\ast > 0$ and $(ii)$, there exists $c > 0$ such that
  $\| (X^{(i)} - X^{(i+1)}) \|^2_F\  \lambda_{\min}({\widehat{\mathcal{W}}^{(i)}}) \geq c \| (X^{(i)} - X^{(i+1)}) \|^2_F$.
 Summing over all $i = 1,\ldots,N$, we obtain
  \begin{align*}
  c \sum_{i = 1}^N \| (X^{(i)} - X^{(i+1)}) \|^2_F & \leq \sum_{i = 1}^N f^{\mathcal{K}^{S_i}}_{\gamma^{(i)}}(X^{(i)}) - f^{\mathcal{K}^{S_i}}_{\gamma^{(i+1)}}(X^{(i+1)}) \\
   \leq \sum_{S \subset \mathcal{K}} \sum_{i = 1}^N f^{\mathcal{K}^{\mathcal{S}}}_{\gamma^{(i)}}(X^{(i)}) - f^{\mathcal{K}^{\mathcal{S}}}_{\gamma^{(i+1)}}(X^{(i+1)}) 
   & = \sum_{S \subset \mathcal{K}} f^{\mathcal{K}^{\mathcal{S}}}_{\gamma^{(1)}}(X^{(1)}) - f^{\mathcal{K}^{\mathcal{S}}}_{\gamma^{(N+1)}}(X^{(N+1)}). 
 \end{align*}
 As for each $S \subset \mathcal{K}$, $f^{\mathcal{K}^{\mathcal{S}}}_{\gamma^{(i)}}(X^{(i)})$ remains bounded, this implies $\| (X^{(i)} - X^{(i+1)}) \|^2_F \rightarrow 0$ for $i \rightarrow \infty$.\\
$(iii/2)$: 
This part is largely independent of choices of $S \subset \mathcal{K}$ since the stationary points
of all $f_\gamma^{\mathcal{K}^\mathcal{S}}$ are equal (cf. \cref{sec:complwe}). 
Let $X^{(i_\ell)}$ be a convergent subsequence of $X^{(i)}$ with limit point $X^\ast$.
In light of \cref{stabilizertensor}, it suffices to show that $X^\ast = X^{\mathcal{K}^\mathcal{S}}_{W^\ast}$ for $W^{(\ast,J)} = W^{(J)}_{\gamma^\ast,X^\ast}$, $J \in {\mathcal{K}^\mathcal{S}}$
for one $\mathcal{S} \subset \mathcal{K}$.
Due to $(iii/1)$ so far, we have $\lim_{\ell \rightarrow \infty} X^{(i_\ell+1)} = X^\ast$. As $W^{(J)}_{\gamma,X}$, $J \in [d]$, depend continuously on $X$ and $\gamma > 0$, 
it follows that
\begin{align*}
 W^{(i_\ell,J)} = W^{(J)}_{\gamma^{(i_\ell,J)},X^{(i_\ell)}} \rightarrow_{\ell \rightarrow \infty} W^{(J)}_{\gamma^\ast,X^\ast} =: W^{(\ast,J)}.
\end{align*}
Let now $\mathcal{S}$ be one of the sets that appear infinitely often in $\{\mathcal{S}_i\}_{i \in \N_0}$ with respect to a subsubsequence $\{i_{\ell_k}\}_{k \in \N}$, $S_{i_{\ell_k}} = \mathcal{S}$, $k \in \N$. Then as $X^{\mathcal{K}^\mathcal{S}}_W$ depends continuously on $W^{(J)}$, $J \in \mathcal{K}^\mathcal{S}$, the last remaining step is shown by
\begin{align*}
 X^\ast \leftarrow_{k \rightarrow \infty} X^{(i_{\ell_k}+1)} = X^{\mathcal{K}^{\mathcal{S}}}_{W^{(i_{\ell_k})}} \rightarrow_{k \rightarrow \infty} X^{\mathcal{K}^{\mathcal{S}}}_{W^\ast}
\end{align*}
$(iv)$: This part is word for word the same as in \cite{Kr21_Asy}.
\end{proof}

%% file: IRLS_tensor_acknowledgments.tex
The author would like to thank Maren Klever and Lars Grasedyck for fruitful discussions,
as well as Paul Breiding and Nick Vannieuwenhoven for conversations on generic recoverability within varities.

%% file: IRLS_tensor_suppl_content.tex
\section{Alternating ASRM (further experiment)}\label{altasrm}
\begin{experiment}\label{exp0_}
For $d = 4$, $\overline{n} = 5$, $\overline{r}_\rs = 3$ and $\ell \in \{69,83,111,138\}$, we consider the \ASRMK{}-$\mathcal{K}_{\mathrm{bbHT}}$ problem based on \textit{samples} for reference solutions given via \textit{bbHT} representations. We use the following four solution methods:
\begin{enumerate}
 \item[(a)] \textit{full, image based} (as already considered in \cref{exp0})
 \item[(b)] \textit{full, relaxed}
 \item[(c)] \textit{alternating}, based on fixed ranks $r^{(J)} = 5$, $J \in \mathcal{K}_{\mathrm{bbHT}}$ 
 \item[(d)] \textit{alternating}, with adaptive ranks $r^{(J)} \in [5]$, $J \in \mathcal{K}_{\mathrm{bbHT}}$ (cf. \cref{sec:explra})
\end{enumerate}
 Each constellation is repeated $100$ times, for $k_{\max} = 8$. The results are covered in \cref{tab0_,res0_,res0b_}.
\end{experiment}%
\def\externaltablecaption#1#2#3{\label{#1}table as specified in \cref{sec:presres} for \cref{#2} (see \cref{#3} for more details)}
\externaltable{htb!}{table_10_HT_completion_by_method}{\externaltablecaption{tab0_}{exp0_}{res0b_}}
There does not seem to be a relevant difference between full image based or relaxed optimization.
Further, only for $\ell = 83$ alternating optimization performs slightly worse for. The explicit adaption of the rank in turn likewise yields no notable difference. The quality of approximation is thus seemingly only reduced (and only slightly so) through the change to an alternating optimization. However, this effect might go stronger with increased dimensions $d$.

\newpage
\section{Visualization of numerical results}\label{sec:visres}
Each of the following even and odd numbered pair of pages contains two related visualizations of the results of one of \cref{exp0,exp0_,exp1,exp2} as summarized in \cref{metatable}. These additional visualizations are constructed as described further below.
\begin{table}[H]
 \begin{center}
  \begin{tabular}{c||c|c||c}
   experiment & $\gamma$-sensitivity & ASRM/recovery & -- table\\
   \hline\hline
      \cref{exp00} & \cref{res00} & \cref{res00b} & \cref{tab00} \\
   \hline
   \cref{exp0} & \cref{res0} & \cref{res0b} & \cref{tab0} \\
   \hline
   \cref{exp0_} & \cref{res0_} & \cref{res0b_} & \cref{tab0_} \\
   \hline   
   \cref{exp1} & $(\nu = 1.002^{-1})$ & \cref{res1b} & \cref{tab1} \\
   \hline
   \cref{exp2} & \cref{res2} & \cref{res2b} & \cref{tab2} \\
   \hline
   \cref{exp2} $(s_{(\mathrm{expfac})} = \frac{1}{3})$ & \cref{res2c} & \cref{res2d} & \cref{tab2c} \\   
  \end{tabular}
 \end{center}
\caption{\label{metatable}overview over experiments, related figures and tables}
\end{table}%
\paragraph{\texorpdfstring{$\gamma$}{gamma}-decline sensitivity}\label{par:sensitivity}
To each single trial that did not yield a failure, 
we assign the one index $k$ for which the parameter $\nu = \nu_k$ first led to a successful 
or improving run as described in \cref{sec:expsetup}. The frequencies of these indices 
as well as fails are then plotted as bars, where improvements are plotted below the x-axis.
\paragraph{ASRM/recovery figures}\label{par:recfigures}
We display the following points as \textit{button plot} (as defined below).
Given the $i$-th result $X^\alg$ 
as well as reference solution $X^\rs$, 
the x-value of the $i$-th point is given by the bounded quotient 
\begin{align*}
 x_i = \max(0.9,\min(\mathcal{Q}_\varepsilon(X^\alg,X^\rs),1.05)),
\end{align*}
Each y-value is given by 
\begin{align*}
 y_i = \min(\|X^\alg - X^\rs\|_F/\|X^\rs\|_F,1),
\end{align*}
Note that the algorithm stops automatically if that value falls below $10^{-6}$.
\paragraph{button plot}
With a button plot (with logarithmic scale in $y$), we refer to a two dimensional, 
clustered scatter plot. Therein, any circular markers with centers $(x_i,y_i)$ 
and areas $s_i$, $i = 1,\ldots,k$, that would (visually) overlap,
are recursively combined to each one larger circle $(\widehat{x},\widehat{y})$ 
with area $\widehat{s}$ according to the appropriately weighted means
\begin{align*}
\widehat{x} = \sum_{i = 1}^k \frac{s_i}{\widehat{s}} x_i,\quad \widehat{y} = \prod_{i = 1}^k y_i^{s_i/\widehat{s}},\quad \widehat{s} = \sum_{i = 1}^k s_i. 
\end{align*}
The centers of all resulting circles are indicated as crosses. Thus, if only one circle remains, then the position of that cross is given by the arithmetic mean of all initial x-coordinates and the geometric mean of all initial y-coordinates. If no disks are combined, then their centers are the initial coordinates and their areas are all equal.

\newpage
\section{Sensitivity and ASRM/recovery figures}\label{sec:sens}\mbox{}
\FloatBarrier
\def\barfigurecaption#1#2{\label{#1}Results for \cref{#2} as described in \cref{par:sensitivity}.}
\def\buttonfigurecaption#1#2{\label{#1}Results for \cref{#2} as described in \cref{par:recfigures}.}

\pdffigure{H}{bar_n2_tensor_conj_new}{\barfigurecaption{res00}{exp00}}
\pdffigure{htb!}{button_n2_tensor_conj_new}{\buttonfigurecaption{res00b}{exp00}}
\newpage\mbox{}
\pdffigure{H}{bar_7_tensor_by_setting_new}{\barfigurecaption{res0}{exp0}}
\pdffigure{htb!}{button_7_tensor_by_setting_new}{\buttonfigurecaption{res0b}{exp0}} 
\newpage\mbox{}
\pdffigure{htb!}{bar_10_HT_completion_by_method}{\barfigurecaption{res0_}{exp0_}}
\pdffigure{htb!}{button_10_HT_completion_by_method}{\buttonfigurecaption{res0b_}{exp0_}}

\pdffigure{htb!}{button_8_tensor_comparisons}{\buttonfigurecaption{res1b}{exp1}}

\pdffigure{htb!}{bar_9_alternating_tensor_by_setting}{\barfigurecaption{res2}{exp2}}
\pdffigure{htb!}{button_9_alternating_tensor_by_setting}{\buttonfigurecaption{res2b}{exp2}}

\pdffigure{htb!}{bar_9b_alternating_tensor_by_setting_expdec}{\barfigurecaption{res2c}{exp2}}
\pdffigure{htb!}{button_9b_alternating_tensor_by_setting_expdec}{\buttonfigurecaption{res2d}{exp2}}

\FloatBarrier

\section{Proof of \texorpdfstring{\cref{thm:patheval}}{theorem}}\label{technicalproof}

\def\mlcr{\\}
\def\deltatensor{\delta}
\def\AJtensor{A^{({\hat{J}})}_{\alpha_c',\{{\beta^e}'\}_{e \in E_c};\, \alpha_c,\{\beta^e\}_{e \in E_c}}}
\def\MJmatrix{M^{(\hat{J})}_{{\beta^{e_1}}',\beta^{e_1}}}
\def\Nvprime#1{(N_{#1})_{\alpha_{#1}, \{{\beta^e}'\}_{e \in E_{#1}}}}
\def\Nvtensor#1{(N_{#1})_{\alpha_{#1}, \{{\beta^e}\}_{e \in E_{#1}}}}
\def\HJmatrix#1#2{H^{(#1)}_{{\beta^{#2}}',\beta^{#2}}}
\def\HJinv#1#2{(H^{(#1)} + \gamma I)^{-1}_{{\beta^{#2}}',\beta^{#2}}}
\def\GJprime{G^{({\hat{J}})}_{\{\alpha_{J_e}\}_{e \in E_Y},\alpha_c',\{{\beta^e}'\}_{e \in E_c}}}
\def\GJtensor{G^{({\hat{J}})}_{\{\alpha_{J_e}\}_{e \in E_Y},\alpha_c,\{\beta^e\}_{e \in E_c}}}
\def\WJmatrix{W^{({\hat{J}})}_{\alpha_{{\hat{J}}}',\alpha_{{\hat{J}}}}}
\def\YJmatrix#1{Y^{(J_{#1})}_{\alpha_{J_{#1}}, \beta^{#1}}}
\def\PJprime{P^{(\hat{J})}_{\{\alpha_v\}_{v \in p},\{{\beta^{e}}'\}_{e \in \partial{E}_p}}}
\def\PJtensor{P^{(\hat{J})}_{\{\alpha_v\}_{v \in p},\{{\beta^{e}}\}_{e \in \partial{E}_p}}}
\def\PcJprime{P^{(+c,\hat{J})}_{\alpha_p,\{{\beta^e}'\}_{e \in E_Y}}}
\def\PcJtensor{P^{(+c,\hat{J})}_{\alpha_p,\{\beta^e\}_{e \in E_Y}}}
\def\BJ#1#2{B^{(J_{#1})}_{{\beta^{#2}}',\beta^{#2}}}
\def\BJtilde#1#2{\widetilde{B}^{(J_{#1})}_{{\beta^{#2}}',\beta^{#2}}}
\def\netprod{\prod}
\def\netsum#1{\sum_{#1}}

Following is the proof of \cref{thm:patheval} in minimally deviating notation. For the more elegant version, see \cref{sec:thmalt}.

\begin{proof}
 Firstly, we consider a $|E_Y| + |E_c|$ dimensional tensor representation $G^{(\hat{J})} \in \R^{\bigtimes_{e \in E_Y} [n_{J_e}]} \otimes \mathfrak{H}_{\mathfrak{m}_c}$ of $\mathcal{N}_{\neq c}: \mathfrak{H}_{\mathfrak{m}_c} \rightarrow \R^{n_1 \times \ldots \times n_d}$ and its adjoint. We can thus write
\begin{multline*}
 \AJtensor
 \mlcr = \netsum{\alpha_{J_e}\,:\, e  \in E_Y} 
 \GJprime
 \WJmatrix
 \GJtensor.
\end{multline*}
The representation $G^{(\hat{J})}$ can further be decomposed into a set of orthonormal matrices $Y^{(J_e)} \in \R^{[n_{J_e}] \times [r^{(J_e)}]}$, $e \in E_Y$, and the tensor $P^{(\hat{J})} \in \R^{ \bigtimes_{v \in p} [n_v] \times \bigtimes_{{e} \in \partial{E}_p} [r^{(J_e)}]}$ obtained via a contraction along the path $p$,

\begin{align*}
 \GJtensor
  = \netsum{\beta^{e}\, :\, {e} \in E_Y \setminus E_c} \netprod_{e \in E_Y}  \YJmatrix{e}
\ P^{(\hat{J})}_{\{\alpha_v\}_{v \in p},\{{\beta^{e}}\}_{e \in \partial{E}_p}}.
\end{align*}
whereas the path evaluation is given by
\begin{align*}
P^{(\hat{J})}_{\{\alpha_v\}_{v \in p},\{{\beta^{e}}\}_{e \in \partial{E}_p}}
= \tau_r(\{N_v\}_{v \in p})
 = \netsum{\beta^{{e}} \, :\, {e} \in \mathring{E}_{p}} \netprod_{v \in p} \Nvtensor{v}.
\end{align*}
As \cref{switchlemma} provides, we may replace $W^{({\hat{J}})} = W^{({\hat{J}})}_{\gamma,N,c}$.
The matrices $Y^{(J)}$, $J \in \mathcal{K}^{\mathcal{S}_c}$, then cancel out due to orthonormality and we obtain \cref{AtoM}
for
\begin{multline*}
\MJmatrix = \mlcr
 \netsum{\begin{array}{c} 
            \scriptstyle {\beta^{e}}',\beta^{e}\, :\, {e} \in E_Y \setminus E_c,\\
            \scriptstyle \alpha_v\,:\, v \in p
            \end{array} 
} 
\big( \netprod_{e \in \partial{E}_p \setminus \{e_1,\hat{e}\}} \deltatensor_{{\beta^{e}}',\beta^{e}} \big)
 \PJprime
 \HJmatrix{\hat{J}}{\hat{e}}
 \PJtensor
 \end{multline*}
As the term $H^{(J)}$ can similarly be simplified, we have
\begin{align*}
 \HJmatrix{\hat{J}}{\hat{e}}
 = 
            \netsum{\begin{array}{c} 
            \scriptstyle {\beta^{e}}',\beta^{e}\, :\, {e} \in E_Y \setminus \{\hat{e}\},\\
            \scriptstyle \alpha_v\,:\, v \in p
            \end{array} 
} 
\big( \netprod_{e \in E_Y \setminus \{\hat{e}\}} \deltatensor_{{\beta^{e}}',\beta^{e}} \big)
\PcJprime \PcJtensor,
 \end{align*}
 where
 \begin{align*}
 \PcJtensor
= \netsum{\beta^{e_1}} \Nvtensor{c} \PJtensor.
  \end{align*}
  By expanding and reordering the contractions within the path evaluations, we then arrive at \cref{theMJ,theHJ}.
\end{proof}
%
\def\deltatensor{\Delta}
\def\mlcr{}

\def\AJtensor{A^{({\hat{J}})}}
\def\MJmatrix{M^{(\hat{J})}}
\def\Nvprime#1{N_{#1}'}
\def\Nvtensor#1{N_{#1}}

\def\HJmatrix#1#2{H^{(#1)}}
\def\HJinv#1#2{(H^{(#1)} + \gamma I)^{-1}}

\def\GJprime{{G^{({\hat{J}})}}'}
\def\GJtensor{G^{({\hat{J}})}}
\def\WJmatrix{W^{({\hat{J}})}}
\def\YJmatrix#1{Y^{(J_{#1})}}
\def\PJprime{{P^{(\hat{J})}}'}
\def\PJtensor{P^{(\hat{J})}}
\def\PcJprime{{P^{(+c,\hat{J})}}'}
\def\PcJtensor{P^{(+c,\hat{J})}}

\def\BJ#1#2{B^{(J_{#1})}}
\def\BJtilde#1#2{\widetilde{B}^{(J_{#1})}}

\def\netprod{\mathop{\boxtimes}}
\def\netsum#1{\mathop{\boxtimes}}

\newpage
\section{Tensor nodes}\label{app:sec:reloptHT}
As indicated in \cref{sec:notdev}, we in the following dismiss the indices in tensor contractions. What is here introduced as notation, is a simplified version
of the formal arithmetic established in \cite{Kr20_Tre}.
\subsection{Self-emergent contractions}
Though it is clear by \cref{sec:reloptHT}, which tensors are assigned which labels, we here repeat this formal step. Such is indicated by writing
$X = X(\{\alpha_\mu\}_{\mu \in [d]}) \in \mathfrak{H}_{\{\alpha_\mu\}_{\mu \in [d]}}$
for any full tensor, or in case of its representation network $X = \tau_r(N)$, by
\begin{align*}
 N = \{N_v\}_{v \in V}, \quad N_v = N_v(\{\gamma\}_{\gamma \in \mathfrak{m}_v}) \in \mathfrak{H}_{\mathfrak{m}_v}.
\end{align*}
Avoiding the redundant notation such as in the expression \cref{eq:taur}, we simply write
\begin{align*}
 X = \netprod_{v \in V} \Nvtensor{v}.
\end{align*}
The same symbol is used for any other contraction, such as in \cref{eq:partcont}, translating to
\begin{align*}
 \tau_r(\{N_s\}_{s \in S})_{\{\alpha_s\}_{s \in S}} = \netprod_{v \in S} \Nvtensor{s}.
\end{align*}
For any label $\gamma$, we denote the priorly used Kronecker deltas as formal objects $\Delta_{\gamma',\gamma} \in \mathfrak{H}_{\{\gamma',\gamma\}}$, or equivalently so for more than two labels. 
Instead of explicitly denoting primed labels, we instead define
\begin{align*}
(N_v)_{\gamma \rightarrow \gamma'} := \Delta_{\gamma',\gamma} \netprod N_v, \quad \gamma \in \mathfrak{m}_v.
\end{align*}
As shorthand notation, we further define 
$N_v' = N_v'(\alpha_v,\{{\beta^e}'\}_{e \in E_v})$
as
 \begin{align*}
   N_v' := (N_v)_{\{{\beta^e}\}_{e \in E_v} \rightarrow \{{\beta^e}'\}_{e \in E_v}} := \big( \netprod_{e \in E_v} \Delta_{{\beta^e}',\beta^e} \big) \netprod N_v.
 \end{align*}
For other tensors, the operator $(\cdot)'$ likewise
denotes a priming of all labels $\{\beta^e\}_{e \in E}$ assigned to such.
The special case of an element-wise multiplication as in \cref{LXtensor} is 
flagged via a superindex
\begin{align*}
 L_c \boxtimes^\zeta L_v := \Delta_{\zeta'',\zeta',\zeta} \boxtimes ( \Delta_{\zeta',\zeta} \boxtimes L_c ) \boxtimes ( \Delta_{\zeta'',\zeta} \boxtimes L_v ).
\end{align*}
We may thus equivalently write \cref{LXtensor} as
\begin{align*}
 \mathcal{L}(X) = L \netprod X \in \mathfrak{H}_\zeta, \quad  L = \rho_{r_L}(\{L_v\}_{v \in V}) = \boxtimes^\zeta_{v \in V} L_v \in \mathfrak{H}_{\zeta \cup \{\alpha_\mu\}_{\mu \in [d]}}.
\end{align*}
The expression \cref{eq:LNprod} for instance takes the shorter shape
$\mathcal{L}(X) =  \netprod^\zeta_{v \in V} ( L_v \netprod N_v )$.

 \subsection{Alternative \texorpdfstring{\cref{thm:patheval,parthle0}}{theorems}}\label{sec:thmalt}
While we may write
 $A^{(\hat{J})} = A^{(\hat{J})}(\alpha_c',\{{\beta^e}'\}_{e \in E_c},\alpha_c,\{\beta^e\}_{e \in E_c})$, the identities in \cref{thm:patheval} become
 \begin{align} \label{app:AtoM}
 \AJtensor
  = \deltatensor_{\alpha_c',\alpha_c} \netprod \big( \netprod_{e \in E_c \setminus \{ e_1 \}} \deltatensor_{{\beta^e}',\beta^e}\big) \netprod \MJmatrix,
\end{align}
for $\MJmatrix = \MJmatrix({\beta^{e_1}}',\beta^{e_1})$ with
 \begin{align} \label{app:theMJ}
 \MJmatrix & = 
 \big( \netprod_{e \in \partial{E}_p \setminus \{e_1,\hat{e}\}} \deltatensor_{{\beta^{e}}',\beta^{e}} \big) \netprod
 \Big( \netprod_{v \in p}  \Nvprime{v} \netprod \Nvtensor{v} \Big) \netprod \HJinv{\hat{J}}{\hat{e}},
\end{align}
as well as
$\HJmatrix{\hat{J}}{\hat{e}} = \HJmatrix{\hat{J}}{\hat{e}}({\beta^{\hat{e}}}',\beta^{\hat{e}})$ given by
\begin{align} \label{app:theHJ}
 \HJmatrix{\hat{J}}{\hat{e}} = 
 \big( \netprod_{e \in E_Y \setminus \{\hat{e}\}} \deltatensor_{{\beta^{e}}',\beta^{e}} \big) \netprod_{v \in \{c\} \cup p} \Big( \Nvprime{v} \netprod \Nvtensor{v} \Big).
\end{align}
The identities appearing in the proof of \cref{thm:patheval}, in turn, 
become
\begin{align*}
 \AJtensor =  \GJprime \netprod
 \WJmatrix \netprod
 \GJtensor, \quad \GJtensor
  =  \YJmatrix{e}
\netprod P^{(\hat{J})}, \quad P^{(\hat{J})}
 = \netprod_{v \in p} \Nvtensor{v}
\end{align*}
These tensors thereby have labels $G^{(\hat{J})} = G^{(\hat{J})}(\{\alpha_{J_e}\}_{e \in E_Y},\alpha_c,\{\beta^e\}_{e \in E_c})$, as well as $Y^{(J_e)} = Y^{(J_e)}(\alpha_{J_e},\beta^e)$, $e \in E_Y$, and $P^{(\hat{J})} = P^{(\hat{J})}(\{\alpha_v\}_{v \in p \cap [d]}, \{\beta^e\}_{{e} \in \partial{E}_p})$.
Further,
\begin{align*}
\MJmatrix = 
\PJtensor_{\beta^{e_1} \rightarrow {\beta^{e_1}}',\ \beta^{\hat{e}} \rightarrow {\beta^{\hat{e}}}'} \netprod
 \HJmatrix{\hat{J}}{\hat{e}} \netprod
 \PJtensor
 \end{align*}
and similarly
\begin{align*}
 \HJmatrix{\hat{J}}{\hat{e}}
= \PcJtensor_{\beta^{\hat{e}} \rightarrow {\beta^{\hat{e}}}'} \netprod \PcJtensor, \quad \PcJtensor = \Nvtensor{c} \netprod \PJtensor
 \end{align*}
The identity in \cref{parthle0} on the other hand is simply
$
 \HJmatrix{\hat{J}}{\hat{e}} = (N_c)_{\beta^{\hat{e}} \rightarrow {\beta^{\hat{e}}}'} \netprod N_c
$.
\subsection{Proofs of \texorpdfstring{\cref{brevalLN,brevalA,brevalHJ}}{Propositions}}\label{app:sec:brancheval}
\def\Lv#1{L_{#1}}
\def\Stensor#1{S^{(J_{#1})}}
\def\hadzeta{{}^\zeta}
The recursion stated in \cref{brevalLN} is
 \begin{align*}
  F_c =
  \Lv{c}
  \boxtimes^\zeta_{v \in \mathrm{neigh}(c)} 
  \Stensor{\{c,v\}},
 \end{align*}
for $c \in V$ and
\begin{align*}
 \Stensor{\hat{e}} 
 = 
 (\Lv{v} \netprod \Nvtensor{v}) \boxtimes_{b \in \mathrm{desc}_c(v)}^\zeta
 \Stensor{\{v,b\}},
\end{align*}
with $\hat{e} = \{\mathrm{pred}_c(v),v\}$ for $v \in V \setminus \{c\}$.
\begin{proof}
 The recursion implies that
 $\Stensor{\hat{e}} = \netprod^\zeta_{h \in \mathrm{branch}_c(v)} (L_h \netprod N_h)$.
 Thereby,
 \begin{align*}
  F_c & := L_c \netprod^\zeta_{v \in V \setminus \{c\}} (L_v \netprod N_v)
   = L_c \netprod^\zeta_{v \in \mathrm{neigh}(c)} \netprod^\zeta_{h \in \mathrm{branch}_c(v)} (L_h \netprod N_h)
 \end{align*}
 provides the to be shown, first identity.
\end{proof}%
The recursion in \cref{brevalA} is
 \begin{align*}
  \sum_{\hat{J} \in \mathcal{K}}  \AJtensor = 
  \deltatensor_{\alpha_c',\alpha_c} \netprod \sum_{v \in \mathrm{neigh}(c)} \big( \netprod_{e \in E_c \setminus \{\{c,v\}\}} \deltatensor_{{\beta^e}',\beta^e}\big) \netprod \BJ{\{c,v\}}{\{c,v\}},
 \end{align*} 
 where for 
 $\hat{e} = \{p_{-1},v\}$, $p_{-1} = \mathrm{pred}_c(v)$, $v \in V \setminus \{c\}$, it is
  \begin{align*}
 \BJtilde{\{v,b\}}{\hat{e}} = 
(N_v)_{\beta^{\hat{e}} \rightarrow {\beta^{\hat{e}}}',\ \beta^{\{v,b\}} \rightarrow {\beta^{\{v,b\}}}'} \netprod
\BJ{\{v,b\}}{\{v,b\}} \netprod \Nvtensor{v},
 \end{align*}
for $b \in \mathrm{desc}_c(v)$.
\begin{proof}
By definition,
  \begin{align*}
 \sum_{\hat{J} \in \mathcal{K}} \AJtensor
  & = \sum_{v \in V \setminus \{c\}}  \deltatensor_{\alpha_c',\alpha_c} \netprod  \big( \netprod_{e \in E_c \setminus \{ \{c,p_1(v)\} \}} \deltatensor_{{\beta^e}',\beta^e}\big) \netprod M^{(J_{\{\mathrm{pred}_c(v),v\}})} \\
  & = \deltatensor_{\alpha_c',\alpha_c} \netprod \sum_{v \in \mathrm{neigh}(c)} \big( \netprod_{e \in E_c \setminus \{\{c,v\}\}} \deltatensor_{{\beta^e}',\beta^e}\big) \netprod
  \sum_{b \in \mathrm{branch}_c(v)} M^{(J_{\{\mathrm{pred}_c(b),b\}})}.
\end{align*}
Let each be $p = (c,h)$, with $e_1 = \{c,p_1\}$, $\hat{e} = \{p_{-1},h\}$, for $h \in V \setminus \{c\}$. We show that 
\begin{align*}
 \sum_{b \in \mathrm{branch}_c(h)} M^{(J_{\{\mathrm{pred}_c(b),b\}})}
 & = \big( \netprod_{e \in \partial{E}_{p} \setminus \{e_1,\hat{e}\}} \deltatensor_{{\beta^{e}}',\beta^{e}} \big) \netprod
 \Big( \netprod_{v \in {p}}  \Nvprime{v} \netprod \Nvtensor{v} \Big) 
 \netprod B^{(J_{\hat{e}})}
\end{align*}
by induction over the cardinality of $\mathrm{branch}_c(h)$. The induction start for a cardinality of $1$ is then given by \cref{app:theMJ} and the definition of $B^{(J)}$. In turn, given the tree structure of $G$ and the induction hypothesis, it follows that
\begin{align*}
 & \ \sum_{b \in \mathrm{branch}_c(h)} M^{(J_{\{\mathrm{pred}_c(b),b\}})}
= M^{(J_{\hat{e}})} + \sum_{b \in \mathrm{desc}_c(h)} \sum_{w \in \mathrm{branch}_c(b)}
 M^{(J_{\{\mathrm{pred}_c(w),w\}})} 
 \\
= & \ \big( \netprod_{e \in \partial{E}_{p} \setminus \{e_1,\hat{e}\}} \deltatensor_{{\beta^{e}}',\beta^{e}} \big) \netprod
 \Big( \netprod_{v \in {p}}  \Nvprime{v} \netprod \Nvtensor{v} \Big) \netprod (H^{(J_{\hat{e}})} + \gamma I)^{-1} 
 \\
+ & \ \sum_{b \in \mathrm{desc}_c(h)} 
\big( \netprod_{e \in \partial{E}_{(c,b)} \setminus \{e_1,\{h,b\}\}} \deltatensor_{{\beta^{e}}',\beta^{e}} \big) \netprod
 \Big( \netprod_{v \in {(c,b)}}  \Nvprime{v} \netprod \Nvtensor{v} \Big) 
 \netprod B^{(J_{\{h,b\}})}
 \\
 = & \ \big( \netprod_{e \in \partial{E}_{p} \setminus \{e_1,\hat{e}\}} \deltatensor_{{\beta^{e}}',\beta^{e}} \big) \netprod
 \Big( \netprod_{v \in {p}}  \Nvprime{v} \netprod \Nvtensor{v} \Big) \netprod \Big( (H^{(J_{\hat{e}})} + \gamma I)^{-1} 
 \\
+ & \ \sum_{b \in \mathrm{desc}_c(h)} \big( \netprod_{e \in E_h \setminus \{\hat{e},\{h,b\}\}} \Delta_{{\beta^e}',\beta^e} \big) \netprod
N_h' \netprod
\BJ{\{h,b\}}{\{v,b\}} \netprod \Nvtensor{h} \Big) 
\\
= & \ \big( \netprod_{e \in \partial{E}_{p} \setminus \{e_1,\hat{e}\}} \deltatensor_{{\beta^{e}}',\beta^{e}} \big) \netprod
 \Big( \netprod_{v \in {p}}  \Nvprime{v} \netprod \Nvtensor{v} \Big) \netprod \Big( (H^{(J_{\hat{e}})} + \gamma I)^{-1} 
 \\
+ & \ \sum_{b \in \mathrm{desc}_c(h)} 
(N_h)_{\beta^{\hat{e}} \rightarrow {\beta^{\hat{e}}}',\ \beta^{\{h,b\}} \rightarrow {\beta^{\{h,b\}}}'} \netprod
\BJ{\{h,b\}}{\{v,b\}} \netprod \Nvtensor{h} \Big).
\end{align*}
The last to be shown step follows as by definition
\begin{align*}
 B^{(J_{\hat{e}})} = (H^{(J_{\hat{e}})} + \gamma I)^{-1} +  \sum_{b \in \mathrm{desc}_c(h)}
  (N_h)_{\beta^{\hat{e}} \rightarrow {\beta^{\hat{e}}}',\ \beta^{\{h,b\}} \rightarrow {\beta^{\{h,b\}}}'} \netprod
\BJ{\{h,b\}}{\{h,b\}} \netprod \Nvtensor{h}.
\end{align*}
\end{proof}
The recursion in \cref{brevalHJ}, for $\hat{e} = \{p_{-1},v\}$, $p_{-1} = \mathrm{pred}_c(v)$, $v \in V \setminus \{c\}$, is
\begin{align*}
\HJmatrix{J_{\{v,b\}}}{\{v,b\}} = 
(N_v)_{\beta^{\{v,b\}} \rightarrow {\beta^{\{v,b\}}}',\ \beta^{\hat{e}} \rightarrow {\beta^{\hat{e}}}'}
\netprod
\HJmatrix{J_{\hat{e}}}{\hat{e}} \netprod \Nvtensor{v},
\end{align*}
for $b \in \mathrm{desc}_c(v)$.
\begin{proof}
Using the identity \cref{app:theHJ} on both sides, we obtain
\begin{align*}
&\ (N_v)_{\beta^{\{v,b\}} \rightarrow {\beta^{\{v,b\}}}',\ \beta^{\hat{e}} \rightarrow {\beta^{\hat{e}}}'}
\netprod
\HJmatrix{J_{\hat{e}}}{\hat{e}} \netprod \Nvtensor{v} 
\\
=&\ 
(N_v)_{\beta^{\{v,b\}} \rightarrow {\beta^{\{v,b\}}}',\ \beta^{\hat{e}} \rightarrow {\beta^{\hat{e}}}'}
\netprod
 \big( \netprod_{e \in \partial{E}_{\{c\} \cup p} \setminus \{\hat{e}\}} \deltatensor_{{\beta^{e}}',\beta^{e}} \big) \netprod_{w \in \{c\} \cup p} \Big( \Nvprime{w} \netprod \Nvtensor{w} \Big)
 \netprod \Nvtensor{v}
 \\
 = &\ 
 \big( \netprod_{e \in E_v \setminus \{\hat{e},\{v,b\}\}} \Delta_{{\beta^e}',\beta^e} \big) \netprod N_v'
\netprod
 \big( \netprod_{e \in \partial{E}_{\{c\} \cup p} \setminus \{\hat{e}\}} \deltatensor_{{\beta^{e}}',\beta^{e}} \big) \netprod_{w \in \{c\} \cup p} \Big( \Nvprime{w} \netprod \Nvtensor{w} \Big)
 \netprod \Nvtensor{v} \\
 =&\ 
 \big( \netprod_{e \in \partial{E}_{\{c\} \cup (c,b)} \setminus \{\{v,b\}\}} \deltatensor_{{\beta^{e}}',\beta^{e}} \big) \netprod_{w \in \{c\} \cup (c,b)} \Big( \Nvprime{w} \netprod \Nvtensor{w} \Big) = \HJmatrix{J_{\{v,b\}}}{\{v,b\}},
\end{align*}
which was to be shown.
\end{proof}

\newpage
\subsection{Detailed \texorpdfstring{\AIRLS{}-$0K$}{AIRLS-0K} algorithm}
\Cref{AIRLS-0K} summarizes the \AIRLS{}-$0\mathcal{K}$ method as covered in \cref{AIRLS}.
In our experiments, we have chosen
the therein appearing constant as $c_{\mathcal{L}} = \frac{1}{4} |\mathcal{K}|^{-1} \|L\|_F^2 / n_{[d]}$.
The heuristics laid out in \cref{sec:practheurasp} are marked as \textit{possibly} applicable statements.\\

 \begin{breakablealgorithm}
  \caption{Detailed \AIRLS{}-$0\mathcal{K}$ method}
  \begin{algorithmic}[1]\label{AIRLS-0K}
  \STATE{derive tree $G = (V,E)$ from $\mathcal{K}$}
  \STATE{set $N =\{N\}_{v \in V} \in \mathcal{D}_r$, $\gamma^{(0)} > 0$, $c_0 \in V$}
  \STATE{\textit{possibly} introduce validation set (cf. \cref{sec:valiset})}
   \item[]
   \item[] \COMMENTx{let $J_e \in \mathcal{K}^{\mathcal{S}_{c_{0}}}$, $e \in E$}

   \item[] \COMMENTx{orthonormalize $\{N_v\}_{v \in V}$ with respect to $c_0$ (cf. \cref{canonN}) and initialize the branch evaluations $\{S^{(J_e)}\}_{e \in E}$ (cf. \cref{brevalLN})}
  \STATE{$\widetilde{V} \sgets \mathrm{leaves}(c_{0})$}
  \WHILE{$\widetilde{V} \neq \{c_0\}$ }  
  \FOR{$v \in \widetilde{V} \setminus \{c_0\}$}
  \STATE{$p_{-1} \sgets \mathrm{pred}_{c_0}(v)$, $\hat{e} \sgets \{p_{-1},v\}$} 
  \STATE{$QR \sgets N_v^{[\mathfrak{m}_v \setminus \{\beta^{\hat{e}}\}]}$, $N_v^{[\mathfrak{m}_v \setminus \{\beta^{\hat{e}}\}]} \sgets Q$, $N_{p_{-1}}^{[\beta^{\hat{e}}]} \sgets R N_{p_{-1}}^{[\beta^{\hat{e}}]}$} 
  \STATE{$S^{(J_{\hat{e}})} \sgets (L_v \boxtimes N_v) \boxtimes^\zeta_{h \in \mathrm{desc}_{c_0}(v)} S^{(J_{\{v,h\}})}$}
  \ENDFOR{} 
      \STATE{$\widetilde{V} \sgets \bigcup_{v \in \widetilde{V}} \mathrm{pred}_{c_0}(v)$}
  \ENDWHILE 
  \item[] 
     \item[] \COMMENTx{implicitly declare iterate $X^{(0)} = \tau_r(\{N_v\}_{v \in V})$}
  \FOR{$i = 1,2,\ldots$}
   \item[] \COMMENTx{calculate $\{H^{(i-1,J)}\}_{J \in \mathcal{K}^{\mathcal{S}_{c_{i-1}}}}$ (cf. \cref{brevalHJ,parthle0}) or \textit{possibly} limit set via maximal distance of ${\widetilde{v}} \in \widetilde{V}$ to $c_{i-1}$:}
  \FOR{$v \in \mathrm{neigh}(c_{i-1})$}
    \STATE{$\hat{e} \sgets \{c_{i-1},v\}$}
    \STATE{$\HJmatrix{J_{\hat{e}}}{\hat{e}} \sgets (N_{c_{i-1}})_{\beta^{\hat{e}} \rightarrow {\beta^{\hat{e}}}'} \netprod N_{c_{i-1}}$}
  \ENDFOR
    
   \STATE{$\widetilde{V} \sgets \bigcup_{v \in \mathrm{neigh}(c_{i-1})} \mathrm{desc}_{c_{i-1}}(v)$}
  \WHILE{$\widetilde{V} \neq \mathrm{leaves}(c_{0})$ } 
\FOR{${\widetilde{v}} \in \widetilde{V} \setminus \mathrm{lea{\widetilde{v}}es}(c_{0})$}
\STATE{$p_{-1} \sgets \mathrm{pred}_{c_0}({\widetilde{v}})$, $\hat{e} \sgets \{p_{-1},{\widetilde{v}}\}$} 
    \FOR{$b \in \mathrm{desc}_{c_{i-1}}({\widetilde{v}})$}
        \STATE{$\HJmatrix{J_{\{{\widetilde{v}},b\}}}{\{{\widetilde{v}},b\}} \sgets 
(N_{\widetilde{v}})_{\beta^{\{{\widetilde{v}},b\}} \rightarrow {\beta^{\{{\widetilde{v}},b\}}}',\beta^{\hat{e}} \rightarrow {\beta^{\hat{e}}}'}
\netprod
\HJmatrix{J_{\hat{e}}}{\hat{e}} \netprod \Nvtensor{{\widetilde{v}}}$}
    \ENDFOR
\ENDFOR
  \STATE{$\widetilde{V} \sgets \bigcup_{{\widetilde{v}} \in \widetilde{V}} \mathrm{desc}_{c_{i-1}}({\widetilde{v}})$}
\ENDWHILE 
\item[]
\item[] \COMMENTx{calculate $\{B^{(i-1,J)}\}_{J \in \mathcal{K}^{\mathcal{S}_{c_{i-1}}}}$ (cf. \cref{brevalA}) or \textit{possibly} limit set via maximal distance of ${\widetilde{v}} \in \widetilde{V}$ to $c_{i-1}$:}
  \STATE{$\widetilde{V} \sgets \mathrm{leaves}(c_{0})$}
  \WHILE{$\widetilde{V} \neq \{c_0\}$ }  
  \FOR{${\widetilde{v}} \in \widetilde{V} \setminus \{c_0\}$}
  \STATE{$p_{-1} \sgets \mathrm{pred}_{c_0}({\widetilde{v}})$, $\hat{e} \sgets \{p_{-1},{\widetilde{v}}\}$} 
\FOR{$b \in \mathrm{desc}_{c_{i-1}}({\widetilde{v}})$}
  \STATE{$\widetilde{B}^{(i-1,J_{\{{\widetilde{v}},b\}})} \sgets (N_h)_{\beta^{\hat{e}} \rightarrow {\beta^{\hat{e}}}', \beta^{\{{\widetilde{v}},b\}} \rightarrow {\beta^{\{{\widetilde{v}},b\}}}'} \netprod
B^{(i-1,J_{\{{\widetilde{v}},b\}})}  \netprod \N{\widetilde{v}}tensor{{\widetilde{v}}}$}
\ENDFOR
  \STATE{$B^{(i-1,J_{\hat{e}})} \sgets (H^{(i-1,J_{\hat{e}})} +  \gamma^{(i-1)} I)^{-1} + \sum_{b \in \mathrm{desc}_{c_{i-1}}({\widetilde{v}})} \widetilde{B}^{(J_{\{{\widetilde{v}},b\}})}$}
  \ENDFOR 
  \STATE{$\widetilde{V} \sgets \bigcup_{{\widetilde{v}} \in \widetilde{V}} \mathrm{pred}_{c_0}({\widetilde{v}})$}
  \ENDWHILE 
  \item[]
  \item[] \COMMENTx{solve and update representation (cf. \cref{updateformula,brevalA,brevalLN}) \textit{possibly} using iterative solver (cf. \cref{para:cg})}
  \STATE{$  A^{(i-1)} :=
  \deltatensor_{\alpha_c',\alpha_c} \netprod \sum_{v \in \mathrm{neigh}({c_{i-1}})} \big( \netprod_{e \in E_{c_{i-1}} \setminus \{\{{c_{i-1}},v\}\}} \deltatensor_{{\beta^e}',\beta^e}\big) \netprod B^{(J_{{\{{c_{i-1}},v\}}})}$}
    \STATE{$F^{(i-1)} \sgets L_{c_{i-1}}
  \boxtimes^\zeta_{v \in \mathrm{neigh}(c_{i-1})} 
  S^{(J_{\{c_{i-1},v\}})}$}
  \STATE{solve $(F^{(i-1)}_{\mathfrak{m}_{c_{i-1}} \rightarrow \mathfrak{m}_{c_{i-1}}'} \boxtimes F^{(i-1)} + c_{\mathcal{L}} \gamma^{(i-1)} A^{(i-1)}) \boxtimes N^+_{c_{i-1}} = F^{(i-1)}_{\mathfrak{m}_{c_{i-1}} \rightarrow \mathfrak{m}_{c_{i-1}}'} \boxtimes y$}
  \STATE{$N_{c_{i-1}} \sgets N_{c_{i-1}}^+$}
  \item[]
\item[] \COMMENTx{let $J_e \in \mathcal{K}^{\mathcal{S}_{c_{i}}}$, $e \in E$, and shift root:}
    \STATE{set $c_i \in V \setminus \{c_{i-1}\}$ (cf. \cref{repbascal})}
    \item[]
    \item[] \COMMENTx{orthonormalize $\{N_v\}_{v \in V}$ with respect to $c_i$ (cf. \cref{canonN}) and supplement the missing branch evaluations of $\{S^{(J_e)}\}_{e \in E}$ (cf. \cref{brevalLN})}
     \STATE{set $p$ as path from including $c_{i-1}$ to including $c_i$ (cf. \cref{sec:graphtheory})} 
        \FOR{$j = 2,\ldots,|p|$}
        \STATE{$\hat{e} \sgets \{p_{j-1},p_{j}\}$} 
        \STATE{$U \diag(\sigma^{(J_{\hat{e}})}) V^T \sgets N_{p_{j-1}}^{[\mathfrak{m}_{p_{j-1}} \setminus \{\beta^{\hat{e}}\}]}$}
        \STATE{\textit{possibly} adapt rank $r \in \N^\mathcal{K}$ via according modification of the SVD components $U$, $\diag(\sigma^{(J_{\hat{e}})})$ and $V^T$ (cf. \cref{sec:explra})} \STATE{$N_{p_{j-1}}^{[\mathfrak{m}_{p_{j-1}} \setminus \{\beta^{\hat{e}}\}]} \sgets U \diag(\sigma^{(J_{\hat{e}})})$, $N_{p_{j}}^{[\beta^{\hat{e}}]} \sgets V^T N_{p_{j}}^{[\beta^{\hat{e}}]}$} 
  \STATE{$S^{(J_{\hat{e}})} \sgets (L_{p_{j-1}} \boxtimes N_{p_{j-1}}) \boxtimes^\zeta_{h \in \mathrm{desc}_{c_i}({p_{j-1}})} S^{(J_{\{{p_{j-1}},h\}})}$}
    \ENDFOR
    \item[]
      \item[] \COMMENTx{implicitly declare iterate $X^{(i)} = \tau_r(\{N_v\}_{v \in V})$}
    \STATE{set $\gamma^{(i)} \leq \gamma^{(i-1)}$ (\textit{possibly} bound from above, cf. \cref{sec:adaptgammma})}
     \STATE{consider break \textit{possibly} also based on validation residual}
  \ENDFOR
    \STATE{\textit{possibly} repeat as post-iteration with adapted parameters to appropriately truncate representation (cf. \cref{para:postiter})}
    \end{algorithmic}
\end{breakablealgorithm}

\newpage
\section{Practical and Heuristic Aspects}\label{sec:practheurasp}
In \cref{exp2}, the \textsc{AIRLS}-$0\mathcal{K}$ algorithm is enhanced through the use of the following heuristics as embedded in \cref{AIRLS-0K}.

\subsection{Validation set}\label{sec:valiset}
A fraction of measurements is passively used to instead validate the progress allowing for more suitable breaking criteria and to adaptively control the parameter $\gamma$ (cf. \cref{sec:adaptgammma}). This however assumes that the algorithm despite the decreased number of actively used measurements still converges to the essentially same solution.

\subsection{Adaptive decay of regularization parameter \texorpdfstring{$\gamma$}{gamma}}\label{sec:adaptgammma}
Practice shows that, additionally to the constant decline, carefully bounding $\gamma$ from above by a value proportional to the residual norm on the validation measurements (cf. \cref{sec:valiset}) can speed up convergence considerably without infringing upon the approximation.

\subsection{Explicit rank adaption}\label{sec:explra}
The \textsc{AIRLS}-$0\mathcal{K}$ algorithm necessarily relies on the choice of some $\{r^{(J)}\}_{J \in \mathcal{K}}$ which bounds the ranks of the iterate. An adaptive determination can save a considerable amount of computational complexity. Introducing or removing a singular value $\sigma^{(J)}_{r^{(J)}}(X)$ (thus changing the rank of the iterate), that is small compared to $\gamma$, only marginally influences the iteration.
A method that has proven itself reliable in practice is to adapt each single rank $r^{(J)}$, $J \in \mathcal{K}$, such that always
$\sigma^{(J)}_{r^{(J)}-2}(X) > \frac{1}{2} \sqrt{\gamma}$, but $\sigma^{(J)}_{r^{(J)}-1}(X) < \frac{1}{2} \sqrt{\gamma}$.
Thereby, there are always exactly two comparatively low singular values with respect to each subset $J$.

\subsection{\texorpdfstring{\AIRLS-$0\mathcal{K}$}{AIRLS-0K} internal post-iteration}\label{para:postiter}
In particular if the ranks are explicitly adapted, some singular values of the final iterate may be small enough such that
a truncation of such seems more reasonable. Instead of a separate procedure that does not consider the original problem setting,
a better approximation can be achieved by letting the algorithm proceed some additional iterations but with adapted meta parameters
and for a specifically chosen value $\gamma$. Alternatively for small dimensions, the post-iteration scheme as discussed in \cref{sec:postiter} may be utilized.

\subsection{Solving linear subsystems with iterative solvers}\label{para:cg}
The linear subproblems
that appear in each optimization step might become too large to solve explicitely using ordinary, full matrix vector calculus.
Iterative solvers, such as preconditioned CG, can be applied to reduce the order of complexity significantly by exploiting the given low rank as well as additive structures.
Whether this is truly beneficial naturally depends on the exact sizes that are involved, and not least the implementation.